\documentclass[namedate,webpdf,imanum]{ima-authoring-template}

\graphicspath{{Fig/}}

\theoremstyle{thmstyletwo}%
\newtheorem{theorem}{Theorem}
\newtheorem{proposition}[theorem]{Proposition}%
\newtheorem{lemma}[theorem]{Lemma}%

\numberwithin{equation}{section}

\renewcommand{\nu}{\textnormal{n}}

\usepackage{bm}

\newcommand{\Om}{\varOmega}
\newcommand{\Omh}{\varOmega_h}
\newcommand{\Omt}{\varOmega(t)}

\newcommand{\Omhx}{\varOmega_h[\bfx]}
\newcommand{\Ga}{\varGamma}
\newcommand{\Gah}{\varGamma_h}
\newcommand{\Gat}{\varGamma(t)}

\newcommand{\Gahx}{\varGamma_h[\bfx]}

\newcommand{\dnu}{\partial_\nu}
\newcommand{\dtheta}{\,\mathrm{d}\theta}
\newcommand{\mat}{\partial^\bullet}
\newcommand{\nb}{\nabla}
\newcommand{\nbg}{\nabla_\Ga}
\newcommand{\nbgh}{\nabla_{\Ga_h}}
\newcommand{\nbghx}{\nabla_{\Ga_h[\bfx]}}

\newcommand{\lb}{\Delta_\Ga}

\newcommand{{\dGa}}{\kappa}

\newcommand{\tr}{\gamma}
\newcommand{\trh}{\gamma_h}
\newcommand{\T}{^{\mathrm{T}}}

\usepackage{color}

\newcommand\bfd{{\mathbf d}}
\newcommand\bfe{{\mathbf e}}
\newcommand\bff{{\mathbf f}}

\newcommand\bfp{{\mathbf p}}
\newcommand\bfn{{\mathbf n}}

\newcommand\bfu{{\mathbf u}}
\newcommand\bfv{{\mathbf v}}
\newcommand\bfw{{\mathbf w}}
\newcommand\bfx{{\mathbf x}}

\newcommand\bfz{{\mathbf z}}
\newcommand\bfA{{\mathbf A}}

\newcommand\bfD{{\mathbf D}}

\newcommand\bfH{{\mathbf H}}

\newcommand\bfK{{\mathbf K}}
\newcommand\bfL{{\mathbf L}}
\newcommand\bfM{{\mathbf M}}
\newcommand\bfQ{{\mathbf Q}}

\newcommand\bfV{{\mathbf V}}

\newcommand\bfphi{{\boldsymbol \varphi}}

\newcommand\bfzero{{\mathbf 0}}
\newcommand\calI{{\mathcal I}}

\newcommand\calP{{\mathcal P}}

\newcommand\calS{{\mathcal S}}
\newcommand\calT{{\mathcal T}}
\newcommand\calV{{\mathcal V}}

\newcommand{\bftr}{\bm{\gamma}}
\newcommand{\bffu}{\bff_{u}}

\newcommand\andquad{\quad\hbox{ and }\quad}
\newcommand\for{\quad\hbox{ for }\quad}

\renewcommand{\d}{\text{d}}

\newcommand{\ga}{\gamma}

\newcommand{\laplace}{\Delta}

\newcommand{\diff}{\frac{\d}{\d t}}
\newcommand{\eps}{\varepsilon}

\newcommand{\inv}{^{-1}}

\newcommand{\pa}{\partial}
\newcommand{\R}{\mathbb{R}}

\def \t {(t)}

\def \to {\rightarrow}
\newcommand{\vphi}{\varphi}

\newcommand{\Ih}{\widetilde{I}_h}
\newcommand{\wt}{\widetilde}
\newcommand{\aalpha}{\alpha }

\newcommand{\normMs}[1]{\| #1 \|_{\bfM_\Ga(\xs(s))}}

\newcommand{\normKt}[1]{\| #1\|_{\bfK_\Ga(\xs(t))}}
\newcommand{\normKs}[1]{\| #1\|_{\bfK_\Ga(\xs(s))}}
\newcommand{\normKo}[1]{\| #1\|_{\bfK_\Ga(\xs(0))}}

\newcommand{\Hs}{\bfH^\ast}
\newcommand{\ns}{\bfn^\ast}
\newcommand{\us}{\bfu^\ast}

\newcommand{\vs}{\bfv^\ast}
\newcommand{\Vs}{\bfV^\ast}
\newcommand{\ws}{\bfw^\ast}
\newcommand{\xs}{\bfx^\ast}

\newcommand{\xls}{\bfx_\ast}

\newcommand{\eu}{\bfe_\bfu}
\newcommand{\ev}{\bfe_\bfv}

\newcommand{\ex}{\bfe_\bfx}

\newcommand{\en}{\bfe_\bfn}
\newcommand{\eH}{\bfe_\bfH}

\newcommand{\dv}{\bfd_\bfv}

\newcommand{\Rh}{\widetilde{R}_h}

\renewcommand{\a}{\overline{a}}

\newenvironment{rcases}
{\left.\begin{aligned}}
	{\end{aligned}\right\rbrace}

\newcommand{\n}{\nu}

\newcommand{\clb}{\mu}

\newcommand{\dof}{N}

\newcommand{\ecl}{\color{black}}

\begin{document}

\DOI{}
\copyrightyear{2024}
\vol{}
\pubyear{2024}
\access{}
\appnotes{Paper}
\copyrightstatement{}
\firstpage{1}


\title[Numerical analysis of a tumour growth model]{Numerical analysis of an evolving bulk--surface model of tumour growth}

\author{Dominik~Edelmann
\address{\orgdiv{Mathematisches Institut}, \orgname{Universit\"at T\"{u}bingen}, \orgaddress{\street{Auf der Morgenstelle 10.}, \postcode{72076}, \state{Tübingen}, \country{Germany}}}}
\author{Bal\'{a}zs~Kov\'{a}cs*
\address{\orgdiv{Institute of Mathematics}, \orgname{Paderborn University}, \orgaddress{\street{Warburgerstr.~100.}, \postcode{33098}, \state{Paderborn}, \country{Germany}}}}
\author{Christian~Lubich
\address{\orgdiv{Mathematisches Institut}, \orgname{Universit\"at T\"{u}bingen}, \orgaddress{\street{Auf der Morgenstelle 10.}, \postcode{72076}, \state{Tübingen}, \country{Germany}}}}

\authormark{D.~Edelmann, B.~Kov\'{a}cs and Ch.~Lubich}

%


\abstract{This paper studies an evolving bulk--surface finite element method for a model of tissue growth, which is a modification of the model of Eyles, King and Styles (2019). The model couples a Poisson equation on the domain with a forced mean curvature flow of the free boundary, with nontrivial bulk--surface coupling in both the velocity law of the evolving surface and the boundary condition of the Poisson equation. The numerical method discretizes evolution equations for the mean curvature and the outer normal and it uses a harmonic extension of the surface velocity into the bulk. The discretization admits a convergence analysis in the case of continuous finite elements of polynomial degree at least two. The stability of the discretized bulk--surface coupling is a major concern. The error analysis combines stability estimates and consistency estimates to yield optimal-order $H^1$-norm error bounds for the computed tissue pressure and for the surface position, velocity, normal vector and mean curvature. 
Numerical experiments illustrate and complement the theoretical results.}



\maketitle

\section{Introduction}
Eyles, King and Styles \citep{EKS19} proposed and studied a `tractable' model for tumour growth that determines the evolving tumour domain $\varOmega(t)$ at time $t$ with the free boundary surface $\Ga(t)$ together with the tissue pressure $u(x,t)$ on $\varOmega(t)$:
\begin{itemize}
	\item The tissue pressure $u$ solves a Poisson equation on the bulk $\varOmega(t)$ with an inhomogeneous Robin boundary condition, where the inhomogeneity is a sum of the mean curvature on $\Ga(t)$ and a given source term. 
	\item The surface $\Ga(t)$ follows a forced mean curvature flow with the tissue pressure on the boundary as the forcing term. 
\end{itemize}
This model has non-trivial bulk--surface coupling in both the velocity law of the evolving surface and the boundary condition for the tissue pressure. A related model based on the Stokes flow, instead of Darcy's law which yields the Poisson equation, was proposed in \citep{KingVenkataraman_Stokes}. 

While the original model of \citep{EKS19} turned out intractable to us, we show that a modified tumour model is indeed tractable by numerical analysis (alas, not the tumour itself). The modification replaces the Robin boundary condition by a generalized Robin boundary condition that adds the surface Laplacian of $u$ as a regularizing term. 
This allows us to control  the $H^1(\Ga(t))$ norm of the error in $u$ in the numerical discretization, which is essential in the stability analysis.

We propose and analyse an evolving finite element method for the modified bulk--surface problem, which differs from the method proposed in \citep{EKS19} in how the forced mean curvature flow is handled numerically. We use the numerical approach of \citep{MCF,MCF_soldriven}, which discretizes evolution equations for the outer normal and the mean curvature, and which 
allows for an error analysis that yields convergence with optimal-order error bounds.

In Section~\ref{section:tumour growth model} we formulate the model equations and further equations that are to be solved numerically.
In Section~\ref{section:semi-discretization} we describe the evolving finite element semi-discretization and formulate the main result on optimal-order error bounds in the $H^1$ norm (Theorem~\ref{theorem:main}). This theorem is  proved in the course of Section~\ref{section:stability}  on stability (bounding errors in terms of defects) and Section~\ref{section:consistency} on consistency (bounding defects), using also auxiliary results from the appendix (Section~\ref{section:laplace}). In Subsection~\ref{subsec:why-not-EKS} we explain which term in the error equations prevents us from a stability analysis for the original Eyles--King--Styles model.
Numerical experiments in Section~\ref{section:numerics} illustrate the behaviour of the numerical method and the (tiny) effect of the regularization on the solution.

\section{The Eyles--King--Styles model of tumour growth}
\label{section:tumour growth model}

\subsection{Basic notions and notation}
\label{section:basic notations}
For times $t$ in an interval $[0,T]$, we consider a time-dependent closed surface $\Ga(t)$ in $\R^3$ that is the boundary of a bounded domain $\varOmega(t)\subset\R^3$. With $\Om^0=\Om(0)$ and $\Ga^0=\Ga(0)$, we assume that the closure
$\overline{\Omt} = \Omt \cup \Gat \subset \R^3$ is the image of a smooth map $X \colon \Om^0 \cup \Ga^0 \times [0,T] \to \R^3$ with the properties that $X(\cdot,0)$ is the identity map and, for each $t$, 
$X(\cdot,t)$ is an embedding and in particular 
\[ \Omt = \{ X(q,t)\ : \ q \in \Om^0 \} \ , \qquad \Gat = \{ X(q,t) \ : \ q \in \Ga^0 \} \,. \]
In view of the subsequent numerical discretization, it is convenient to think of $X(q,t)\in\overline{\Omt}$ as the position at time $t$ of a particle labelled by its initial position $q \in \Om^0 \cup \Ga^0$, and of $\overline{\Omt}$ as a collection of such particles. Particles that are initially in the interior $\Om^0$ remain in the interior $\Om(t)$ and those on the boundary surface $\Ga^0$ remain on the boundary surface $\Gat$. In the following, we will refer to $\Omt$ as the bulk and to $\Gat$ as the boundary.

To indicate the dependence of the domain and its boundary on~$X$, we will write
\begin{equation*}
	\begin{aligned}
		& \Om(t) = \Om[X(\cdot,t)] \andquad \Ga(t) = \Ga[X(\cdot,t)] , 
	\end{aligned}	
\end{equation*}
or briefly $\Om[X]$ and $\Ga[X]$ when the time $t$ is clear from the context. The {\it velocity} $v(x,t)\in\R^3$ at a point $x=X(q,t) \in \Omt \cup \Gat$ equals
\begin{equation}
	\label{velocity}
	\partial_t X(q,t)= v(X(q,t),t).
\end{equation}
For a known velocity field  $v$, the position $X(q,t)$ at time $t$ of the particle with label $q$ is obtained by solving the ordinary differential equation \eqref{velocity} from $0$ to $t$ for a fixed $q$. We denote the {\it surface velocity} and {\it surface position}
$$
v_\Ga(\cdot,t)=v(\cdot,t)|_{\Gat} \quad\text{ and }\quad X_\Ga(\cdot,t)=X(\cdot,t)|_{\Ga^0}.
$$ 

For a function $u(x,t)$ ($x\in \Omt \cup \Gat$, $0\le t \le T$) we denote the {\it material derivative} (with respect to the parametrization $X$) as
\begin{align*}
	\mat u(x,t) &= \frac \d{\d t} \,u(X(q,t),t) = \nabla u (x,t) \cdot v(x,t) + \partial_t u(x,t) \qquad \hbox{ for } \ x=X(q,t).
\end{align*} 

On any regular surface $\Ga\subset\R^3$, we denote by $\nbg u \colon \Ga \to \R^3$ the  {\it tangential gradient} of a function $u \colon \Ga\to\R$, and in the case of a vector-valued function $u=(u_1,u_2,u_3)^T \colon \Ga\to\R^3$, we let $\nbg u= ( \nbg u_1 , \nbg u_2 , \nbg u_3 )$. We thus use the convention that the gradient of $u$ has the gradient of the components as column vectors. We denote by $\nabla_{\Ga} \cdot f = \text{tr}(\nbg f)$ the {\it surface divergence} of a vector field $f$ on $\Ga$, cf.~\cite[equation~(2.7)]{DziukElliott_acta}, 
and by $\laplace_{\Ga} u=\nabla_{\Ga}\cdot \nabla_{\Ga}u$ the {\it Laplace--Beltrami operator} applied to $u$; see the review \citep{DeckelnickDE2005} or \cite[Appendix~A]{Ecker2012} or any textbook on differential geometry for these notions. 

We denote the unit outer normal vector field to $\Ga$ by $\n \colon \Ga\to\R^3$. Its surface gradient contains the (extrinsic) curvature data of the surface $\Ga$. At every $x\in\Ga$, the matrix of the extended Weingarten map,
$$
A(x)=\nabla_\Ga \n(x),
$$ 
is a symmetric $3\times 3$ matrix (see, e.g., \cite[Proposition~20]{Walker2015}). Apart from the eigenvalue $0$ with eigenvector $\n(x)$, its other two eigenvalues are
the principal curvatures $\kappa_1$ and $\kappa_2$ at the point $x$ on the surface. They determine the fundamental quantities
\begin{equation}
	\label{Def-H-A2}
	H:={\rm tr}(A)=\kappa_1+\kappa_2, \qquad 
	|A|^2 = \kappa_1^2 +\kappa_2^2 ,
\end{equation}
where $H$ is the {\it mean curvature} (as in most of the literature, defined here without a factor 1/2) and
$|A|$ is the Frobenius norm of the matrix $A$.

\subsection{Model equations and more equations}
\subsubsection{A `tractable' model for tumour growth}
In \citep{EKS19}, Eyles, King and Styles studied the  
following model for tumour growth, which determines the tissue pressure $u(x,t)$ on the evolving tumour domain $\Omt$ with the free boundary $\Gat$ by the following Robin boundary value problem ($\alpha$ here corresponds to $1/\alpha$ in \citep{EKS19}):
\begin{subequations}
	\label{robin}
	\begin{alignat}{3}
		\label{robin-pde}
		- \Delta u &= - 1 & \qquad & \text{in }\Omt \, , \\
		\label{robin-bc}
		\dnu{u} + \aalpha u   &= \beta H + Q & \qquad & \text{on }\Gat \, ,
	\end{alignat}
\end{subequations}
where $\partial_\nu$ is the outer normal derivative, $\alpha >0$ and $\beta>0$  are given constants, $H(x,t)$ is the mean curvature of the surface $\Gat$, and $Q(x,t)$ is a given surface source term. We assume throughout that $Q \colon \R^3\times \R \to \R$ is a smooth function.
The surface velocity is determined by the velocity law of a forced mean curvature flow:
\begin{align}
	v_\Ga &= V\nu , \quad\text{ with the normal velocity}\quad V= - \beta H + \aalpha u  \quad\text{on }\Gat \,.
	\label{v-Gamma}
\end{align}
Note that the above equations are strongly coupled through the appearance of $H$ as a source term in \eqref{robin-bc} and of $u$ as the forcing term in \eqref{v-Gamma}.

\subsubsection{Forced mean curvature flow}
The velocity law \eqref{v-Gamma} describes a forced mean curvature flow. In \cite[Lemma~2.1]{MCF_soldriven} it was shown that the mean curvature and surface normal from \eqref{v-Gamma} satisfy parabolic evolution equations along the flow, which read (note that here $\alpha u$ corresponds to $u$ of \citep{MCF_soldriven}):
\begin{subequations}
	\label{eq:forced mcf}
	\begin{alignat}{3}
		\label{eq:forced mcf - nu}
		\mat \nu &= \beta \lb \nu + \beta |A|^2 \nu - \aalpha \nbg u & \qquad & \text{on }\Gat \, , \\ 
		\label{eq:forced mcf-H}
		\mat H &= \beta \lb H + \beta |A|^2 H -  \aalpha \lb u - \aalpha\,|A|^2  u & \qquad & \text{on }\Gat \,  .
	\end{alignat}
\end{subequations}

In our numerical method, we will discretize the weak form of these equations by evolving finite elements and use the so obtained approximations to $H$ and $\nu$ in \eqref{v-Gamma} to compute an approximate surface velocity. The second-order term $ \lb u$ in \eqref{eq:forced mcf-H} requires special attention in the stability analysis. 

%
%
%
%

\subsubsection{Harmonic extension of the surface velocity into the bulk}
For the finite element method, a moving mesh is required not only on the boundary $\Ga(t)$ but also on $\Om(t)$; cf.~\cite[Section~6.1.2]{EKS19}. This is achieved by extending the velocity $v_\Ga$, which is given by the forced mean curvature flow, harmonically into $\varOmega(t)$, i.e., we have the equation
\begin{align}
	\begin{aligned} -\laplace v &= 0 \quad&&\text{in }\Omt\,,\\
		v &= v_\Ga \quad&&\text{on }\Gat\,.
	\end{aligned}
\end{align}
With this velocity, the positions $X(q,t)$ are obtained by solving the ordinary differential equation \eqref{velocity}. The equations given so far fully describe the coupled bulk--surface model of \citep{EKS19}.

\subsubsection{Regularization}
\label{section:regularization}
In view of the term $\Delta_\Ga u$ in the evolution equation \eqref{eq:forced mcf - nu} for the forced mean curvature flow, it is crucial that the trace of $u$ on $\Ga$, denoted $\tr u$,  be in $H^1(\Ga)$ with controlled norm in the exact and the numerical solution, and in the error. We cannot ensure this for the original Eyles--King--Styles model. We therefore add a regularization term in the boundary condition \eqref{robin-bc},
\begin{align}
	\label{robin-bc-gen}
	\dnu u  - \clb \lb u +\aalpha u = Q+\beta H \qquad \text{on } \Gat , \tag{2.3b'}
\end{align}
with a positive constant $\clb > 0$. Such a  boundary condition is sometimes referred to as a \emph{generalized Robin boundary condition}. It is shown in \cite[Section~3.2]{KashiwabaraCDQ} that the solution $u$ has improved regularity on the boundary. We will show that here we can control the $H^1(\Ga)$ norm of the numerical error in $u$ whenever the bulk inhomogeneity in the error equation is suitably bounded in the dual space of $H^1(\Om)$ but not necessarily in $L^2(\Om)$.

\subsubsection{Collected equations used for the discretization}
Altogether, we consider the following strong formulation of the modified model problem, which consists of four coupled groups of equations:
\newcommand{\hpllap}[1]{\hphantom{- \clb \lb u + \frac{u}{\alpha}}\llap{$#1$}}
\newcommand{\hprlap}[1]{\rlap{$#1$}\hphantom{\beta \lb H + \beta |A|^2 H - \aalpha \lb u - \frac{|A|^2}{\alpha} u}}
\begin{align}
	\label{robin-bvp} 
	&\begin{rcases}
		\hpllap{- \Delta u} &= \hprlap{- 1} \quad &&\text{in }\Omt ,\\
		\hpllap{\dnu{u}} & \hprlap{- \clb \lb u +\aalpha u  =\beta H + Q}  &&\text{on }\Gat ; \quad
	\end{rcases} 
	\\[3mm]
	\label{forced-mcf}
	&\begin{rcases}
		\hpllap{\mat \nu} &= \hprlap{\beta \lb \nu + \beta |A|^2 \nu - \aalpha \nbg u} \quad\!\!\! &&\text{on }\Gat ,\\[1mm]
		\hpllap{\mat H} &= \beta \lb H + \beta |A|^2 H - \aalpha \lb u - \aalpha |A|^2 u \quad\!\!\! &&\text{on }\Gat ,
		\\[1mm]
		\hpllap{v_\Ga} &=  V\nu \quad\text{with }V=- \beta H + \aalpha u &&\text{on }\Gat ; \quad
	\end{rcases}
	\\[3mm]
	\label{harmonic-v}
	&\begin{rcases}
		\hpllap{-\laplace v} &= \hprlap{0} \quad &&\text{in }\Omt ,\\
		\hpllap{v} &= \hprlap{v_\Ga} \quad &&\text{on }\Gat ; \quad\\
	\end{rcases}
	\\[3mm]
	\label{ODE}
	&\begin{aligned}
		\hpllap{\pa_t X} &= \hprlap{v\circ X } \quad\!\!\! &&\text{on }\Om^0 \cup \Ga^0 .
	\end{aligned}
\end{align}

In the following weak formulation and matrix--vector formulation of the discretization and in the error analysis, the corresponding equations are clearly separated into these four parts. We refer to \eqref{robin-bvp} as the generalized Robin boundary value problem, to \eqref{forced-mcf} as the forced mean curvature flow, to \eqref{harmonic-v} as the harmonic velocity extension, and to \eqref{ODE} as the ODE for the positions.

\subsection{Weak formulation}
\label{section:weak formulation}

\subsubsection{Generalized Robin boundary value problem}
In this subsection, we briefly write $\Om$ and $\Ga$ and omit the argument $X$ or $t$.
The weak formulation of the Robin boundary value problem \eqref{robin} is to find $u\in H^1(\Om)$ such that (with $\gamma$ denoting the trace operator):
\begin{equation}\label{robin-weak}
	\begin{aligned}
		& \int_{\Om} \nb u \cdot \nb \varphi^u   + \clb \int_{\Ga} \nbg (\tr u) \cdot \nbg (\tr \varphi^u) 
		+ \aalpha \int_{\Ga} (\tr u) (\tr \varphi^u)  
		= -\int_{\Om} \varphi^u + \int_{\Ga} (\beta H+Q) (\tr \varphi^u) 
	\end{aligned}
\end{equation}
for all $\varphi^u \in H^1(\Om)$.

%
\ecl

\subsubsection{Forced mean curvature flow}
The weak formulation of the parabolic equations for $\nu$ and $H$ in \eqref{forced-mcf} reads:
\begin{equation}
	\label{forced-mcf-weak}
	\begin{aligned}
		\int_{\Ga} \mat \nu \,\varphi^\nu + \beta \int_{\Ga} \nbg \nu \cdot \nbg \varphi^\nu = &\ \beta \int_{\Ga} |A|^2 \nu \cdot \varphi^\nu - \aalpha \int_{\Ga} \nbg (\tr u) \cdot \varphi^\nu \,, \\
		\int_{\Ga} \mat H \varphi^H + \beta \int_{\Ga} \nbg H \cdot \nbg \varphi^H = &\ - \int_{\Ga} |A|^2 V \varphi^H + \aalpha \int_{\Ga} \nbg (\tr u) \cdot \nbg \varphi^H \, ,
	\end{aligned}
\end{equation}
for all test functions $\vphi^\nu \in H^1(\Ga[X])^3$ and $\vphi^H \in H^1(\Ga[X])$. 
Recall that the system is coupled to the velocity law \eqref{v-Gamma} and the ordinary differential equation \eqref{velocity} that determines the boundary surface $\Ga[X(\cdot,t)] = \pa\Om[X(\cdot,t)]$.

\subsubsection{Harmonic velocity extension}

After extending $v_\Ga\in H^{1/2}(\Ga)$ to an (arbitrary) function $w \in H^1(\Om)$ with $\gamma w=v_\Ga$,
the weak formulation of the harmonic extension \eqref{harmonic-v} (which is an inhomogeneous Dirichlet problem) 
reads:  Find $v \in H^1(\Om)^3$ with $ v-w \in H^1_0(\Om)^3$ that satisfies
\begin{equation}
	\label{harmonic-v-weak}
	\int_{\Omt} \nb  v \cdot \nb \varphi^v = 0 ,
\end{equation}
for all $\varphi^v \in H_0^1(\Om)^3$. 
Recall that the equation for $v$ is coupled to the ordinary differential equation \eqref{velocity} that determines the domain $\Om[X(\cdot,t)]$.

\section{Spatial semi-discretization with bulk--surface finite elements}
\label{section:semi-discretization}


We formulate the evolving bulk--surface finite element discretization for the coupled tumour growth system, following the descriptions in \citep{ElliottRanner_bulksurface} and \citep{MCF}, as well as details from \citep{Dziuk88} and \citep{Demlow2009}. We use tetrahedral and triangular finite elements in the bulk and on the surface, respectively, and continuous piecewise polynomial basis functions of degree~$k$, which are compatible on the boundary. For details on higher-order finite elements we refer to \cite[Section~5]{ElliottRanner_bulksurface}, \cite[Section~2.5]{Demlow2009}, \citep{Edelmann_harmonicvelo,Edelmann_isoparametric},  \citep{highorderESFEM}, and \citep{ElliottRanner_unified}.

\subsection{Evolving bulk--surface finite elements}
\label{section:evolving bulk--surface FEM}

The initial bulk--surface domain $\Om^0 \cup \Ga^0$ is discretized by a tetrahedral mesh of degree $k$, denoted by $\Om_h^0$. The meshes are assumed to form an admissible family of triangulations $\mathcal{T}_h$ of decreasing maximal element diameter $h$; see \citep{DziukElliott_ESFEM,ElliottRanner_bulksurface} for the notion of an admissible bulk--surface approximation, which includes quasi-uniformity and shape regularity. 
By piecewise polynomial interpolation of degree $k$, the nodal vector defines an approximate bulk--surface domain $\Om_h^0$ that interpolates $\Om^0$ in the nodes $q_j$.
By construction, the boundary of the tetrahedral mesh $\Om_h^0$ forms an equally admissible triangular approximation $\Ga^0_h= \partial \Om_h^0$ of the initial boundary $\Ga^0$.

The nodes $q_j$ of the triangulation are collected in the vector
\begin{equation*}
	\bfx^0 = \begin{pmatrix} \bfx_\Ga^0 \\ \bfx_\Om^0 \end{pmatrix} \in \bigl(\R^3\bigr)^\dof ,
\end{equation*}
where we assume that the boundary nodes $\bfx_\Ga^0 = (q_j)_{j=1}^{\dof_\Ga}$ lie on the boundary $\Ga^0$, while the remaining ($N_\Om = N - N_\Ga$) nodes $\bfx_\Om^0$ are in the interior of $\Om^0$.

The nodes $\bfx^0$ will evolve in time and at time $t$ they are collected in the vector 
\begin{equation*}
	\bfx\t = \begin{pmatrix} \bfx_\Ga\t \\ \bfx_\Om\t \end{pmatrix} \in \bigl(\R^3\bigr)^\dof ,
\end{equation*}
which employs the same partitioning as before. We will often suppress the omnipresent argument $t$.

By piecewise polynomial interpolation on the plane reference tetrahedron that
corresponds to every curved tetrahedron of the triangulation, the nodal vector $\bfx$
defines a domain $\Om_h[\bfx]$ and its boundary surface  $\Ga_h[\bfx]$, which actually depends only on $\bfx_\Ga$.

We then define globally continuous finite element {\it basis functions} 
\begin{alignat*}{3}
	&\ \varphi_i[\bfx] \colon \Om_h[\bfx]\to\R, & \qquad & i = 1, \dotsc, \dof , \\
	&\ \psi_i[\bfx] \colon \Ga_h[\bfx]\to\R, & \qquad & i = 1, \dotsc, \dof_\Ga , 
\end{alignat*}
which have the property that on every tetrahedron and triangle, respectively, their pullback to the reference element are polynomials of degree $k$, and which satisfy at the node $x_j$
\begin{align*}
	&\ \varphi_i[\bfx](x_j) = \delta_{ij} \quad  \text{ for all } i,j = 1,  \dotsc, \dof , \\
	&\ \psi_i[\bfx](x_j) = \delta_{ij} \quad  \text{ for all } i = 1,  \dotsc, \dof_\Ga , \  j = 1,  \dotsc, \dof.
\end{align*}
Since we have $\Ga_h\t = \partial \Om_h\t$, by construction the basis functions satisfy the crucial property that $\psi_j[\bfx]$ is the discrete trace of $\varphi_j[\bfx]$, i.e.
\begin{equation}
	\label{eq:trace of basis functions}
	\psi_j[\bfx] = \trh \varphi_j[\bfx].
\end{equation}

These functions span the bulk and surface finite element spaces in $\Om_h[\bfx]$ and on $\Ga_h[\bfx]$, respectively:
\begin{align}
	\calV_h[\bfx] &= \mathrm{span} \{ \varphi_1[\bfx],\ldots,\varphi_\dof[\bfx] \} \,,\\
	\calS_h[\bfx] &= \mathrm{span} \{ \psi_1[\bfx],\ldots,\psi_{\dof_\Ga}[\bfx] \} \,.
\end{align}
For a finite element function $u_h\in \calS_h[\bfx]$, the tangential gradient $\nabla_{\Ga_h[\bfx]}u_h$ is defined piecewise on each element. Note that $\calV_h[\bfx] \subset H^1(\Omh[\bfx];\Gah[\bfx])$ and $\calS_h[\bfx] \subset H^1(\Gah[\bfx])$. 
We denote the space of discrete functions with vanishing trace by
\begin{equation}
	\label{eq:FEM space with zero trace}
	\calV_h^0[\bfx] = \{ v \in \calV_h[\bfx]\,:\, \gamma_h v=0\} =
	\mathrm{span} \{ \varphi_{N_\Ga+1}[\bfx],\ldots,\varphi_\dof[\bfx] \} .
\end{equation}


The discrete domain at time $t$ is parametrized by the initial discrete domain via the map 
$X_h(\cdot,t) \colon \overline\Om_h^0\to\overline \Om_h[\bfx(t)]$ defined by
\begin{equation}
	\label{Xh}
	X_h(q_h,t) = \sum_{j=1}^\dof x_j(t) \, \vphi_j[\bfx(0)](q_h), \qquad q_h \in \Om_h^0,
\end{equation}
which has the properties that $X_h(q_j,t)=x_j(t)$ for $j=1,\dots,\dof$ and  $X_h(q_h,0) = q_h$ for all $q_h\in\Om_h^0$. We then have
$$
\Om_h[\bfx(t)]=\Om[X_h(\cdot,t)] \andquad \Ga_h[\bfx(t)]=\Ga[X_h(\cdot,t)],
$$
where the right-hand sides equal the domain $\{ X_h(q_h,t) : q_h \in \Om_h^0 \}$ and surface $\{ X_h(q_h,t) : q_h \in \Ga_h^0 \}$, respectively, as in Section~\ref{section:basic notations}.

The {\it discrete velocity} $v_h(x,t)\in\R^3$ at a point $x=X_h(q_h,t) \in \Om[X_h(\cdot,t)]$ is given by
$$
\partial_t X_h(q_h,t) = v_h(X_h(q_h,t),t).
$$
In view of the transport property of the basis functions \citep{DziukElliott_ESFEM},
\begin{equation}
	\begin{aligned}
		\frac\d{\d t} \Bigl( \vphi_j[\bfx(t)](X_h(q_h,t)) \Bigr) =0  \for q_h \in \Om_h^0 ,\\
		\frac\d{\d t} \Bigl( \psi_j[\bfx(t)](X_h(q_h,t)) \Bigr) =0  \for q_h \in \Ga_h^0 , \\
	\end{aligned}	
\end{equation}
the discrete velocity equals, for $x \in \Om_h[\bfx(t)]$,
$$
v_h(x,t) = \sum_{j=1}^\dof v_j(t) \, \vphi_j[\bfx(t)](x) \qquad \hbox{with } \ v_j(t)=\dot x_j(t),
$$
where the dot denotes the time derivative $\d/\d t$. 
Hence, the discrete velocity $v_h(\cdot,t)$ is in the bulk finite element space $\calV_h[\bfx(t)]$, with nodal vector $\bfv(t)=\dot\bfx(t)$, while its discrete trace $v_{\Ga_h}(\cdot,t) = \ga_h v_h(\cdot,t)$ is in the surface finite element space $\calS_h[\bfx\t]$, with nodal values $\bfv_\Ga = (v_j\t)_{j=1}^{\dof_\Ga} = (\dot x_j\t)_{j=1}^{\dof_\Ga} = \dot\bfx_\Ga\t$.

The {\it discrete material derivative} of a finite element function $u_h(\cdot,t) \in \calV_h[\bfx\t]$, with nodal values $(u_j(t))_{j=1}^N$, at $x=X_h(q_h,t) \in \Om_h[\bfx\t]$ is
\begin{equation}
	\mat_h u_h(x,t) = \frac{\d}{\d t} u_h(X_h(q_h,t)) = \sum_{j=1}^\dof \dot u_j(t)  \vphi_j[\bfx(t)](x) ,
\end{equation}
and similarly, for $w_h(\cdot,t) \in \calS_h[\bfx\t]$, at $x=X_h(q_h,t) \in \Ga_h[\bfx\t]$ is
\begin{equation}
	\mat_h w_h(x,t) = \frac{\d}{\d t} w_h(X_h(q_h,t)) = \sum_{j=1}^{\dof_\Ga} \dot w_j(t)  \psi_j[\bfx(t)](x) .
\end{equation}
Since $\psi_j = \ga_h \vphi_j$, for $j=1,\dotsc,\dof_\Ga$, this directly implies that we have $\mat_h ( \ga_h u_h ) = \ga_h \big( \mat_h u_h \big)$.

\subsubsection{Lifts}

We now introduce a lift operator for bulk--surface functions, following \citep{Dziuk88,ElliottRanner_bulksurface,ElliottRanner_unified}.

Following \citep{Dziuk88}, we define the \emph{lift} of functions $w_h \colon \Ga_h\to \R$ as 
\begin{equation}
	\label{eq:lift definition}
	w_h^\ell \colon \Ga \to \R \qquad \text{with} \qquad w_h^\ell(p) = w_h(x), \quad \forall p \in \Ga ,
\end{equation}
where $x\in\Ga_h$ is the \emph{unique} point on $\Ga_h$ with $x-p$  orthogonal to the tangent space $T_p\Ga$.
We further consider the \emph{lift} of functions $w_h \colon \Om_h\to \R$ to $w_h^\ell \colon \Om\to\R$ by setting $w_h^\ell(p)=w_h(x)$  if $x\in\varOmega_h$ and $p\in \varOmega$ are related as described in detail in \cite[Section~4]{ElliottRanner_bulksurface}. 
The mapping $G_h \colon \Om_h \to \Om$ is defined piecewise, for an element $E \in \calT_h$, by
\begin{equation}
	\label{eq:bulk mapping}
	G_h|_E (x) = F_e\big((F_E)^{-1}(x)\big), \qquad \text{for } x \in E,
\end{equation}
where $F_e$ is a $C^1$ map (see \cite[equation~(4.2) \& (4.4)]{ElliottRanner_bulksurface}) from the reference element onto the smooth element $e \subset \Om$, and $F_E$ is the standard affine linear map between the reference element and $E$, see, e.g.~\cite[equation~(4.1)]{ElliottRanner_bulksurface}. 
The \emph{inverse lift} $w^{-\ell} \colon \Ga_h \to \R$ denotes a function whose lift is $w\colon \Ga \to \R$, and similarly for the bulk as well. 
Note that both definitions of the lift coincide on $\Ga$. Finally, the lifted finite element space is denoted by $\calS_h^\ell$, and is given as $\calS_h^\ell = \{  w_h^\ell \mid w_h \in \calS_h \}$.

Then the \emph{composed lift operator} $\,^L$, as used in \citep{MCF}, lifts finite element functions $w_h = \sum_{j = 1}^N w_j \phi_j[\bfx]$ on the discrete surface $\Ga_h[\bfx]$ to functions on the exact surface $\Ga[X]$ via the  finite element function $\widehat w_h = \sum_{j = 1}^N w_j \phi_j[\xs]$ on the interpolated surface $\Ga_h[\xs]$, by setting
\begin{equation}
	\label{Lift}
	w_h^L = (\widehat w_h)^\ell.
\end{equation}
We will compare the positions of the exact surface $\Ga[X(\cdot,t)]$ and the discrete surface $\Ga_h[\bfx(t)]$ by comparing any point $x\in \Ga[X(\cdot,t)]$ with its associated point $x_h^L(x,t) \in \Ga_h[\bfx(t)]$, which is defined as
\begin{equation}
	\label{xhL}
	\begin{aligned}
		x_h^L(x,t) =  X_h^L(q,t) \in \Ga_h[\bfx(t)] \qquad\  &\text{ for }  q \in \Ga^0 
		\text{ such that } \quad
		X(q,t)=x\in\Ga[X(\cdot,t)] .
	\end{aligned}
\end{equation}
Here, the discrete flow map  $X_h(\cdot,t) \colon \Ga_h^0\to \R^3$ is defined in \eqref{Xh} and we denote its composed lift by $X_h^L(\cdot,t)=\bigl(X_h(\cdot,t)\bigr)^L \colon \Ga^0 \to \R^3$.

\subsection{Semi-discretization of the coupled bulk--surface system}
\label{subsec:semi-disc-bs}

\subsubsection{Semi-discretization of the generalized Robin boundary value problem}
The semi-discretization of the Robin boundary value problem \eqref{robin} computes $u_h\in \calV_h[\bfx]$ such that (with $\gamma_h$ denoting the discrete trace operator)
\begin{equation}\label{robin-h}
	\begin{aligned}
		& \int_{\Om_h[\bfx]} \nb u_h \cdot \nb \varphi_h^u 
		+ \clb \int_{\Gahx} \nbgh (\trh u_h) \cdot \nbgh (\trh \varphi_h^u) 
		+ \aalpha \int_{\Gahx} (\trh u_h) (\trh \varphi_h^u) 
		\\
		&= - \int_{\Omhx} \varphi_h^u + \int_{\Gahx}(\beta H_h+Q_h)(\trh \varphi_h^u)
	\end{aligned}
\end{equation}
for all $\varphi_h^u \in \calV_h[\bfx]$. Here $Q_h$ is the finite element interpolation of $Q$ on $\Gahx$.

\subsubsection{Semi-discretization of the forced mean curvature flow}
A finite element semi-discretization of the parabolic equations \eqref{forced-mcf-weak} for $\nu$ and $H$ reads as follows: Find $\nu_h(\cdot,t) \in \calS_h[\bfx\t]^3$ and $H_h(\cdot,t) \in \calS_h[\bfx\t]$ that
satisfy, with $u_h$ from \eqref{robin-h} and with $A_h = (\nbgh \nu_h + (\nbgh \nu_h)^T ) / 2$ and $V_h = - \beta H_h + \aalpha \ga_h u_h $, 
\begin{equation}
	\label{mcf-h}
	\begin{aligned}
		\int_{\Ga_h[\bfx]} \!\!\! \mat_h \nu_h \varphi_h^\nu + \beta \int_{\Ga_h[\bfx]} \!\!\!\! \nbgh \nu_h \cdot \nbgh \varphi_h^\nu = &\ \beta \int_{\Ga_h[\bfx]} \!\!\! |A_h|^2 \nu_h \cdot \varphi_h^\nu 
		- \aalpha \int_{\Ga_h[\bfx]} \!\!\!\! \nbgh u_h \cdot \varphi_h^\nu , \\
		\int_{\Ga_h[\bfx]} \!\!\! \mat_h H_h \varphi_h^H + \beta \int_{\Ga_h[\bfx]} \!\!\!\! \nbgh H_h \cdot \nbgh \varphi_h^H = &\ - \int_{\Ga_h[\bfx]} \!\!\! |A_h|^2 V_h \varphi_h^H 
		+ \aalpha \int_{\Ga_h[\bfx]} \!\!\!\! \nbgh u_h \cdot \nbgh \varphi_h^H ,
	\end{aligned}
\end{equation}
for all test functions $\vphi_h^\nu \in \calS_h[\bfx]^3$ and $\vphi_h^H \in \calS_h[\bfx]$. 

The discrete surface velocity $v_{\Ga_h}(\cdot,t) \in \calS_h[\bfx\t]^3$ with nodal vector $\bfv_\Ga$ is determined by 
\begin{equation}\label{v-Gamma-h}
	v_{\Ga_h} = \Ih \big( V_h \nu_h \big) , \qquad\text{where } \quad  V_h = - \beta H_h + \aalpha \ga_h u_h  \in \calS_h[\bfx]
\end{equation}
and $\Ih \colon C(\Ga_h[\bfx]) \to \calS_h[\bfx]$ denotes the finite element interpolation operator on $\Ga_h[\bfx]$.
The nodal vector $\bfx_\Ga\t$ that determines the discrete surface $\Ga_h[\bfx\t]$ is obtained by integrating
\begin{equation}
	\dot \bfx_\Ga = \bfv_\Ga, \qquad\  \bfx_\Ga(0)=\bfx_\Ga^0.
\end{equation}

\subsubsection{Discrete harmonic velocity extension and ODE for positions}
We compute $v_h \in \calV_h[\bfx]^3$ with $\gamma_h v_h = v_{\Ga_h}$ such that 
\begin{align}
	\label{harmonic-v-h}
	\int_{\Omhx} \nb v_h \cdot \nb \varphi_h^v = 0
\end{align}
for all $\varphi_h^v \in \calV_{h}^0[\bfx]^3$. 
From the nodal vector $\bfv=(\bfv_\Ga;\bfv_\Om)$ of the finite element function $v_h$, the nodal vector $\bfx\t$ that determines the discrete domain $\Om_h[\bfx\t]$ is obtained by integrating
\begin{equation}
	\dot \bfx_\Om = \bfv_\Om, \qquad\  \bfx_\Om(0)=\bfx_\Om^0.
\end{equation}


\subsection{Matrix--vector formulation}
\label{section:matrix-vector formulation}

\subsubsection{Mass and stiffness matrices}
We collect the nodal values of the semi-discrete approximations to tissue pressure $u_h(\cdot,t)$, velocity $v_h(\cdot,t)$, normal vector $\nu_h(\cdot,t)$, mean curvature $H_h(\cdot,t)$ in the column vectors 
$$
\bfu = (u_j) \in \R^{\dof},\qquad \bfv = (v_j) \in \R^{3\dof},\qquad \bfn = (\nu_j) \in \R^{3\dof_\Ga},\qquad \bfH = (H_j) \in \R^{\dof_\Ga}, 
$$ 
respectively. As we did with the position vector $\bfx$, we partition
$$
\bfu=\begin{pmatrix}\bfu_\Ga \\ \bfu_\varOmega \end{pmatrix}, \quad   \bfv=\begin{pmatrix}\bfv_\Ga \\ \bfv_\varOmega \end{pmatrix},
$$
corresponding to the nodes on the surface and in the bulk.
Furthermore, the nodal values of the semi-discrete normal velocity $V_h = - \beta H_h + \alpha \ga_h u_h  \in \calS_h[\bfx]$ are collected in the vector $\bfV = (V_j) \in \R^{\dof_\Ga}$:
$$
\bfV = -\beta \bfH + \aalpha \bfu_\Ga  .
$$
We denote the domain-dependent bulk stiffness and mass matrices by $\bfA_{\bar\Om}(\bfx) \in \R^{\dof \times \dof}$ and $\bfM_{\bar\Om}(\bfx) \in \R^{\dof \times \dof}$, respectively, and define them as
\begin{equation}
	\begin{aligned}
		\bfM_{\bar\Om}(\bfx)|_{ij} &= \int_{\Omhx}  \! \vphi_i[\bfx] \vphi_j[\bfx]  ,\\
		\bfA_{\bar\Om}(\bfx)|_{ij} &= \int_{\Omhx}  \!\! \nb \vphi_i[\bfx] \cdot \nb \vphi_j[\bfx]  ,
	\end{aligned}
	\quad\  i, j = 1,\ldots,\dof .
\end{equation}
We partition these matrices as
\begin{equation} \label{A-partition}
	\bfA_{\bar\Om}(\bfx) = \begin{pmatrix} \bfA_{\Ga\Ga}(\bfx) & \bfA_{\Ga\Om}(\bfx)
		\\
		\bfA_{\Om\Ga}(\bfx) & \bfA_{\Om\Om}(\bfx)
	\end{pmatrix}
\end{equation}
and similarly for $\bfM_{\bar\Om}(\bfx)$.

We denote the surface mass and stiffness matrices, (recall $\psi_j = \ga_h \vphi_j$, for $j=1,\dotsc,\dof_\Ga$), by
\begin{equation}
	\begin{aligned}
		\bfM_\Ga(\bfx)|_{ij} &= \int_{\Gahx}  \!\! \psi_i[\bfx] \, \psi_j[\bfx]  ,\\
		\bfA_\Ga(\bfx)|_{ij} &= \int_{\Gahx}  \!\!\!\! \nbghx \psi_i[\bfx] \cdot \nbghx \psi_j[\bfx]  ,
	\end{aligned}
	\quad i, j = 1, \ldots, \dof_\Ga.
\end{equation}

The discrete trace matrix, corresponding to the discrete trace operator $\ga_h$, is given by
\[ 
\bftr = \begin{pmatrix} I_{\dof_\Ga} & 0 \end{pmatrix} \in \R^{\dof_\Ga \times \dof}, 
\]
where $I_{\dof_\Ga}$ denotes the $\dof_\Ga \times \dof_\Ga$ identity matrix. 
For any nodal vector $\bfu \in \R^\dof$ corresponding to a finite element function $u_h \in \calV_h[\bfx]$, the vector $\bftr \bfu =\bfu_\Ga \in \R^{\dof_\Ga}$ is the nodal vector of the finite element function $\trh u_h \in \calS_h[\bfx]$, cf.~\eqref{eq:trace of basis functions}.

Moreover, for a matrix and for an arbitrary dimension $d$, we use the notation
\begin{align}
	\bfM_{\Ga}^{[d]}(\bfx) = I_d \otimes \bfM_{\Ga}(\bfx) , \qquad 
	\bfA_{\Ga}^{[d]}(\bfx) = I_d \otimes \bfA_{\Ga}(\bfx) ,
\end{align}
where $I_d \in \R^{d \times d}$ denotes the identity matrix, and $\otimes$ denotes the Kronecker product of matrices. If the dimension $d$ is clear from the context, we will drop the superscript $^{[d]}$.

Finally, we define the tangential gradient matrix $\bfD(\bfx) \in \R^{3N_\Ga \times N_\Ga}$
\[
\bfD(\bfx)|_{i+(\ell-1)N_\Ga,j} = \int_{\Gahx} \psi_i[\bfx] ( \nbghx \psi_j[x] )_\ell 
\]
for $i,j = 1,\ldots,N_\Ga$, $\ell= 0, 1, 2$.
Alternatively, this could be written as
\[ \bfD(\bfx) = \begin{pmatrix} \bfD_1(\bfx) \\ \bfD_2(\bfx) \\ \bfD_3(\bfx) \end{pmatrix}  \quad\text{ with }\quad
\bfD_\ell(\bfx)|_{ij} = \int_{\Gah} \psi_i \underline{D}_{h,\ell} \psi_j ,
\]
where $\underline{D}_{h,\ell} \psi_j$ denotes the $\ell$-th component of the discrete tangential gradient of $\psi_j$.

\subsubsection{Matrix--vector formulation of the generalized Robin boundary value problem}
We define the vector $\bffu(\bfx,\bfH) \in \R^N$  by
\begin{equation}
	\label{rhs-robin}
	\bffu(\bfx,\bfH)|_j =  - \int_{\Omhx} \!\!\!\! \varphi_j[\bfx]+ \int_{\Gahx} \!\!\!\! (\beta H_h+Q_h)\, \trh \varphi_j[\bfx]  \ , \quad j=1,\ldots,\dof \ ,
\end{equation}
so that
$$
\bffu(\bfx,\bfH) = - \bfM_{\bar\Om}(\bfx) \mathbf{1} + \bftr\T\bfM_{\Ga}(\bfx) (\beta \bfH+\bfQ(\bfx)),
$$
where $ \mathbf{1} \in \R^N$ is the vector with all entries equal to 1 and $\bfQ(\bfx)$ is the nodal vector of $Q$ on $\Gahx$, i.e. 
$\bfQ(\bfx)=\bigl( Q(x_i)\bigr)$.

The discretized Robin boundary value problem \eqref{robin-h} determines the nodal vector $\bfu\in\R^{N}$ as the solution of the linear system of equations
\begin{equation}\label{robin-mv}
	\Big( \bfA_{\bar\Om}(\bfx)  + \clb \bftr\T \bfA_\Ga(\bfx) \bftr + \aalpha \bftr\T \bfM_\Ga(\bfx) \bftr
	\Big) \bfu = \bffu(\bfx,\bfH).
\end{equation}

\subsubsection{Matrix--vector formulation of forced mean curvature flow}
The discrete velocity law \eqref{v-Gamma-h} written in terms of the nodal vectors becomes simply
\begin{equation} \label{v-Gamma-mv}
	\bfv_\Ga = \bfV \bullet \bfn \qquad \text{with} \qquad \bfV = - \beta\bfH + \aalpha \bftr \bfu ,
\end{equation}
where $\bullet$ denotes the componentwise product of vectors, i.e.~$(\bfV \bullet \bfn)_j = V_j \nu_j $ for
$\bfV=(V_j)\in\R^{N_\Ga}$ and $\bfn=(\nu_j)\in (\R^3)^{N_\Ga}$.
We define the functions 
$\bff_\nu(\bfx,\bfn) \in \R^{3 \dof_\Ga}$ and $\bff_H(\bfx,\bfn,\bfV) \in \R^{\dof_\Ga}$ by
\begin{equation*}
	\begin{aligned}
		\bff_\nu(\bfx,\bfn)|_{j+(\ell-1)\dof_\Ga} = &\ \beta \int_{\Gahx} \!\!\!\! |A_h|^2 \, (\nu_h)_\ell \, \psi_j[\bfx] 
		,\\
		\bff_H(\bfx,\bfn,\bfV)|_j = &\ - \int_{\Gahx} \!\!\!\! |A_h|^2 \, V_h \, \psi_j[\bfx] 
		,
	\end{aligned}	
\end{equation*}
for $j = 1,\dotsc,\dof_\Ga$ and $\ell = 1,2,3$.

The formulation of the semi-discretized forced mean curvature flow equations \eqref{mcf-h} then leads to the matrix--vector formulation
\begin{subequations} \label{mcf-mv}
	\begin{align} 
		\bfM_\Ga^{[3]}(\bfx) \dot{\bfn} + \beta \bfA_\Ga^{[3]} (\bfx) \bfn &=  \bff_\nu(\bfx,\bfn)  -\aalpha  \bfD_\Ga(\bfx) \bftr \bfu ,
		\\
		\bfM_\Ga(\bfx) \dot{\bfH} + \beta \bfA_\Ga(\bfx) \bfH &=  \bff_H(\bfx,\bfn,-\beta \bfH+ \alpha\gamma \bfu) +  \aalpha  \bfA_\Ga(\bfx) \bftr \bfu.
	\end{align}
\end{subequations}
Without the coupling terms containing $\bfu$ or $\dot \bfu$, these are the discrete equations studied in \citep{MCF} for the discretization of pure mean curvature flow. 


\subsubsection{Matrix--vector formulation of the harmonic velocity extension}

In view of the partitioning \eqref{A-partition},
the matrix--vector formulation for the discrete harmonic velocity extension \eqref{harmonic-v-h}  reads
\begin{equation} \label{harmonic-v-mv}
	\bfA_{\Om\Om}^{[3]}(\bfx) \bfv_\Om = - \bfA_{\Om\Ga}^{[3]}(\bfx) \bfv_\Ga \, .
\end{equation}
In summary, the matrix--vector formulation of the bulk--surface finite element semi-discretization of the coupled problem \eqref{robin-bvp}--\eqref{harmonic-v} is given by equations \eqref{robin-mv}--\eqref{harmonic-v-mv}.

\subsubsection{Matrix--vector formulation of the coupled bulk--surface problem}
In summary,  with $\bfx = (\bfx_\Ga; \bfx_\Om)$, $\bfv = (\bfv_\Ga; \bfv_\Om)$, and
\begin{equation}\label{L-matrix}
	\bfL(\bfx) := \bfA_{\bar\Om}(\bfx)  + \clb \bftr\T \bfA_\Ga(\bfx)\bftr + \aalpha \bftr\T \bfM_\Ga(\bfx) \bftr,
\end{equation}
the full matrix--vector formulation of the bulk--surface finite element semi-discretization of the coupled problem \eqref{robin-bvp}--\eqref{harmonic-v} thus reads as follows:

\begin{subequations}
	\label{eq:matrix-vector formulation - full system}
	\begin{align}
		\label{eq:matrix-vector formulation - full system - a}
		\bfL(\bfx) \bfu &= \ \bffu(\bfx,\bfH) 
		\\[2mm]
		\label{eq:matrix-vector formulation - full system - b}
		\bfM_\Ga(\bfx) \dot{\bfn} + \beta \bfA_\Ga (\bfx) \bfn &=  \bff_\nu(\bfx,\bfn)  -\aalpha  \bfD_\Ga(\bfx) \bfu_\Ga 
		\\
		\label{eq:matrix-vector formulation - full system - c}
		\bfM_\Ga(\bfx) \dot{\bfH} + \beta \bfA_\Ga(\bfx) \bfH &=  \bff_H(\bfx,\bfn,\bfV) +  \aalpha  \bfA_\Ga(\bfx) \bfu_\Ga
		\\[2mm]
		\label{eq:matrix-vector formulation - full system - d}
		\bfV &=  - \beta\bfH + \aalpha \bfu_\Ga 
		\\
		\label{eq:matrix-vector formulation - full system - 4}
		\bfv_\Ga &=  \bfV \bullet \bfn    
		\\[1mm]
		\label{eq:matrix-vector formulation - full system - f}
		\bfA_{\Om\Om}(\bfx) \bfv_\Om &=  - \bfA_{\Om\Ga}(\bfx) \bfv_\Ga 
		\\[1mm]
		\label{eq:matrix-vector formulation - full system - g}
		\dot{\bfx} &=  \bfv .
	\end{align}
\end{subequations}

\subsection{Error bounds}
\label{section:main result}



Our main result yields optimal-order error bounds for the finite element semi-discretization using finite elements of polynomial degree $k \geq 2$ under the assumption that the exact solution is sufficiently regular.
The following result will be proved in the course of this paper.

\begin{theorem}
	\label{theorem:main} 
	Consider the space discretization of Section~\ref{subsec:semi-disc-bs} of the coupled bulk--surface problem \eqref{robin-bvp}--\eqref{ODE}, using evolving bulk--surface finite elements of polynomial degree $k\ge 2$. 
	Suppose that the problem admits an exact solution $(X,v,\n,H,u)$ that is sufficiently differentiable on the time interval $t\in[0,T]$, and that 
	for each $t$, the flow map $X(\cdot,t)$ is an embedding.
	
	Then, there exists a constant $h_0 > 0$ such that for all mesh sizes $h \leq h_0$ the following error bounds for the lifts \eqref{Lift} of the approximate tissue pressure and the discrete boundary position hold over the exact domain $\Om(t)=\Om[X(\cdot,t)]$ and exact surface $\Ga(t)=\Ga[X(\cdot,t)]$ for $0 \leq t \leq T$:
	\begin{align*}
		\|u_h^L(\cdot,t) - u(\cdot,t)\|_{H^1(\Om(t),\Ga(t))} \leq &\ C h^k, \\
		\|x_h^L(\cdot,t) - \mathrm{id}_{\Ga(t)}\|_{H^1(\Ga(t))^3} \leq &\ Ch^k,
	\end{align*}
	where $x_h^L$ is defined in \eqref{xhL}.
	
	Furthermore, $\|X_h^\ell(\cdot,t) - X(\cdot,t)\|_{H^1(\Ga_0)^3} \leq Ch^k$,
	and there are analogous error bounds for velocity, normal vector and mean curvature:
	\begin{align*}
		\|v_h^L(\cdot,t) - v(\cdot,t)\|_{H^1(\Ga(t))^3} \leq &\ C h^k, \\
		\|\n_h^L(\cdot,t) - \n(\cdot,t)\|_{H^1(\Ga(t))^3} \leq &\ C h^k, \\
		\|H_h^L(\cdot,t) - H(\cdot,t)\|_{H^1(\Ga(t))} \leq &\ C h^k. 
	\end{align*}
	The constant $C$ is independent of $h$ and $t$, but depends on bounds of higher derivatives of the solution $(u,X,v,\n,H)$  and on the length $T$ of the time interval.
\end{theorem}

Theorem~\ref{theorem:main} will be proved in the course of Sections~\ref{section:stability} and \ref{section:consistency}, using also results from the appendix (Section~\ref{section:laplace}).


\section{Stability of the spatially discretized bulk--surface problem}
\label{section:stability}

In this section we prove stability of the semi-discretization of the coupled bulk--surface problem in the sense that errors are bounded by defects in the semi-discrete equations in appropriate norms. The precise result is stated in Proposition~\ref{prop:stability} below.
Its proof requires error bounds for each of the three sub-problems (generalized Robin boundary value problem, forced mean curvature flow, harmonic velocity extension) and their non-trivial coupling.

For ease of presentation we choose in the following the parameters
$$
\alpha = \beta = \mu= 1.
$$ 
The general case just leads to different constants, as we will not study asymptotic limits as $\alpha$ or $\beta$ or $\mu$ goes to zero or infinity.

\subsection{Preparation: Estimates relating different mass and stiffness matrices}
\label{subsec:aux} 
In the following stability proof, we use technical results relating different finite element domains, which were proved in \citep{KLLP2017} for the surface mass and stiffness matrices and in \citep{Edelmann_harmonicvelo} for the bulk mass and stiffness matrices. We use the following setting.

Let $\bfx \in \R^{3N}$ be a nodal vector defining the discrete domain $\Omhx$ with boundary $\Gahx$. We denote by $\bfx_\Ga \in \R^{3 N_\Ga}$ the nodes of $\bfx$ that lie on the boundary $\Gahx$. For any nodal vector $\bfu = (u_j) \in \R^N$ and $\bfw = (w_j) \in \R^{N_\Ga}$ with corresponding finite element functions $u_h = \sum_{j=1}^N u_j \varphi_j[\bfx] \in \calV_h[\bfx]$ and $w_h = \sum_{j=1}^{N_\Ga} w_j \varphi_j[\bfx] \in \calS_h[\bfx]$, respectively, the mass and stiffness matrices 
define norms on $\Omhx$: 
\begin{equation}
	\begin{aligned}
		\| \bfu \|_{\bfM_{\bar\Om}(\bfx)}^2 &= \bfu\T \bfM_{\bar\Om}(\bfx) \bfu = \|u_h\|_{L^2(\Omhx)}^2 \,, \\
		\| \bfu \|_{\bfA_{\bar\Om}(\bfx)}^2 &= \bfu\T \bfA_{\bar\Om}(\bfx) \bfu = \| \nb u_h\|_{L^2(\Omhx)}^2 \,,\\
	\end{aligned}
\end{equation}
and in the same way for the surface mass and stiffness matrices.

Let $\bfx,\xs \in \R^{3N}$ be two nodal vectors that define discrete domains $\Omh[\bfx]$ and $\Omh[\xs]$ with boundaries $\Gahx$ and $\Gah[\xs]$. We denote the difference by $\bfe = \bfx-\xs$. For $\theta \in [0,1]$, we consider the intermediate domain $\Omh^\theta = \Omh[\xs+\theta \bfe]$ and the corresponding finite element function given by
\[ e_h^\theta = \sum_{j=1}^N e_j \phi_j[\xs+\theta \bfe] \,. \]
In the same way, any vectors $\bfu \in \R^N$ or $\bfw \in \R^{N_\Ga}$ define finite element functions $u_h^\theta \in \calV_h[\xs+\theta \bfe]$ and $w_h^\theta \in \calS_h[\xs+\theta \bfe]$, respectively.

\begin{lemma}\label{lemma:matrix differences}
	In the above setting, the following identities hold for any bulk nodal vectors $\bfw, \bfz \in \R^N$ or surface nodal vectors $\bfw,\bfz \in \R^{N_\Ga}$:
	\begin{equation}
		\begin{aligned}
			\bfw\T \Big( \bfM_{\bar\Om}(\bfx) - \bfM_{\bar\Om}(\xs) \Big) \bfz &= \int_0^1 \int_{\Omh^\theta} w_h^\theta ( \nb \cdot e_h^\theta) z_h^\theta \  \dtheta \\
			\bfw\T \Big( \bfA_{\bar\Om}(\bfx) - \bfA_{\bar\Om}(\xs) \Big) \bfz &= \int_0^1 \int_{\Omh^\theta} \nb w_h^\theta \cdot \Big( D_{\Omh^\theta} e_h^\theta \Big) \nb z_h^\theta \  \dtheta \\
			\bfw\T \Big( \bfM_\Ga(\bfx) - \bfM_\Ga(\xs) \Big) \bfz &= \int_0^1 \int_{\Gah^\theta} w_h^\theta ( \nb_{\Gah^\theta} \cdot e_h^\theta) z_h^\theta \  \dtheta \\
			\bfw\T \Big( \bfA_\Ga(\bfx) - \bfA_\Ga(\xs) \Big) \bfz &= \int_0^1 \int_{\Gah^\theta} \nb_{\Gah^\theta} w_h^\theta \cdot \Big( D_{\Gah^\theta} e_h^\theta \Big) \nb_{\Gah^\theta} z_h^\theta \  \dtheta, \\
		\end{aligned}
	\end{equation}
	where we have $D_{\Omh^\theta} e_h^\theta = \textnormal{tr}(E^\theta) I_3 - (E^\theta+(E^\theta)^T)$ with $E^\theta=\nabla e_h^\theta \in \R^{3\times 3}$ and 
	$D_{\Ga_h^\theta} e_h^\theta =  \textnormal{tr}(G^\theta) I_3 - (G^\theta+(G^\theta)^T)$ with $G^\theta=\nabla_{\Ga_h^\theta} e_h^\theta \in \R^{3\times 3}$.	
	
	Furthermore, for any $\bfw \in \R^{3 N_\Ga}$, $z \in \R^{N_\Ga}$, the following identity holds:
	\begin{equation}
		\bfw\T \Big( \bfD(\bfx) - \bfD(\xs) \Big) \bfz = \int_0^1 \int_{\Gah^\theta} \Big( w_h^\theta \cdot \nb_{\Gah^\theta} z_h^\theta \Big) \nb_{\Gah^\theta} \cdot e_h^\theta  \dtheta \,. 
	\end{equation}
\end{lemma}
\begin{proof}
	The formulae result from the Leibniz rule (transport formula). A proof of the first pair of identities is found in \cite[Lemma 5.1]{Edelmann_harmonicvelo}, for the second pair in \cite[Lemma 4.1]{KLLP2017}. The last identity follows from the third one.
\end{proof}

The following lemma combines Lemmas 4.2 and 4.3 of \citep{KLLP2017} and Lemmas 5.2 and 5.3 of \citep{Edelmann_harmonicvelo}.
\begin{lemma}\label{lemma:theta-independence}
	In the above setting, suppose that
	$$
	\| \nb e_h^0 \|_{L^\infty(\Omh[\xs])} \le \tfrac{1}{2} \,.
	$$
	Then, for $0 \le \theta \le 1$, the finite element functions $u_h^\theta = \sum_{j=1}^N u_j \phi_j[\xs+\theta\bfe]$ on $\Omh^\theta$ and $w_h^\theta = \sum_{j=1}^{N_\Ga} w_j \phi_j[\xs+\theta \bfe]$ on $\Gah^\theta$ are bounded by
	\begin{equation}
		\begin{alignedat}{3}
			\| u_h^\theta \|_{L^p(\Omh^\theta)} &\le c_p \| u_h^0 \|_{L^p(\Omh^\theta)},  &\ \qquad \| \nb u_h^\theta \|_{L^p(\Omh^\theta)} \le &\ c_p \| \nb u_h^0 \|_{L^p(\Omh^0)} , \\
			\| w_h^\theta \|_{L^p(\Gah^\theta)} &\le c_p \| w_h^0 \|_{L^p(\Gah^0)},  &\ \qquad  \| \nb_{\Gah^\theta} w_h^\theta \|_{L^p(\Gah^\theta)} \le &\ c_p \| \nb_{\Gah^0} w_h^0 \|_{L^p(\Gah^0)} ,
		\end{alignedat}
	\end{equation}
	for $1\le p \le\infty$, where $c_p<\infty$ is independent of $\theta \in [0,1]$ and of $h$. 
\end{lemma}

If $\| \nabla_{\Ga_h[\xs]} e_h^0 \|_{L^\infty(\Ga_h[\xs])}\le \frac14$, using the lemma for $w_h^\theta=e_h^\theta$ shows that
\begin{equation}
	\label{e-theta}
	\| \nabla_{\Ga_h^\theta} e_h^\theta \|_{L^\infty(\Ga_h^\theta)} \le \tfrac12, \qquad 0\le\theta\le 1,
\end{equation}
and then Lemma~\ref{lemma:theta-independence} with $p=2$ and interchanging the roles of $\xs$ and $\xs+\theta\bfe$ yield the following:
\begin{equation}
	\label{norm-equiv}
	\begin{aligned}
		&\text{The norms $\|\cdot\|_{\bfM_{\bar\Om}(\xs+\theta\bfe)}$ are $h$-uniformly equivalent for $0\le\theta\le 1$,}
		\\
		&\text{and so are the norms $\|\cdot\|_{\bfA_{\bar\Om}(\xs+\theta\bfe)}$, 
			$\|\cdot\|_{\bfM_{\Ga}(\xs+\theta\bfe)}$, $\|\cdot\|_{\bfA_{\Ga}(\xs+\theta\bfe)}$.}
	\end{aligned}
\end{equation}

Under the condition that 
$\eps := \| \nabla_{\Ga_h[\xs]} e_h^0 \|_{L^\infty(\Ga_h[\xs])}\le \tfrac14$, using \eqref{e-theta} in Lemma~\ref{lemma:matrix differences}  and applying the Cauchy--Schwarz inequality yields the bounds, with $c=c_\infty c_2^2$,
\begin{equation}
	\label{matrix difference bounds}
	\begin{aligned}
		\bfw^T (\bfM_{\bar\Om}(\bfx)-\bfM_{\bar\Om}(\xs)) \bfz \leq &\ c \eps \, \|\bfw\|_{\bfM_{\bar\Om}(\xs)} \|\bfz\|_{\bfM_{\bar\Om}(\xs)} , \\[1mm]
		\bfw^T (\bfA_{\bar\Om}(\bfx)-\bfA_{\bar\Om}(\xs)) \bfz \leq &\ c \eps \, \|\bfw\|_{\bfA_{\bar\Om}(\xs)} \|\bfz\|_{\bfA_{\bar\Om}(\xs)} 
	\end{aligned}
\end{equation}
and the analogous formulae for $\bfM_\Ga$ and $\bfA_\Ga$.
We will also use similar bounds where we use the $L^\infty$ norm of $z_h$ or its gradient and the $L^2$ norm of the gradient of $e_h$.

Consider now a continuously differentiable function $\bfx \colon [0,T]\to\R^{3N}$ that defines a finite element domain $\Om_h[\bfx(t)]$ for every $t\in[0,T]$, and assume that its time derivative $\bfv(t)=\dot\bfx(t)$ is the nodal vector of a finite element function $v_h(\cdot,t)$ that satisfies
\begin{equation}
	\label{vh-bound}
	\| \nabla v_h(\cdot,t) \|_{L^{\infty}(\Ga_h[\bfx(t)])} \le K, \qquad 0\le t \le T.
\end{equation}
With $\bfe=\bfx(t)-\bfx(s)=\int_s^t \bfv(r)\,dr$, 
the bounds \eqref{matrix difference bounds}
then yield the following bounds, which were first shown in \cite[Lemma~4.1]{DziukLubichMansour_rksurf} for surfaces:
for $0\le s, t \le T$ with $K|t-s| \le \tfrac14$, 
we have with $C=c K$
\begin{equation}
	\label{matrix difference bounds-t}
	\begin{aligned}
		\bfw^T \bigl(\bfM_{\bar\Om}(\bfx(t))  - \bfM_{\bar\Om}(\bfx(s))\bigr)\bfz \leq&\ C \, |t-s| \, \|\bfw\|_{\bfM_{\bar\Om}(\bfx(t))}\|\bfz\|_{\bfM_{\bar\Om}(\bfx(t))} , \\[1mm]
		\bfw^T \bigl(\bfA_{\bar\Om}(\bfx(t))  - \bfA_{\bar\Om}(\bfx(s))\bigr)\bfz \leq&\ C\,  |t-s| \, \|\bfw\|_{\bfA_{\bar\Om}(\bfx(t))}\|\bfz\|_{\bfA_{\bar\Om}(\bfx(t))}.   
	\end{aligned}
\end{equation}
Letting $s\to t$, this implies the bounds stated in \cite[Lemma~4.6]{KLLP2017} for surfaces 
and  \cite[Lemma 5.4]{Edelmann_harmonicvelo} for domains:
\begin{equation}
	\label{matrix derivatives}
	\begin{aligned}
		\bfw^T \frac\d{\d t}\bfM_{\bar\Om}(\bfx(t))  \bfz \leq&\ C  \,\|\bfw\|_{\bfM_{\bar\Om}(\bfx(t))}\|\bfz\|_{\bfM_{\bar\Om}(\bfx(t))}  \\[1mm]
		\bfw^T \frac\d{\d t}\bfA_{\bar\Om}(\bfx(t))  \bfz \leq&\ C \, \|\bfw\|_{\bfA_{\bar\Om}(\bfx(t))}\|\bfz\|_{\bfA_{\bar\Om}(\bfx(t))} 
	\end{aligned}
\end{equation}
and analogously for surfaces.
Patching together finitely many intervals over which $K|t-s| \le \tfrac14$, we obtain the following:
\begin{equation}
	\label{norm-equiv-t}
	\begin{aligned}
		&\text{The norms $\|\cdot\|_{\bfM_{\bar\Om}(\bfx(t))}$ are $h$-uniformly equivalent for $0\le t \le T$,}
		\\
		&\text{and so are the norms $\|\cdot\|_{\bfA_{\bar\Om}(\bfx(t))}$,}
	\end{aligned}
\end{equation}
and analogously for surfaces.

\subsection{Defects and errors for the coupled system}
\label{section:defects and errors}

We consider finite element functions $x_h^*$, $u_h^*$, $v_h^* = (v_{\Ga,h}^*,v_{\Om,h}^*)$, and $w_h^* = (\nu_h^*,H_h^*)$, with corresponding nodal vectors $\us$, $\xs$, $\vs = (\vs_\Ga;\vs_\Om)$, and $\ws = (\ns;\Hs)$, which are \emph{suitable} finite element projections of the corresponding exact solution $X$, $u$, $v = (v_\Ga,v_\Om)$, and $w = (\nu,H)$, respectively. The precise definition of these interpolations and Ritz-type projections will be given in Section~\ref{section:consistency - projections}, where we will also prove that the obtained defects (i.e.~consistency errors) are $O(h^k)$ in the appropriate norms.
%
%

For each variable the corresponding error vector will be denoted by a subscript, e.g.~for $X$ the error between the semi-discrete solution $\bfx$ and the finite element interpolation of the exact solution $\xs$ is defined by
\begin{equation*}
	\ex = \bfx - \xs ,
\end{equation*}
which collects the nodal values of the finite element function $e_x$. The errors $\ev,\en,\eH,\eu$ -- and the corresponding finite element functions -- are defined analogously.

\subsubsection{Error equation for the generalized Robin boundary value problem}
The nodal vectors  $\bfu^*, \bfx^*, \bfv^*,\bfn^*,\bfH^*$ of the projected exact solution are inserted into the first equation of the numerical scheme \eqref{eq:matrix-vector formulation - full system - a} and  yield the defect $\bfd_\bfu$ defined by
\begin{equation}
	\label{robin-defect}
	\bfL(\bfx^*) \bfu^* = \bffu(\xs,\bfH^*) + \bfM_{\bar\Om}(\bfx^*) \bfd_\bfu \,.
\end{equation}
Subtracting \eqref{robin-defect} from the first equation of \eqref{eq:matrix-vector formulation - full system}, we obtain that the errors $\eu$ satisfy the error equation
\begin{align}
	\label{err-u}
	\bfL(\bfx) \bfe_\bfu =&\  \bigl( \bffu(\bfx,\bfH) - \bffu(\bfx^*,\bfH^*) \bigr)    \nonumber
	\\
	&\ - \bigl(\bfL(\bfx)-\bfL(\bfx^*) \bigr)\bfu^* \\
	&\ - \bfM_{\bar\Om}(\bfx^*) \bfd_\bfu \,. \nonumber
\end{align}


\subsubsection{Error equations for the forced mean curvature flow}

When the nodal vectors of the projected exact solutions $\bfu^*, \bfx^*, \bfv^*,\bfH^*,\bfn^*$ are inserted into the numerical scheme \eqref{eq:matrix-vector formulation - full system - b}--\eqref{eq:matrix-vector formulation - full system - d}, they yield defects $\bfd_\bfn, \bfd_\bfH$ and $\bfd_{\bfv_\Ga}$ defined by
\begin{equation}
	\label{eq:defects - forced MCF}
	\begin{aligned}
		\bfM_\Ga(\bfx^*) \dot{\bfn}^* + \bfA_\Ga(\bfx^*) \bfn^* &= \bff_\nu(\bfx^*,\bfn^*) - \bfD_\Ga(\bfx^*) \bftr\bfu^* + \bfM_\Ga(\bfx^*) \bfd_\bfn \,,
		\\
		\bfM_\Ga(\bfx^*) \dot{\bfH}^* + \bfA_\Ga(\bfx^*) \bfH^* &=  \bff_H(\bfx^*,\bfn^*,-\bfH^*+\bfu^*) + \bfA_\Ga(\bfx^*) \bftr\bfu^* +\bfM_\Ga(\bfx^*) \bfd_\bfH \,,
		\\
		\Vs &=  - \Hs + \us
		\\
		\bfv_\Ga^{*} &= \Vs \bullet \ns + \bfd_{\bfv_\Ga} .
	\end{aligned}	
\end{equation}

As before, we subtract \eqref{eq:defects - forced MCF} from \eqref{mcf-mv} and \eqref{v-Gamma-mv}, and obtain that the errors $\bfe_\bfn$ and $\bfe_\bfH$ satisfy the error equations
\begin{align}
	\label{err-n}
	\bfM_\Ga(\bfx) \dot{\bfe}_\bfn + \bfA_\Ga(\bfx) \bfe_\bfn = &\ -\big( \bfM_\Ga(\bfx)-\bfM_\Ga(\bfx^*) \big) \dot{\bfn}^* 
	\nonumber \\
	&\ 
	- \big( \bfA_\Ga(\bfx) - \bfA_\Ga(\bfx^*) \big) \bfn^* \nonumber \\
	&\ + \big( \bff_\nu(\bfx,\bfn) - \bff_\nu(\bfx^*,\bfn^*) \big) \nonumber
	\\
	&\ - \bfD_\Ga(\bfx^*) \bftr{\bfe}_\bfu - \bigl(  \bfD_\Ga(\bfx) - \bfD_\Ga(\bfx^*) \big) \bftr\bfu 
	\\ 
	&\ - \bfM_\Ga(\bfx^*) \bfd_\bfn , \nonumber
	\\
	\bfM_\Ga(\bfx) \dot{\bfe}_\bfH + \bfA_\Ga(\bfx) \bfe_\bfH = &\ -\big( \bfM_\Ga(\bfx)-\bfM_\Ga(\bfx^*) \big) \dot{\bfH}^* \nonumber \\
	&\ - \big( \bfA_\Ga(\bfx) - \bfA_\Ga(\bfx^*) \big) \bfH^* \nonumber \\
	&\ + \big( \bff_H(\bfx,\bfn,-\bfH+\bftr\bfu) - \bff_H(\bfx^*,\bfn^*,-\bfH^*+\bftr\bfu^*) \big)  \nonumber \\
	&\ + \bfA_\Ga(\bfx^*) \bftr {\bfe}_\bfu + \bigl(  \bfA_\Ga(\bfx) - \bfA_\Ga(\bfx^*) \big) \bftr\bfu \label{err-H} \\
	&\ - \bfM_\Ga(\bfx^*) \bfd_\bfH \,, \nonumber 
\end{align}
and the error $\bfe_{\bfv_\Ga}$ is obtained from
\begin{equation}
	\label{err-v-Gamma}
	\bfe_{\bfv_\Ga} = \bfV \bullet \bfn - \Vs \bullet \ns - \bfd_{\bfv_\Ga} .
\end{equation}
Apart from the coupling terms with $u$ included here, stable error propagation has been shown in \citep{MCF}. In the following we will therefore concentrate on the critical coupling terms involving $\bfe_\bfu$ and  
$\bfu$.

\subsubsection{Error equations for the harmonic velocity extension}
The interpolated exact solution $\vs = (\vs_\Ga;\vs_\Om)$ satisfies the numerical scheme \eqref{eq:matrix-vector formulation - full system - f} up to a defect $\bfd_{\bfv_\Om}$ defined by
\begin{align}
	\label{eq:defects - harmonic extension}
	\bfA_{\Om\Om}(\bfx^*) \bfv_\Om^* &= - \bfA_{\Om\Ga}(\bfx^*) \bfv_\Ga^* + \bfM_{\Om\Om}(\bfx^*) \bfd_{\bfv_\Om} \,.
\end{align}
Subtracting this equation from \eqref{harmonic-v-mv}, we obtain the error equation

%
$$
\begin{aligned}
	\bfA_{\Om\Om}(\bfx^*) \bfe_{\bfv_\Om} + \bfA_{\Om\Ga}(\bfx^*)\bfe_{\bfv_\Ga} 
	&= -\left(\bfA_{\Om\Om}(\bfx)-\bfA_{\Om\Om}(\bfx^*) \right) \bfe_{\bfv_\Om} 
	-\left(\bfA_{\Om\Om}(\bfx)-\bfA_{\Om\Om}(\bfx^*) \right) \bfv_\Om^* \\
	&\quad\, -\left(\bfA_{\Om\Ga}(\bfx)-\bfA_{\Om\Ga}(\bfx^*) \right) \bfe_{\bfv_\Ga} 
	-\left(\bfA_{\Om\Ga}(\bfx)-\bfA_{\Om\Ga}(\bfx^*) \right) \bfv_\Ga^* \\
	&\quad\, - \bfM_{\Om\Om}(\bfx^*) \bfd_{\bfv_\Om} \,.
\end{aligned}
$$
With $\bfe_\bfv=(\bfe_{\bfv_\Ga};\bfe_{\bfv_\Om})$ and $\bfd_{\bfv}=(\bfzero;\bfd_{\bfv_\Om})$, this equation simplifies to
\begin{equation}
	\label{err-v-Omega}
	\begin{aligned}
		\begin{pmatrix} 0 & 0 \\ 0 & \mathbf{I}_\Om \end{pmatrix}
		\bfA_{\bar\Om}(\xs) \bfe_\bfv = -
		\begin{pmatrix} 0 & 0 \\ 0 & \mathbf{I}_\Om \end{pmatrix}
		&\Bigl(  \left(\bfA_{\bar\Om}(\bfx)-\bfA_{\bar\Om}(\bfx^*) \right) \bfe_{\bfv} \Bigr.\\
		&\Bigl. + \left(\bfA_{\bar\Om}(\bfx)-\bfA_{\bar\Om}(\bfx^*) \right) \bfv^*\\
		&+ \bfM_{\bar\Om}(\bfx^*) \bfd_{\bfv} \Bigr).
	\end{aligned}
\end{equation}

\subsection{Stability bound}

We bound the errors at time $t$ in terms of the defects up to time $t$ and the errors in the initial values.
The errors are bounded in the $H^1$ norms on the surface $\Ga$ and, where relevant, in the bulk $\varOmega$.
For a nodal vector $\bfe$ associated with a surface finite element function $e \in \calS_h[\bfx]$, 
we define with the symmetric positive definite matrix $\bfK_{\Ga}(\bfx)=\bfA_{\Ga}(\bfx)+\bfM_{\Ga}(\bfx)$ 
$$
\| \bfe \|_{\bfK_{\Ga}(\bfx)}^2 = \bfe\T \bfK_{\Ga}(\bfx) \bfe = \| e \|_{H^1(\Gahx)}^2 \,,
$$
and similarly for $\bfe$ associated with a bulk finite element function $e \in \calV_h[\bfx]$, 
we define with the matrix $\bfK_{\bar\Om}(\bfx)=\bfA_{\bar\Om}(\bfx)+\bfM_{\bar\Om}(\bfx)$ 
$$
\| \bfe \|_{\bfK_{\bar\Om}(\bfx)}^2 = \bfe\T \bfK_{\bar\Om}(\bfx) \bfe = \| e \|_{H^1(\Omhx)}^2 .
$$
With the matrix $\bfL(\bfx)=\bfA_{\bar\Om}(\bfx) + \bftr\T\bfA_\Ga(\bfx)\bftr+\bftr\T\bfM_\Ga(\bfx)\bftr$ of \eqref{L-matrix} we set
$$
\| \bfe \|_{\bfL(\bfx)}^2 = \bfe\T \bfL(\bfx) \bfe \sim \| e \|_{H^1(\Omhx)}^2 + \| \gamma u_h \|_{H^1(\Gahx)}^2 = \| e \|_{H^1(\Omhx,\Gahx)}^2,
$$
where the constant in the equivalence of norms is independent of $h$.
To bound the defect $\bfd_\bfv$, we further need  the following norm for a bulk nodal vector $\bfd\in\R^N$, which equals
the dual norm for the corresponding finite element function $d\in \calV_h[\bfx]$:
$$
\|\bfd\|_{\star,\bfx} = \|d\|_{H_h\inv({\Om}_h[\bfx])} := 
\sup_{0\ne\varphi_h\in \calV_h^0[\bfx]} \frac{ \int_{{\Om}_h[\bfx]} d  \cdot\varphi_h } { \| \varphi_h \|_{H^1({\Om}_h[\bfx])} } \, .
$$
The following result provides the key stability estimate, which bounds the errors in terms of the defects and the initial errors.

\begin{proposition}
	\label{prop:stability} Assume that the reference finite element functions $x_h^*(\cdot,t)$, $v_h^*(\cdot,t)$, $u_h^*(\cdot,t),H_h^*(\cdot,t),\nu_h^*(\cdot,t)$ on the interpolated surface $\Ga_h[\xs(t)]$ and bulk $\Om_h[\xs(t)]$ have $W^{1,\infty}$ norms that are bounded independently of $h$, uniformly for all $t\in[0,T]$. 
	Let
	\begin{equation}
		\label{eq:stability - defect bounds}
		\begin{aligned}
			\delta = \max_{0\le t \le T} \Bigl( &\| \bfd_\bfu(t) \|_{\bfM_{\bar\Om}(\xs(t))} + \| \dot \bfd_\bfu(t) \|_{\bfM_{\bar\Om}(\xs(t))} 
			+ \| \bfd_\bfH(t) \|_{\bfM_{\Ga}(\xs(t))} 
			\\ &
			+ \| \bfd_\bfn(t) \|_{\bfM_{\Ga}(\xs(t))} 
			+ \| \bfd_{\bfv_\Ga}(t) \|_{\bfK_{\Ga}(\xs(t))} +
			\| \bfd_\bfv(t) \|_{\star,\xs(t)} \Bigr)
		\end{aligned}
	\end{equation}
	be a bound of the defects and 
	\begin{equation}
		\label{eq:stability - initial errors}
		\eps^0 =   \| \bfe_\bfH(0) \|_{\bfK_\Ga(\bfx^0)} + \| \bfe_\bfn(0) \|_{\bfK_\Ga(\bfx^0)}
	\end{equation}
	be a bound of the errors of the initial values.
	If the defects and initial errors are bounded by 
	\begin{equation}
		\label{eq:assumed defect bounds}
		\delta + \eps^0 \le ch^\kappa \quad\text{ for some $\kappa$ with $\frac{3}{2} < \kappa \leq k$,}
	\end{equation}
	then there exists  $\bar h>0$ such that the following stability bound holds for all $h\leq \bar h$ and $0\le t \le T$:
	\begin{equation}
		\label{eq:stability bound}
		\begin{aligned}
			& 	  \hspace{-10pt}
			\|\eu(t)\|_{\bfL(\xs(t))} +  \|\ex(t)\|_{\bfL(\xs(t))} +  \|\ev(t)\|_{\bfL(\xs(t))}
			\\ &+
			\normKt{\bfe_\bfH(t)} + \normKt{\bfe_\bfn(t)} 
			\leq 
			C ( \delta + \eps^0 ),
		\end{aligned}
	\end{equation}
	where $C$ is independent of $h$ and $t$, but depends on the final time $T$. 
\end{proposition}

\subsection{Proof of Proposition~\ref{prop:stability}}

\subsubsection{Preliminaries}

Let $t^*\le T$ be the maximal time such that the following inequalities hold true:
\begin{equation}
	\label{ass-err-linf}
	\begin{aligned}
		\|e_x(\cdot,t)\|_{W^{1,\infty}(\Om_h[\xs(t)])} \leq &\ h^{(\kappa-3/2)/2}  \\
		\|e_v(\cdot,t)\|_{W^{1,\infty}(\Om_h[\xs(t)])} \leq &\ h^{(\kappa-3/2)/2}  \\
		\|e_u(\cdot,t)\|_{W^{1,\infty}(\Om_h[\xs(t)])} \leq &\ h^{(\kappa-3/2)/2} \\
		\|e_H(\cdot,t)\|_{W^{1,\infty}(\Ga_h[\xs(t)])} \leq &\ h^{(\kappa-1)/2}  \\
		\|e_n(\cdot,t)\|_{W^{1,\infty}(\Ga_h[\xs(t)])} \leq &\ h^{(\kappa-1)/2} 
	\end{aligned} \qquad \textrm{ for } \quad t\in[0,t^*].
\end{equation}
We first prove the stated error bounds for $0\leq t \leq t^*$. At the end of the proof we will show that $t^*$ actually coincides with $T$.

Since the reference finite element functions $x_h^*(\cdot,t)$, $v_h^*(\cdot,t)$, $u_h^*(\cdot,t)$ and $H_h^*(\cdot,t)$, $\nu_h^*(\cdot,t)$ on the interpolated bulk $\bar\varOmega_h[\xs(t)]$ and surface $\Ga_h[\xs(t)]$, respectively, have $W^{1,\infty}$ norms that are bounded independently of $h$ for all $t\in[0,T]$,
the bounds~\eqref{ass-err-linf} together with Lemma~\ref{lemma:theta-independence}  imply that the $W^{1,\infty}$ norms of the bulk--surface finite element functions $x_h(\cdot,t),v_h(\cdot,t),u_h(\cdot,t)$ and $H_h(\cdot,t)$,  $\nu_h(\cdot,t)$  on $\bar\varOmega_h[\bfx(t)]$ and $\Ga_h[\bfx(t)]$, respectively, are also bounded independently of $h$ and $t\in[0,t^*]$, and so are their lifts to $\bar\varOmega_h[\xs(t)]$ and $\Ga_h[\xs(t)]$. 

The estimate on the position errors $e_x$ in \eqref{ass-err-linf} and the $W^{1,\infty}$ bound on $v_h$ immediately imply that the results of Section~\ref{subsec:aux} apply. In particular, due to the bounds in \eqref{ass-err-linf} (for a sufficiently small $h \leq h_0$), the main condition of Lemma~\ref{lemma:theta-independence} and also \eqref{e-theta} is satisfied (with $e_h^\theta=e_x^\theta$). Hence the $h$-uniform norm equivalences and  estimates in \eqref{matrix difference bounds} hold between the surfaces defined by $\bfx$ and $\xs$. Similarly, again due to \eqref{ass-err-linf}, the bound \eqref{vh-bound} also holds, and hence the estimates in \eqref{matrix difference bounds-t}, \eqref{matrix derivatives} and the $h$-uniform norm equivalences in time \eqref{norm-equiv-t} also hold. 

In the following, $c$ and $C$ denote generic constants that take different values on different occurrences. In contrast, constants with a subscript (such as $c_0$) will play a distinctive role in the proof and will not change their value between appearances.

\subsubsection{Error bound for the generalized Robin boundary value problem}
\label{subsubsec:err-u-bounds}
We test \eqref{err-u} with $\bfe_\bfu$ and obtain
\begin{align}
	\label{err-u-test}
	\bfe_\bfu\T \bfL(\bfx) \bfe_\bfu =&\  \bfe_\bfu\T \bigl( \bffu(\bfx,\bfH) - \bffu(\bfx^*,\bfH^*) \bigr)    \nonumber
	\\
	& - \  \bfe_\bfu\T \bigl(\bfL(\bfx)-\bfL(\bfx^*) \bigr)\bfu^* 
	\\
	& -\  \bfe_\bfu\T\bfM_{\bar\Om}(\bfx^*) \bfd_\bfu \,. \nonumber
\end{align}
We begin with estimating the first term on the right-hand side. We introduce finite element functions on the intermediate surface 
$\Ga_h^\theta=\Ga[\bfx^\theta]$ with $\bfx^\theta=\xs+\theta \bfe_\bfx$ $(0\le \theta\le 1)$ that are defined by
\begin{align*}
	H_h^\theta &= \sum_{i=1}^{N_\Ga} (H_i^* + \theta (H_i-H_i^*)) \varphi_i[\xs+\theta \bfe_\bfx] ,
	\\
	Q_h^\theta &= \sum_{i=1}^{N_\Ga} (Q_i^* + \theta (Q_i-Q_i^*)) \varphi_i[\xs+\theta \bfe_\bfx] ,
\end{align*}
where $Q_i^*$ and $Q_i$ are the values of the source term $Q$ at $x_i^*$ and $x_i$, respectively. Their nodal vectors are denoted
$\bfH^\theta =(H_i^\theta)$ and $\bfQ^\theta =(Q_i^\theta)$.
We let $e_u^\theta$ be the finite element function on $\varOmega_h^\theta=\varOmega[\bfx^\theta]$ with nodal vector $\bfe_\bfu$, and we set $e_u=e_u^0$ for later use. We then have 
\begin{align*}
	I := \bfe_\bfu\T \bigl( \bffu(\bfx,\bfH) - \bffu(\bfx^*,\bfH^*) \bigr) 
	= &\ \int_0^1  \frac{d}{d\theta} \Bigl( \bfe_\bfu\T  \bffu(\bfx^\theta ,\bfH^\theta)\Bigr) \, d\theta
	\\
	= &\ \int_0^1 \frac{\d}{\d\theta} \biggl( -\int_{\Om_h^\theta} e_u^\theta  +
	\int_{\Ga_h^\theta} e_u^\theta (H_h^\theta+Q_h^\theta) \biggr) \d\theta.
\end{align*}
By the Leibniz formula and using that $\mat_\theta e_u^\theta=0$ by the transport property, this becomes
\begin{align*}
	I = \int_0^1  \biggl( -\int_{\Om_h^\theta} e_u^\theta  (\nb \cdot e_x^\theta) +
	\int_{\Ga_h^\theta} e_u^\theta (H_h^\theta+Q_h^\theta) (\nb \cdot e_x^\theta) 
	+ \int_{\Ga_h^\theta} e_u^\theta (\mat_\theta H_h^\theta+ \mat_\theta Q_h^\theta)  \biggr) \d\theta,
\end{align*}
which we bound with the Cauchy--Schwarz inequality, noting further that 
$\partial_\theta^\bullet H_h^\theta = e_H^\theta$ and $\partial_\theta^\bullet Q_h^\theta = e_Q^\theta$,
\begin{align*}
	I \le &\int_0^1  \biggl( \| e_u^\theta\|_{L^2(\Om_h^\theta)}  \, \| \nb \cdot e_x^\theta \| _{L^2(\Om_h^\theta)}  
	+ \| e_u^\theta \|_{L^2(\Ga_h^\theta)}  \,    \| H_h^\theta+Q_h^\theta \|_{L^\infty(\Ga_h^\theta)}  \,  \| \nb \cdot e_x^\theta \|_{L^2(\Ga_h^\theta)}  
	\\
	& \quad 
	+ \| e_u^\theta \|_{L^2(\Ga_h^\theta)}  \, \bigl( \| e_H^\theta + e_Q^\theta \|_{L^2(\Ga_h^\theta)}\bigr) \biggr) \d\theta.
\end{align*}
By Lemma~\ref{lemma:theta-independence} and condition~\eqref{ass-err-linf}, which yields that $H_h^\theta$ is bounded in the maximum norm by a constant independent of $h$ and $\theta$,
this integral can be further bounded by
\begin{align*}
	I \le  &\ \| e_u \|_{L^2(\Om_h[\xs])}  \, \| e_x \| _{H^1(\Om_h[\xs])}
	+ c \| e_u \|_{L^2(\Ga_h[\xs])}   \Bigl( \| e_x \|_{H^1(\Ga_h[\xs])}   + \| e_H + e_Q \| _{L^2(\Om_h[\xs])} \Bigr).
\end{align*}
The second term (denoted II) on the right-hand side of \eqref{err-u-test} 
can be estimated with Lemma~\ref{lemma:matrix differences}
(cf. \cite[Lemma~4.1]{KLLP2017} for the surface terms and \cite[Lemma~5.1]{Edelmann_harmonicvelo} for the bulk term). 
We obtain
$$
II 
\le c \| e_u \|_{H^1(\Om_h[\xs],\Ga_h[\xs])} \,  \| e_x\|_{H^1(\Om_h[\xs],\Ga_h[\xs])} \| u_h^* \|_{W^{1,\infty}(\Omh[\xs],\Gah[\xs])} \,.
$$
Finally, we have 
$$
III = - \bfe_\bfu\T\bfM_{\bar\Om}(\xs) \bfd_\bfu = \int_{\Om_h[\xs]} e_u  d_u \le 
c \| e_u \|_{L^2(\Om_h[\xs])} \cdot  \| d_u \|_{L^2(\Om_h[\xs])},
$$
so that altogether
\begin{align*}
	\nonumber
	\| \bfe_\bfu \|_{\bfL(\xs)} & =
	\| e_u \|_{H^1(\Om_h[\xs],\Ga_h[\xs])}  
	\\
	& \le c \Bigl(  \| e_x \|_{H^1(\Om_h[\xs],\Ga_h[\xs])}  +  \| e_H + e_Q \| _{L^2(\Ga_h[\xs])} 
	+ \| d_u \|_{L^2(\Om_h[\xs])}\Bigr)
	\nonumber
	\\
	& = c \Bigl(  \| \bfe_\bfx \|_{\bfL(\xs)}  +  \| \bfe_\bfH + \bfe_\bfQ \| _{\bfM_\Ga(\xs))} 
	+ \| \bfd_\bfu \|_{\bfM_{\bar\Om}(\xs)}\Bigr).
	\label{eu-bound}
\end{align*}
Using that  $\| \bfe_\bfQ \| _{\bfM_\Ga(\xs))} \le c \| \bfe_\bfx  \| _{\bfM_\Ga(\xs))}  \le c \| \bfe_\bfx \|_{\bfL(\xs)} $,
this becomes
\begin{equation}
	\| \bfe_\bfu \|_{\bfL(\xs)} 
	\le
	c \Bigl(  \| \bfe_\bfx \|_{\bfL(\xs)}  +  \| \bfe_\bfH \| _{\bfM_\Ga(\xs))} 
	+ \| \bfd_\bfu \|_{\bfM_{\bar\Om}(\xs)}\Bigr).
	\label{eu-bound}
\end{equation}

In addition to this bound of $e_u$, we also need a bound for its material time derivative $\mat e_u$. To this end we differentiate the error equation \eqref{err-u} with respect to time and test with $\dot \bfe_\bfu$. Using the same arguments as in bounding $\bfe_\bfu$ before, together with the bound 
$\| \dot \bfe_\bfQ \| _{\bfM_\Ga(\xs)} \le c \bigl(\| \bfe_\bfx  \| _{\bfM_\Ga(\xs)} + \| \bfe_{\bfv_\Ga}  \|_{\bfM_\Ga(\xs)} \bigr)$,
we find
\begin{align} \label{dot-eu-bound}
	\| \dot \bfe_\bfu \|_{\bfL(\xs)} \le c &\Bigl( \| \bfe_\bfu \|_{\bfL(\xs)}  + \| \bfe_\bfx \|_{\bfL(\xs)} + \| \bfe_{\bfv_\Ga}  \|_{\bfM_\Ga(\xs)}\Bigr.
	\\ \nonumber
	&\Bigl.+  \| \bfe_\bfH \| _{\bfM_\Ga(\xs)} +   \| \dot \bfe_\bfH  \| _{\bfM_\Ga(\xs)} 
	+ \| \dot \bfd_\bfu \|_{\bfM_{\bar\Om}(\xs)}\Bigr).
\end{align}

\subsubsection{Error bound for the forced mean curvature flow}


We test the error equation \eqref{err-H} for $\bfe_\bfH$ with $\dot \bfe_\bfH$. On the left-hand side we obtain
\begin{align*}
	\dot\bfe_\bfH \T \bfM_\Ga(\bfx) \dot\bfe_\bfH + \dot\bfe_\bfH \T \bfA_\Ga(\bfx) \bfe_\bfH
	= \|\dot\bfe_\bfH\|_{\bfM_\Ga(\bfx)}^2 + 
	\frac12 \,\diff\, \|\bfe_\bfH\|_{\bfA_\Ga(\bfx)}^2 - \frac12\, \bfe_\bfH\T \Bigl(\diff \bfA(\bfx)\Bigr) \bfe_\bfH.
\end{align*}
On the right-hand side, the resulting terms that do not depend on $u$ are estimated in exactly the same way as in \cite[Section 7]{MCF}, and for the extra terms we estimate as follows. The most critical term is $\dot \bfe_\bfH^T \bfA_\Ga(\bfx^*)\bftr\bfe_\bfu$, 
whose direct estimation with the Cauchy--Schwarz inequality would contain the $\bfA_\Ga(\bfx^*)$-norm of $\dot\bfe_\bfH$, which cannot be suitably bounded.
Instead we rewrite
$$
\dot \bfe_\bfH^T \bfA_\Ga(\bfx^*)\bftr\bfe_\bfu = \diff (\bfe_\bfH^T \bfA_\Ga(x^*) \bftr\bfe_u ) 
- \bfe_\bfH^T  \biggl(\diff\bfA_\Ga(\bfx^*) \biggr)\bftr \bfe_\bfu 
- \bfe_H^T \bfA_\Ga(\bfx^*)\bftr \dot \bfe_\bfu.
$$
For the last term with $\dot \bfe_\bfu$, after bounding via the Cauchy--Schwarz inequality and Young's inequality with a small $\rho>0$,
\begin{align*}
	\bfe_\bfH^T \bfA_\Ga(\bfx^*)\bftr \dot \bfe_\bfu 
	\leq &\ 
	\| \bfe_\bfH \|_{\bfA_\Ga(\bfx^*)}  \| \bftr \dot\bfe_\bfu \|_{\bfA_\Ga(\bfx^*)} 
	\\
	\leq &\ 
	\| \bfe_\bfH \|_{\bfK_\Ga(\bfx^*)}  \| \dot\bfe_\bfu\|_{\bfL(\bfx^*)} 
	\le
	\tfrac12 \rho^{-1} \| \bfe_\bfH \|_{\bfK_\Ga(\bfx^*)}^2 + \tfrac12 \rho  \| \dot\bfe_\bfu \|_{\bfL(\bfx^*)}^2,
\end{align*}
we use the bound \eqref{dot-eu-bound} for $\| \dot\bfe_\bfu \|_{\bfL(\bfx^*)}$. We note that only the $\bfM_\Ga$-norm of $\dot \bfe_\bfH$ appears in this bound, and $\|\dot \bfe_\bfH\|_{\bfM_\Ga(\bfx^*)}^2$ times a small constant factor can be absorbed in the corresponding term on the left-hand side, since the $\bfM_\Ga$-norms at $\bfx$ and $\bfx^*$ are equivalent uniformly in $h$.

The term $ \bfe_\bfH^T \bigl(  \bfA_\Ga(\bfx) - \bfA_\Ga(\bfx^*) \big) \bftr\bfu$ is of the same type as $\bfe_\bfH^T  
\bigl(  \bfA_\Ga(\bfx) - \bfA_\Ga(\bfx^*) \big) \bfH^*$, which was estimated in \cite[after (7.27)]{MCF}. Here, the same proof and the use of condition \eqref{ass-err-linf} for $\bfu=\bfu^*+\bfe_\bfu$ yield the bound
\begin{align*}
	\dot\bfe_\bfH^T \bigl(  \bfA_\Ga(\bfx) - \bfA_\Ga(\bfx^*) \bigr) \bftr\bfu \,\le\,
	& \diff \Bigl(\bfe_\bfH^T \bigl(  \bfA_\Ga(\bfx) - \bfA_\Ga(\bfx^*) \bigr) \bftr\bfu\Bigr) 
	+ c \,\| \bfe_\bfH \|_{\bfA_\Ga(\bfx^*)} \bigl(  \| \bfe_\bfx \|_{\bfK_\Ga(\bfx^*)} + \| \bfe_\bfv \|_{\bfK_\Ga(\bfx^*)} \bigr).
\end{align*}
Proceeding as in \cite[after (7.30)]{MCF}, which includes an integration in time, we obtain a bound for $\bfe_\bfH$, cf.~\cite[(7.33)]{MCF},
\begin{align} \label{eH-bound}
	\begin{aligned}
		\normKt{\bfe_\bfH(t)}^2  \leq &\
		c \int_0^t \Bigl( \|\ex(s)\|_{\bfL(\xs(s))}^2 + \normKs{\bfe_{\bfv_\Ga}(s)}^2 
		+ \normKs{\bfe_\bfn(s)}^2 + \normKs{\bfe_\bfH(s)}^2 \Bigr) \d s 
		\\
		&\ + c \,\normKt{\ex(t)}^2 +  c\,\normKo{\bfe_\bfH(0)}^2 
		\\
		&\ + c \int_0^t \Bigl( \normMs{\bfd_\bfH(s)}^2 		
		+ \| \bfd_\bfu(s)\|_{\bfM_{\bar\varOmega}(\xs(s))}^2+ 
		\| \dot \bfd_\bfu(s)\|_{\bfM_{\bar\varOmega}(\xs(s))}^2 \Bigr)  \d s  .
	\end{aligned}
\end{align}
Testing the error equation \eqref{err-n} for $\bfe_\bfn$ with $\dot\bfe_\bfn$, we obtain in
the same way, but now with terms that can be directly estimated,  the bound for $\bfe_\bfn$,
\begin{align} \label{en-bound}
	\begin{aligned}
		\normKt{\bfe_\bfn(t)}^2  
		\leq &\
		c \int_0^t \Bigl( \|\ex(s)\|_{\bfL(\xs(s))}^2 + \normKs{\bfe_{\bfv_\Ga}(s)}^2 
		+ \normKs{\bfe_\bfn(s)}^2 + \normKs{\bfe_\bfH(s)}^2 \Bigr) \d s 
		\\
		&\ + c \,\normKt{\ex(t)}^2 +  c\,\normKo{\bfe_\bfn(0)}^2 
		\\
		&\ + c \int_0^t \Bigl( \normMs{\bfd_\bfn(s)}^2  + \| \bfd_\bfu(s)\|_{\bfM_{\bar\varOmega}(\xs(s))}^2\Bigr)  \d s  .
	\end{aligned}
\end{align}
Noting $V=-H+u$ and using Lemma 5.3 of \citep{KoLL21} in \eqref{err-v-Gamma} in the same way as in deriving the velocity error bound (5.44) of \citep{KoLL21}, we obtain the following error bound for the boundary velocity:
\begin{align}
	\| \bfe_{\bfv_\Ga} (t) \|_{\bfK_\Ga(\xs(t))} &\le  c \Bigl( \| \bfe_{\bfu} (t) \|_{\bfK_\Ga(\xs(t))} +
	\| \bfe_{\bfH} (t) \|_{\bfK_\Ga(\xs(t))} + \| \bfe_{\bfn} (t) \|_{\bfK_\Ga(\xs(t))} \Bigr) 
	\label{ev-Gamma-bound}
	+ \| \bfd_{\bfv_\Ga} (t) \|_{\bfK_\Ga(\xs(t))}.
\end{align}
We have $\dot\bfe_\bfx=\bfe_\bfv$ and $\bfe_\bfx(0)=0$, and hence
\begin{equation}
	\label{ex-Gamma-bound}
	\| \bfe_{\bfx_\Ga}(t) \|_{\bfK_\Ga(\bfx^*(t))} \le \int_0^t \| \bfe_{\bfv_\Ga}(s) \|_{\bfK_\Ga(\bfx^*(s))}\, \d s.
\end{equation}

\subsubsection{Error bound for the discrete harmonic velocity extension}

The error equation \eqref{err-v-Omega} is the matrix--vector formulation of a discretized Dirichlet problem:
Given $e_{v_\Ga} \in  \calS_h[\xs]$,
find $e_v\in \calV_h[\xs]$ with $\gamma_h e_v = e_{v_\Ga} $ such that
\begin{equation}
	\int_{\varOmega_h[\xs]} \nb e_v \cdot \nb \varphi_h = \ell(\varphi_h) \qquad\text{for all}\, \varphi_h \in \calV_h^0[\xs],
\end{equation}
where $\ell$ denotes the linear form on $\calV_h^0[\xs]$ that maps $\varphi_h \in \calV_h^0[\xs]$ with nodal vector $\bfphi$ to
\begin{equation}
	\label{ell}
	\begin{aligned}
		\ell(\varphi_h) = 
		&-\bfphi\T \left( \bfA_{\bar\Om}(\bfx)-\bfA_{\bar\Om}(\bfx^*) \right) \bfe_\bfv \\
		&-\bfphi\T \left( \bfA_{\bar\Om}(\bfx)-\bfA_{\bar\Om}(\bfx^*) \right) \bfv^* \\
		&-\bfphi\T \bfM_{\bar\Om}(\bfx^*) \dv. 
	\end{aligned}
\end{equation}
Proposition \ref{prop:Dirichlet-h} (concerning the Dirichlet data $e_{v_\Ga}$) and standard finite element theory (concerning the functional $\ell$) give us the bound 
\begin{equation}
	\begin{aligned}\label{erroreq:harmonic extension dual estimate}
		\| e_v \|_{H^1(\Omh[\bfx^*])} \le \| \ell \|_{H_h^{-1}(\Om_h[\xs])} + \| e_{v_\Ga} \|_{H^{1/2}(\Gah[\bfx^*])} \,,
	\end{aligned}
\end{equation}
where
\[ 
\| \ell \|_{H_h^{-1}(\Om_h[\xs])} = \sup_{0 \ne \varphi_h \in \calV_h^0[\bfx^*]} \frac{\ell(\varphi_h)}{\| \varphi_h \|_{H^1(\Omh[\bfx^*])}} .
\]
We estimate the three terms of \eqref{ell} separately. The steps are analogous to previous estimates based on Section~\ref{subsec:aux}.
\\[1mm]
(i) For the first term we have, for sufficiently small $h$, using condition \eqref{ass-err-linf},
\begin{align*}
	\bfphi\T \left( \bfA_{\bar\Om}(\bfx)-\bfA_{\bar\Om}(\bfx^*) \right) \bfe_\bfv 
	&\le c\, \| \varphi_h \|_{H^1(\Omh[\bfx^*])} \| \nb e_x \|_{L^\infty(\Omh[\bfx^*])} \| e_v \|_{H^1(\Omh[\bfx^*])} \\
	&\le c \, h^{(\kappa-3/2)/2} \,\| \varphi_h \|_{H^1(\Omh[\bfx^*])}  \| e_v \|_{H^1(\Omh[\bfx^*])} \\
	&= c \, h^{(\kappa-3/2)/2} \,\| \varphi_h \|_{H^1(\Omh[\bfx^*])}  \| \bfe_\bfv \|_{\bfK_{\bar\Om}(\bfx^*)}.
\end{align*}
(ii) For the second term we obtain
\begin{align*}
	\bfphi\T \left( \bfA_{\bar\Om}(\bfx)-\bfA_{\bar\Om}(\bfx^*) \right) \bfv^* 
	&\le c \| \varphi_h \|_{H^1(\Omh[\bfx^*])} \| \nb e_x \|_{L^2(\Omh[\bfx^*])} \| v_h^* \|_{W^{1,\infty}(\Omh[\bfx^*])}.
\end{align*}
(iii) The third term is bounded by
\begin{align*}
	\bfphi\T\bfM_{\bar\Om}(\bfx^*) \dv \leq \| \varphi_h \|_{H^1(\Omh[\bfx^*])} \| \dv \|_{\star,\bfx^*} .
\end{align*}

Combining these estimates yields
\begin{align*}
	\| \ell \|_{H_h^{-1}(\Om_h[\xs])} \le c \left( h^{(\kappa-3/2)/2} \| \bfe_\bfv \|_{\bfK_{\bar\Om}(\bfx^*)} + \| \bfe_\bfx \|_{\bfA_{\bar\Om}(\bfx^*)} + \| \dv \|_{\star,\bfx^*} \right).
\end{align*}
For $h \le h_0$ sufficiently small we can absorb the error $\bfe_\bfv$ and insert the above estimate for $\ell$ in \eqref{erroreq:harmonic extension dual estimate} to obtain
\begin{equation}\label{ev-bound}
	\| \bfe_\bfv \|_{\bfK_{\bar\Om}(\bfx^*)} \le c \left( \| \bfe_{\bfv_\Ga} \|_{\bfK_\Ga(\bfx^*)} + \| \bfe_\bfx \|_{\bfA_{\bar\Om}(\bfx^*)} + \| \dv \|_{\star,\bfx^*} \right).
\end{equation}
We have $\dot\bfe_\bfx=\bfe_\bfv$ and $\bfe_\bfx(0)=0$, and hence
\begin{equation}
	\label{ex-bound}
	\| \bfe_\bfx(t) \|_{\bfK_{\bar\Om}(\bfx^*(t))} \le \int_0^t \| \bfe_\bfv(s) \|_{\bfK_{\bar\Om}(\bfx^*(s))}\, \d s.
\end{equation}

\subsubsection{Combining the error bounds} 
By now we have the coupled inequalities \eqref{eu-bound}, \eqref{eH-bound} and \eqref{en-bound}, \eqref{ev-Gamma-bound} and \eqref{ex-Gamma-bound}, \eqref{ev-bound} and \eqref{ex-bound} for the error vectors $\bfe_\bfu$, $\bfe_\bfH$ and $\bfe_\bfn$, $\bfe_{\bfv_\Ga}$ and $\bfe_{\bfx_\Ga}$, $\bfe_\bfv$ and $\bfe_\bfx$, respectively.

Combining the bounds \eqref{ex-Gamma-bound} and \eqref{ex-bound} for $\bfe_\bfx$, inserting the bounds \eqref{ev-Gamma-bound} and \eqref{ev-bound} for $\bfe_{\bfv_\Ga}$ and $\bfe_\bfv$ and noting that $\| \bfe_\bfx \|_{\bfL(\xs)}^2 = \| \bfe_\bfx \|_{\bfK_\Ga(\xs)}^2 + \| \bfe_\bfx \|_{\bfK_{\bar\Om}(\xs)}^2$, we obtain after using the Gronwall inequality to get rid of the term $\| \bfe_\bfx \|_{\bfL(\xs)}$ under the integral,
\begin{align}
	\| \bfe_\bfx(t) \|_{\bfL(\xs(t))} 
	\nonumber
	&\le c \int_0^t \Bigl(  \| \bfe_{\bfu} (s) \|_{\bfK_\Ga(\xs(s))} +
	\| \bfe_{\bfH} (s) \|_{\bfK_\Ga(\xs(s))} + \| \bfe_{\bfn} (s) \|_{\bfK_\Ga(\xs(s))} \Bigr) \d s\\
	&\quad + c \int_0^t \Bigl( \| \bfd_{\bfv_\Ga} (s) \|_{\bfK_\Ga(\xs(s))} + \| \dv(s) \|_{\star,\bfx^*(s)} \Bigr) \d s.
	\label{ex-L-bound}
\end{align}
Inserting this bound into \eqref{eu-bound} and using the Gronwall inequality to get rid of the term $\| \bfe_\bfu \|_{\bfL(\xs)}$ under the integral yields
\begin{align}
	\nonumber
	\| \bfe_\bfu(t) \|_{\bfL(\xs(t))} &\le  c \int_0^t \Bigl( \| \bfe_{\bfH} (s) \|_{\bfK_\Ga(\xs(s))} + \| \bfe_{\bfn} (s) \|_{\bfK_\Ga(\xs(s))} \Bigr) \d s
	\\
	&\quad + c\, \| \bfe_\bfH(t)  \| _{\bfM_\Ga(\xs(t)))}   
	+ c\,\| \bfd_\bfu(t) \|_{\bfM_{\bar\Om}(\xs(t))} 
	\nonumber
	\\
	&\quad + c \int_0^t \Bigl( \| \bfd_{\bfv_\Ga} (s) \|_{\bfK_\Ga(\xs(s))} + \| \dv(s) \|_{\star,\bfx^*(s)} \Bigr) \d s.  
	\label{eu-bound-1}
\end{align}
Inserting  \eqref{ev-Gamma-bound}, \eqref{ex-L-bound} and \eqref{eu-bound-1} into \eqref{eH-bound} yields
\begin{align} \label{eH-bound-1}
	\begin{aligned}
		\normKt{\bfe_\bfH(t)}^2  \leq &\
		c \int_0^t \Bigl( \normKs{\bfe_\bfH(s)}^2 + \normKs{\bfe_\bfn(s)}^2 \Bigr) \d s 
		+  c\,\normKo{\bfe_\bfH(0)}^2 \\
		&\ + c \int_0^t \Bigl( \normMs{\bfd_\bfH(s)}^2 		
		 + \| \bfd_\bfu(s)\|_{\bfM_{\bar\varOmega}(\xs(s))}^2+ 
		\| \dot \bfd_\bfu(s)\|_{\bfM_{\bar\varOmega}(\xs(s))}^2 
		\\
		&\ \hphantom{+ c \int_0^t \Big( } + \| \bfd_{\bfv_\Ga} (s) \|_{\bfK_\Ga(\xs(s))}^2 + \| \dv(s) \|_{\star,\bfx^*(s)}^2 
		\Bigr)  \d s  .
	\end{aligned}
\end{align}
Similarly, inserting \eqref{ex-L-bound} and \eqref{ev-Gamma-bound} into \eqref{en-bound} yields
\begin{align} \label{en-bound-1}
	\begin{aligned}
		\normKt{\bfe_\bfn(t)}^2  \leq &\
		c \int_0^t \Bigl( 
		\normKs{\bfe_\bfH(s)}^2 + \normKs{\bfe_\bfn(s)}^2 + \normMs{\bfe_\bfQ(s)}^2 \Bigr) \d s 
		+   c\,\normKo{\bfe_\bfn(0)}^2 
		\\
		&\ + c \int_0^t \Bigl( \normMs{\bfd_\bfn(s)}^2  + \| \bfd_\bfu(s)\|_{\bfM_{\bar\varOmega}(\xs(s))}^2
		+ \| \bfd_{\bfv_\Ga} (s) \|_{\bfK_\Ga(\xs(s))}^2 + \| \dv(s) \|_{\star,\bfx^*(s)}^2 
		\Bigr)  \d s  .
	\end{aligned}
\end{align}
Summing up \eqref{eH-bound-1} and \eqref{en-bound-1} and using the Gronwall inequality thus yields
\begin{align} \label{eH-en-bound}
	\begin{aligned}
		\normKt{\bfe_\bfH(t)}^2  + \normKt{\bfe_\bfn(t)}^2 
		&\  \le c \Bigl( \|\bfe_\bfH(0)\|_{\bfK_\Ga(\xs(0))}^2  + 
		\|\bfe_\bfn(0)\|_{\bfK_\Ga(\xs(0))}^2\Bigr) 
		\\
		&\ + c \int_0^t \Bigl( \normMs{\bfd_\bfH(s)}^2 + \normMs{\bfd_\bfn(s)}^2  
		+ \| \bfd_\bfu(s)\|_{\bfM_{\bar\varOmega}(\xs(s))}^2
		\\
		&\ \hphantom{+ c \int_0^t \Big( } + \| \dot \bfd_\bfu(s)\|_{\bfM_{\bar\varOmega}(\xs(s))}^2
		+ \| \bfd_{\bfv_\Ga} (s) \|_{\bfK_\Ga(\xs(s))}^2 + \| \dv(s) \|_{\star,\bfx^*(s)}^2 
		\Bigr)  \d s  .
	\end{aligned}
\end{align}

Inserting this bound first into \eqref{eu-bound-1}, then both into \eqref{ex-L-bound} and finally that bound into \eqref{ev-Gamma-bound} and \eqref{ev-bound} yields the stated stability bound \eqref{eq:stability bound} on the interval $[0,t^*]$.

It remains to show that $t^* = T$ if $h$ is sufficiently small. 
We use the assumed defect bounds to obtain error bounds of order~$\kappa$:
\begin{align*}
	\|\ex(t)\|_{\bfL(\xs(t))} +  \|\ev(t)\|_{\bfL(\xs(t))} +  \|\eu(t)\|_{\bfL(\xs(t))}  
	+ \ \normKt{\bfe_\bfH(t)} + \normKt{\bfe_\bfn(t)} 
	\leq  Ch^\kappa.
\end{align*}
By the inverse inequality \cite[Theorem~4.5.11]{BrennerScott}, we then have for $t\in[0,t^*]$ 
\begin{align*}
	& \|e_x(\cdot,t)\|_{W^{1,\infty}(\Om_h[\xs\t])}
	+ \|e_v(\cdot,t)\|_{W^{1,\infty}(\Om_h[\xs\t])} 
	+ \|e_u(\cdot,t)\|_{W^{1,\infty}(\Om_h[\xs\t])}
	\\
	&\leq  ch^{-3/2} \bigl(   \|\ex(t)\|_{\bfL(\xs(t))} +  \|\ev(t)\|_{\bfL(\xs(t))} +  \|\eu(t)\|_{\bfL(\xs(t))}  \bigr)
	\leq c \, C h^{\kappa-3/2} \leq \tfrac12 h^{(\kappa-3/2)/2} 
\end{align*}
and
\begin{align*}
	&\ \|e_H(\cdot,t)\|_{W^{1,\infty}(\Ga_h[\xs\t])}
	+ \|e_\nu(\cdot,t)\|_{W^{1,\infty}(\Ga_h[\xs\t])} 
	\\ 
	&\ \leq c h^{-1} \bigl( \normKt{\bfe_\bfH(t)} + \normKt{\bfe_\bfn(t)} \bigr) \leq c \,C h^{\kappa-1} \leq \tfrac12 h^{(\kappa-1)/2}         
\end{align*}
for sufficiently small $h$.
Hence we can extend the bounds \eqref{ass-err-linf} beyond $t^*$, which contradicts the maximality of $t^*$ unless $t^*=T$. Therefore, we have the stability bound \eqref{eq:stability bound} on the whole interval $[0,T]$. 
\qed

\subsection{Why not for the original Eyles--King--Styles model?}
\label{subsec:why-not-EKS}

We cannot prove stability of the numerical method in the case $\clb=0$, which corresponds to the model originally proposed in \citep{EKS19}. The difficulty is caused by the term $\bigl(\bfL(\bfx)-\bfL(\bfx^*) \bigr)\bfu^*$ in the error equation \eqref{err-u} for $\bfe_\bfu$, which in particular contains 
$\bigl(\bfA_{\bar\Om}(\bfx)-\bfA_{\bar\Om}(\bfx^*) \bigr)\bfu^*$. 
We see no way to bound the $L^2(\Om_h[\xs])$ norm (or the $H^{-1/2}$ norm) of the finite element function with this nodal vector  in terms of the $H^1(\Om_h[\xs],\Ga_h[\xs])$ norm of the position error~$e_x$. If such a bound had been available, then a result similar to Proposition~\ref{prop:robin-h} (with an additional $H^{-1/2}(\Om)$ forcing term in the bulk) would have allowed us to control the $H^1(\Ga_h[\xs])$ norm of the error function $e_u$, as is crucially required in the above proof.
As Subsection~\ref{subsubsec:err-u-bounds} shows, this obstacle does not appear with the generalized Robin boundary condition, for which the $H^1(\Ga_h[\xs])$ norm of the error function $e_u$ can be controlled directly by energy estimates.

\section{Consistency}
\label{section:consistency}

In this section we prove that the initial values and the defects appearing in the stability bound of Proposition~\ref{prop:stability} are indeed bounded as $O(h^k)$ in the appropriate norms that appear in the stability result. The defect estimates follow the ideas of \citep{DziukElliott_L2,KLLP2017,MCF,MCF_soldriven}.

\subsection{Finite element projections of the exact solution}

\subsubsection{Interpolations and their errors}

The vectors $\xs$ collect the nodal values of the interpolation of the exact solution $X$, i.e.~of $X_h^*(\cdot,t) = \Ih^\Om X(\cdot,t) \in \calV_h[\xls^0]^3$ , with its lift $I_h X = (\Ih^\Om X)^\ell$ being a function on the exact domain and its boundary. 

Similarly, the interpolation operator on the boundary is denoted by $I_h^\Ga = (\Ih^\Ga)^\ell$. For vector valued functions both interpolations are defined com\-po\-nent\-wise.

We recall that the interpolation operators satisfy the error bounds, cf.~\citep{ElliottRanner_bulksurface,Demlow2009,ElliottRanner_unified}, for any $w \in H^{k+1}(\Om)$ and any $z \in H^{k+1}(\Ga)$, for $1\leq j \leq k$ and $h \leq h_0$,
\begin{equation}
\label{eq:interpolation errors}
\begin{aligned}
	\|w - I_h^\Om w\|_{L^2(\Om)} + h \|w - I_h^\Om w\|_{H^1(\Om)} \leq &\ c h^{j+1} \|w\|_{H^{j+1}(\Om)} , \\
	\|z - I_h^\Ga z\|_{L^2(\Ga)} + h \|z - I_h^\Ga z\|_{H^1(\Ga)} \leq &\ c h^{j+1} \|z\|_{H^{j+1}(\Ga)} .
\end{aligned}
\end{equation}

\subsubsection{Denoting bilinear forms}

We briefly introduce a short yet general notation for bilinear forms on continuous and discrete domains. 
Let $\star$ denote any of the domains $\Om, \Ga$ or their interpolations $\Om_h^*, \Ga_h^*$, then
we define the following bilinear forms, for any $w,\vphi$ in the appropriate space,
\begin{align*}
m_\star(w,\phi) = \int_{\star} w \vphi \qquad \text{and} \qquad
a_\star(w,\phi) = \int_{\star} \nb_\star w \cdot \nb_\star \vphi ,
\end{align*}
i.e.~the (semi-)inner products on $L^2(\star)$ and $H^1(\star)$, respectively.

For brevity we further introduce the notation $\a_{\star} = a_{\star} + m_{\star}$ for the $H^1(\star)$-scalar product on a continuous or discrete domain $\star \in \{\Om, \Ga, \Om_h^*, \Ga_h^*\}$.

\subsubsection{A generalized surface Ritz map}

Recall that $\Om_h^*$ is the interpolation of the exact domain $\Om[X]$, whose boundary $\pa \Om_h^*$ is the nodal interpolation of the boundary $\Ga[X]$, i.e.~we have the relations $\Om_h^* = \Ih^\Om X$ and $\Ga_h^* = \Ih^\Ga X = \pa \Om_h^*$. 

We will now define a few different (generalized) Ritz maps, which will play a crucial role in the consistency analysis. Note that since the bulk and the surface are time dependent, all these Ritz maps are depending on time as well, hence, for instance, they do not commute with time derivatives.

We start by defining a Ritz map on the boundary, cf.~\cite[Section~6]{highorderESFEM}, which will be used for the geometric variables $\nu$ and $H$ determining the nodal values of the vectors $\ns$ and $\Hs$. For any $w\in H^1(\Ga[X])$ we first determine $\Rh^\Ga w \in \calS_h[\xs]$ such that, for all $\vphi_h \in \calS_h[\xs]$,
\begin{equation}
\label{eq:Ritz map}
\a_{\Ga_h^*}(\Rh^\Ga w,\vphi_h) = \a_\Ga(w,(\vphi_h)^\ell) ,
\end{equation}
and then defining the Ritz map via the lift $R_h^\Ga w = (\Rh^\Ga w)^\ell$. For vector valued functions the Ritz map is defined componentwise.

The above Ritz map $R_h^\Ga$ satisfies the error estimates, see \cite[Theorem~6.3--6.4]{highorderESFEM}, for any $w \in H^{k+1}(\Ga)$ and for $1 \leq j \leq k$ and $h \leq h_0$,
\begin{subequations}
\label{eq:Ritz map errors}
\begin{align}
	\|w - R_h^\Ga w\|_{L^2(\Ga)} + h \|w - R_h^\Ga w\|_{H^1(\Ga)} \leq &\ c h^{j+1} \|w\|_{H^{j+1}(\Ga)} , \\
	\|\mat (w - R_h^\Ga w)\|_{L^2(\Ga)} + h \|\mat (w - R_h^\Ga w)\|_{H^1(\Ga)} \leq &\ c h^{j+1} \Big( \|w\|_{H^{j+1}(\Ga)} + \|\mat w\|_{H^{j+1}(\Ga)} \Big) .
\end{align}
\end{subequations}

\subsubsection{A generalized bulk Ritz map}

With the use of the interpolated bulk $\Om_h^* = \Ih^\Om X$, its boundary $\Ga_h^* = \Ih^\Ga X$, and the above defined boundary interpolation operators and the surface Ritz map \eqref{eq:Ritz map} we will define a new bulk Ritz map. 

The new bulk Ritz map will be defined via a Poisson problem with inhomogeneous Dirichlet boundary conditions using the Ritz map $\Rh^\Ga$ as boundary data. 
Using such a Ritz map, with the appropriate boundary data (i.e.~the boundary Ritz map of the trace), allows us to decouple certain bulk and surface terms during the consistency analysis.

For any $w \in H^1(\Om;\Ga)$ we define the bulk Ritz map $\Rh^{\textnormal{Ritz}} w$ such that it solves the following inhomogeneous Dirichlet problem using the \emph{surface Ritz map} $\Rh^\Ga (\ga w)$ as Dirichlet boundary values, for any $\vphi_h \in \calV_h^0[\xs]$,
\begin{equation}
\label{eq:Ritz map with b.c. - Ritz}
\begin{aligned}
	\a_{\Om_h^*}(\Rh^{\textnormal{Ritz}} w,\vphi_h) = &\ \a_\Om(w,\vphi_h^\ell) , \\
	\ga_h^* \Rh^{\textnormal{Ritz}} w = &\ \Rh^\Ga (\ga w) .
\end{aligned}
\end{equation}
As before, we set $R_h^{\textnormal{Ritz}} w = (\Rh^{\textnormal{Ritz}} w)^\ell$.

The following error estimates hold for the bulk Ritz map \eqref{eq:Ritz map with b.c. - Ritz}.
\begin{lemma}
\label{lemma:R^Ritz error estimates}
For the bulk Ritz map defined by \eqref{eq:Ritz map with b.c. - Ritz} we have the following bulk--surface error estimates. Using finite elements of polynomial degree at most $k$, for any $w \in H^{k+1}(\Om;\Ga)$, and for $j \leq k$ and $h \leq h_0$, we have
\begin{subequations}
	\label{eq:Ritz map errors in bulk }
	\begin{align}
		\|w - R_h^{\textnormal{Ritz}} w\|_{L^2(\Om;\Ga)} + h \|w - R_h^{\textnormal{Ritz}} w\|_{H^1(\Om;\Ga)} \leq &\ c h^{j+1} \|w\|_{H^{j+1}(\Om;\Ga)} \\
		\|\mat_\Om (w - R_h^\Ga w)\|_{L^2(\Om;\Ga)} + h \|\mat_\Om (w - R_h^\Ga w)\|_{H^1(\Om;\Ga)} \leq &\ c h^{j+1} \Big( \|w\|_{H^{j+1}(\Om;\Ga)} + \|\mat w\|_{H^{j+1}(\Om;\Ga)} \Big) .
	\end{align}
\end{subequations}
\end{lemma}
\begin{proof}
The proof is a technical adaptation of \cite[Lemma~3.8 and 3.10]{ElliottRanner_unified} (handling the Dirichlet boundary data by the harmonic extension operator and its properties) in the setting \cite[Section~8]{ElliottRanner_unified} (in particular note Lemma~8.24 therein).
\end{proof}

\subsubsection{Stability of finite element projections}

The above finite element projection errors directly imply the stability bounds, for $\calI_h = I_h^\Om, I_h^\Ga$ and $\calP_h = R_h^\Ga, R_h^{\textnormal{Ritz}}$, with $h \leq h_0$,
\begin{equation}
\label{eq:projection stabilities}
\|\calI_h v\|_{H^1} \leq (1+ch) \|v\|_{H^2} \andquad \|\calP_h v\|_{H^1} \leq c \|v\|_{H^1} ,
\end{equation}
where the norm $\|\cdot\|_{H^j}$ denotes the suitable $H^j$-norm in the correct domain corresponding to the projection.

\subsection{The chosen finite element projections of the exact solution}
\label{section:consistency - projections}

We now define the suitable finite element projections:

\begin{alignat*}{5}
x_h^* = &\ \Ih^\Om X, &&& &\\
u_h^* = &\ \Rh^{\textnormal{Ritz}} u , \qquad & \text{therefore} & \qquad & \ga_h^* u_h^* = &\ \Rh^\Ga (\ga u), \\
v_h^* = (v_{\Ga,h}^*,v_{\Om,h}^*) = &\ \Ih^\Om v , \qquad & \text{therefore} & \qquad & \ga_h^* v_h^* = v_{\Ga,h}^* = &\ \Ih^\Ga (\ga v), \\
w_h^* = (\nu_h^*,H_h^*)  = &\ (\Rh^\Ga \nu , \Rh^\Ga H) . \quad &&&& 
\end{alignat*}

With these choices the error estimate for the initial data \eqref{eq:stability - initial errors} directly follows.

\subsection{Consistency estimates for the coupled bulk--surface system}

We will estimate the consistency errors \eqref{eq:stability - defect bounds} of the coupled system using the general strategy given below, which follows the ideas of \citep{DziukElliott_L2,KLLP2017,MCF,MCF_soldriven,Edelmann_harmonicvelo}.

The $L^2_h$ (either $L^2(\Om_h^*)$ or $L^2(\Ga_h^*)$) norm of a defect $d_h$ is given by
\begin{equation*}
\|d_h\|_{L^2_h} = \sup_{0 \ne \vphi_h \in L^2_h} \frac{m_h(d_h,\vphi_h)}{\|\vphi_h\|_{L^2_h}} .
\end{equation*}

We then estimate the scalar product using the following approach (see the above references as well):
\begin{itemize}
\item[(a)] We first subtract the weak formulation of the continuous sub-problems from the corresponding defect equation from Section~\ref{section:defects and errors} (given there in a matrix--vector form). 
Therefore the defect is rewritten as the sum of pairs comparing discrete and continuous objects.
\item[(b)] Crucial pairs will vanish using the definitions of the Ritz maps \eqref{eq:Ritz map} and \eqref{eq:Ritz map with b.c. - Ritz}.
\item[(c)] We estimate the remaining pairs separately, using the following strategy: we add and subtract a suitable term, and then use the geometric approximation errors for the bilinear forms, see \citep{DziukElliott_L2}, \cite[Lemma~5.6]{highorderESFEM}, \cite[Lemma~10.3]{ElliottRanner_unified} and apply the stability bounds \eqref{eq:projection stabilities} and the error bounds for the finite element projections \eqref{eq:Ritz map errors in bulk }, \eqref{eq:Ritz map errors}. 
\end{itemize}

In the defect estimates below we will clearly indicate these steps. 

As in the stability analysis, without loss of generality, we will again set the parameters
\begin{equation*}
\alpha = \beta = \mu = 1 .
\end{equation*}

\subsubsection{Consistency estimates for the Robin boundary value problem}

(a) Recalling that $u_h^* = \Rh^{\textnormal{Ritz}} u$, and by subtracting the weak formulation \eqref{robin-weak} with $\vphi^u = (\vphi_h^u)^\ell$ from (the abstract formulation of) the semi-discrete Robin problem \eqref{robin-defect}, we obtain the following representations of the defect $d_u$:
\begin{equation}
\label{eq:defect pairwise - Robin problem}
\begin{aligned}
	m_{\Om_h}(d_u,\vphi_h^u) 
	= &\ a_{\Om_h}(u_h^*,\varphi_h^u) - a_{\Om}(u,(\vphi_h^u)^\ell) \\
	&\ + \Big( a_{\Ga_h}(\ga_h^* u_h^*,\varphi_h^u) - a_{\Ga}(\ga u,(\vphi_h^u)^\ell) \Big) \\
	&\ + \Big( m_{\Ga_h}(\ga_h u_h^*,\varphi_h^u) - m_{\Ga}(\ga u,(\vphi_h^u)^\ell) \Big) \\
	&\ - \Big( \ell_h^*(\varphi_h^u) - \ell((\vphi_h^u)^\ell) \Big)
\end{aligned}
\end{equation}
for which we recall the definition of the (bi)linear forms $a_\Om$, $a_\Ga$, $m_\Ga$, and $\ell$ as well as their semi-discrete counterparts $a_{\Om_h}$, $a_{\Ga_h}$, $m_{\Ga_h}$, and $\ell_h^*$ (note that the definition of semi-discrete linear from $\ell_h^*$ contains the approximations $Q_h^* = \Ih Q$ and $H_h^* = \Rh^\Ga H$).

(b) We start by analysing the first two pairs in $d_u$. 

For the first two pairs, using the definition of the continuous and discrete bilinearforms $a_\Om$ and $a_{\Om_h}$, and $a_\Ga$ and $a_{\Ga_h}$, and the definition of $u_h^* = \Rh^{\textnormal{Ritz}} u$ \eqref{eq:Ritz map with b.c. - Ritz} together with the definition for its trace $\ga_h^* u_h^* = \ga_h^* \Rh^{\textnormal{Ritz}} u = \Rh^\Ga (\ga u)$ via \eqref{eq:Ritz map}, we directly compute
\begin{equation*}
\begin{aligned}
	&\ \Big(a_{\Om_h}(u_h^*,\varphi_h^u) - a_\Om(u,(\vphi_h^u)^\ell)\Big) 
	+ \Big(a_{\Ga_h}(\ga_h^* u_h^*,\varphi_h^u) - a_\Ga(\ga u,(\vphi_h^u)^\ell)\Big) \\
	= &\ \Big(a_{\Om_h}(u_h^*,\varphi_h^u) - \a_\Om(u,(\vphi_h^u)^\ell)\Big) 
	+ \Big(\a_{\Ga_h}(\ga_h^* u_h^*,\varphi_h^u) - \a_\Ga(\ga u,(\vphi_h^u)^\ell)\Big) \\
	&\ - \Big( m_{\Om_h}(u_h^*,\varphi_h^u) - m_\Om(u,(\vphi_h^u)^\ell) \Big) 
	- \Big( m_{\Ga_h}(\ga_h^* u_h^*,\varphi_h^u) - m_\Ga(\ga u,(\vphi_h^u)^\ell) \Big) .
\end{aligned}
\end{equation*}
Crucially, only the last two pairs remain here as the first two pairs vanish, by the definition of the bulk Ritz map $u_h^* = \Rh^{\textnormal{Ritz}} u$, see \eqref{eq:Ritz map with b.c. - Ritz}, in combination with the definition \eqref{eq:Ritz map} and observing the proper choice of boundary condition $\ga_h^* u_h^* = \ga_h^* \Rh^{\textnormal{Ritz}} u = \Rh^\Ga (\ga u)$ in \eqref{eq:Ritz map with b.c. - Ritz}.

(c) 
The remaining pairs in $d_u$ are therefore
\begin{align*}
m_{\Om_h}(d_u,\vphi_h^u) 
= &\ - \Big( m_{\Om_h}(u_h^*,\varphi_h^u) - m_\Om(u,(\vphi_h^u)^\ell) \Big) \\
&\ - \Big( m_{\Ga_h}(\ga_h^* u_h^*,\varphi_h^u) - m_\Ga(\ga u,(\vphi_h^u)^\ell) \Big) \\
&\ + \Big( m_{\Ga_h}(\ga_h u_h^*,\varphi_h^u) - m_{\Ga}(\ga u,(\vphi_h^u)^\ell) \Big)  \\
&\ - \Big( \ell_h^*(\varphi_h^u) - \ell((\vphi_h^u)^\ell) \Big) .
\end{align*}

The first three pairs with the $L^2$-bilinear forms are estimated, using the error estimates \eqref{eq:Ritz map errors} and the stability bound \eqref{eq:projection stabilities} for the Ritz map, and the bulk geometric approximation errors \cite[Lemma~8.24]{ElliottRanner_unified}, as
\begin{align*}
m_{\Om_h}(u_h^*,\varphi_h^u) - m_\Om(u,(\vphi_h^u)^\ell) 
= &\ m_{\Om_h}(\Rh^{\textnormal{Ritz}} u,\vphi_h^u) - m_\Om(R_h^{\textnormal{Ritz}} u,(\vphi_h^u)^\ell) 
+ m_\Om(R_h^{\textnormal{Ritz}} u - u,(\vphi_h^u)^\ell) \\
\leq &\ c h^{k+1} \| u \|_{H^{k+1}(\Om;\Ga)} \| (\vphi_h^u)^\ell \|_{L^2(\Om)} . 
\end{align*}
The two similar terms are estimated analogously (now using \cite[Lemma~5.6]{highorderESFEM}), as $O(h^{k+1})$.

For the last pair, we will recall that the inhomogeneity $Q_h^*$ is the interpolation $\Ih Q$ and $H_h^*$ is the surface Ritz map $\Rh^\Ga H$. 
We then estimate, using the error estimates \eqref{eq:interpolation errors}, \eqref{eq:Ritz map errors}, the stability bounds \eqref{eq:projection stabilities}, and the geometric approximation errors \cite[Lemma~5.6]{highorderESFEM} and \cite[Lemma~8.24]{ElliottRanner_unified}, as
\begin{align*}
\ell_h^*(\varphi_h^u) - \ell((\vphi_h^u)^\ell) 
= &\ m_\Om(1,(\vphi_h^u)^\ell) - m_{\Om_h}(1,\vphi_h^u) \\
&\ + m_{\Ga_h}(Q_h^*,\ga_h\vphi_h^u) - m_\Ga(Q,\ga (\vphi_h^u)^\ell) \\
&\ + \Big( m_{\Ga_h}(H_h^*,\ga_h \vphi_h^u) - m_\Ga(H,\ga (\vphi_h^u)^\ell) \Big) \\
\leq &\ c h^{k+1} \Big( \|Q\|_{H^{k+1}(\Ga)} + \|H\|_{H^{k+1}(\Ga)} \Big) \| (\vphi_h^u)^\ell \|_{H^1(\Om)} .
\end{align*}

Altogether we obtain the following defect bounds in the discrete dual norm:
\begin{align*}
&\ \|d_u\|_{L^2(\Om_h^*)} = O(h^{k+1}) .
\end{align*}

By differentiating \eqref{eq:defect pairwise - Robin problem} in time, and using similar arguments as above -- analogously to \cite[Section~6]{KoLL21} -- yields the consistency bound
\begin{align*}
&\ \|\dot d_u\|_{L^2(\Om_h^*)} = O(h^{k+1}) .
\end{align*}

\subsubsection{Consistency estimates for the forced mean curvature flow}

The defects $d_\nu$ and $d_H$ in the forced mean curvature flow are estimated using the steps (a), (b), and (c), exactly as in \citep{MCF,MCF_soldriven}, 
heavily relying on the the \emph{proper choice of boundary conditions} in the definition of $R_h^{\textnormal{Ritz}}$, i.e.
\begin{equation*}
\ga_h^* u_h^* = \ga_h^* \Rh^{\textnormal{Ritz}} u = \Rh^\Ga (\ga u) .
\end{equation*}

The defect in the velocity $d_v$ is directly estimated, using \cite[Lemma~5.3]{KoLL21} and the interpolation error estimate on $\Ga$ \eqref{eq:interpolation errors} for $v_\Ga^* = \Ih^\Ga (\ga v)$, as
\begin{align*}
\|d_{v_\Ga}\|_{H^1(\Ga_h^*)} 
\leq  &\ c \big\| I_h^\Ga \ga v - \ga v \big\|_{H^1(\Ga)} 
+ \big\| \Ih^\Ga( V_h^* \, \nu_h^*)  - V^{-\ell} \, \nu^{-\ell} \big\|_{H^1(\Ga_h^*)} \\
\leq &\ c h^k \|\ga v \|_{H^{k+1}(\Ga)} + \| V_h^* \, \nu_h^* - \Ih^\Ga V \, \Ih^\Ga\nu \|_{H^1(\Ga_h^*)} 
+ 2 c h^k \big( \|V\|_{H^{k+1}(\Ga)} + \|\nu\|_{H^{k+1}(\Ga)} \big) .
\end{align*}
The middle term here is estimated further as $O(h^k)$, recalling that $V_h^* = - \Rh^\Ga H + \ga_h \Rh^{\textnormal{Ritz}} u$ and using multiple triangle inequalities, together with the approximation properties of the Ritz maps and the surface interpolation.

Altogether we proved the following defect bounds:
\begin{equation*}
\|d_\nu\|_{L^2(\Ga_h^*)} = O(h^{k+1}) , \quad
\|d_H\|_{L^2(\Ga_h^*)} = O(h^{k+1}) , \quad \text{and} \quad
\|d_{v_\Ga}\|_{H^1(\Ga_h^*)} = O(h^k) .
\end{equation*}

\subsubsection{Consistency estimates for the harmonic velocity extension}

Recalling that we have set $v_h^* = \Ih^\Om v$ hence $\ga_h^* v_h^* = v_{\Ga,h}^* = \Ih^\Ga (\ga v)$, which has nodal values $\vs = (\vs_\Ga,\vs_\Om)^T \in \R^{3N}$ with $\vs_\Ga$ given as the nodal values of the surface interpolation.

In order to derive an expression for the defect in the velocity in $\Om$, we translate \eqref{eq:defects - harmonic extension} to the semi-discrete functional analytic setting, and obtain: that $d_v$ satisfies the equation, for the given $v_h^* = \Ih^\Om v$ and for any $\varphi_h^v \in \calV_{h,0}[\bfx]^3$,
\begin{equation*}
a_{\Om_h^*}( \Ih^\Om v , \vphi_h^v) = m_{\Om_h^*}(d_v , \vphi_h^v) .
\end{equation*}

Subtracting the equivalent weak formulation of the harmonic extension \eqref{harmonic-v-weak} from the above equation, setting $\vphi^v = (\vphi_h^v)^\ell \in H_0^1(\Om)^3$, yields
\begin{align*}
m_{\Om_h^*}(d_v , \vphi_h^v) = a_{\Om_h^*}( \Ih^\Om v , \vphi_h^v) - a_{\Om}(v,(\vphi_h^v)^\ell) .
\end{align*}

By similar arguments as above we obtain
\begin{align*}
m_{\Om_h^*}(d_v , \vphi_h^v) = &\ a_{\Om_h^*}( \Ih^\Om v , \vphi_h^v) - a_{\Om}(v,(\vphi_h^v)^\ell) 
\leq c h^{k} \|v\|_{H^{k+1}(\Om)} \|\vphi_h^v\|_{H^1(\Om)} .
\end{align*} 
using the bulk interpolation error estimate \eqref{eq:interpolation errors}.

Altogether we proved the defect bound
\begin{equation*}
\|d_v\|_{H^{-1}_h(\Om_h^*)} = O(h^{k}) .
\end{equation*}


\section{Numerical experiments}
\label{section:numerics}

We performed numerical simulations and experiments for the coupled bulk--surface system \eqref{robin-bvp}--\eqref{ODE} using its matrix--vector formulation \eqref{eq:matrix-vector formulation - full system} and linearly implicit BDF time discretizations, see the next Section~\ref{section:BDF}:
\begin{itemize}
\item We performed convergence experiments involving an example with radially symmetric exact solutions, and report on the convergence rates to illustrate the theoretical results of Theorem~\ref{theorem:main}. The same test example was used in \cite[Section~6.4.1]{EKS19}. Note that the parameter $\alpha$ here equals to $1 / \alpha$ in \citep{EKS19}.
\item We performed some numerical experiments from \cite[Section~6.4.3]{EKS19}: choosing the same parameters $\alpha$ and $\beta$, as well as using the three-di\-men\-sio\-nal analogues of the initial domain.
\item We performed a numerical experiment which compares the numerical solution for various regularization parameters ($\mu = 0$, $0.01$, $0.1$, and $1$).
\end{itemize}

The numerical experiments use quadratic evolving bulk--surface finite elements, for computing the finite element expressions quadratures of sufficiently high order were employed.
The parametrization of the quadratic elements was inspired by \citep{BCH2006}. 
The initial meshes were all generated using DistMesh \citep{distmesh}, without taking advantage of any symmetry of the surface.

\subsection{Linearly implicit full discretization}
\label{section:BDF}

For the time discretization, we use a $q$-step linearly implicit backward difference formula with $q \le 6$. For a step size $\tau > 0$, $t_n = n \tau \le T$, we introduce for $n \ge q$ the discrete time derivative

\begin{equation}
\label{eq:discrete time derivative}
\begin{aligned}
	\dot{\bfu}^n = \frac{1}{\tau} \sum_{j=0}^q \delta_j \bfu^{n-j}
\end{aligned}
\end{equation}
and the extrapolated value
\begin{equation}
\label{eq:extrapolated value}
\begin{aligned}
	\widetilde{\bfu}^n = \sum_{j=0}^{q-1} \gamma_j \bfu^{n-1-j} \,.
\end{aligned}
\end{equation}
The coefficients are given by the relations $\delta(\zeta)=\sum_{j=0}^q \delta_j \zeta^j = \sum_{\ell=1}^q \frac{1}{\ell}(1-\zeta)^\ell$ and $\gamma(\zeta)=\sum_{j=0}^{q-1} \gamma_j \zeta^j = (1-(1-\zeta)^q)/\zeta$, respectively.

We compute approximations $\bfu^n$ to $u(t_n)$, $\bfp^n$ to $p(t_n)$ etc. by the linearly implicit BDF discretization
\begin{equation}
\label{eq:full discretization}
\begin{aligned}
	\left( \bfA_{\bar\Om}(\widetilde\bfx^n) + \bftr\T \big( \mu \bfA_\Ga(\widetilde\bfx^n) + \aalpha \bfM_\Ga(\widetilde\bfx^n) \big) \bftr \right) \bfu^n &= \bff_\bfu(\widetilde\bfx^n,\widetilde\bfH^n) \\
	\bfM_\Ga^{[3]}(\widetilde\bfx^n) \dot{\bfn}^n + \beta \bfA_\Ga^{[3]} (\widetilde\bfx^n) \bfn^n &=  \bff_\nu(\widetilde\bfx^n,\widetilde\bfn^n,\widetilde\bfH^n)  -\aalpha  \bfD_\Ga(\widetilde\bfx^n) \bfu^n ,
	\\
	\bfM_\Ga(\widetilde\bfx^n) \dot{\bfH}^n + \beta \bfA_\Ga(\widetilde\bfx^n) \bfH^n &=  \bff_V(\widetilde\bfx^n,\widetilde\bfn^n,\widetilde\bfH^n,\widetilde\bfu^n) +  \aalpha  \bfA_\Ga(\widetilde\bfx^n) \bfu^n. \\
	\bfV^n &=-\beta \bfH^n + \aalpha \bfu^n \\
	\bfv_\Ga^n &= \bfV^n \bullet \bfn^n \\
	-\bfA_{\Om \Om}(\widetilde\bfx^n) \bfv_\Om^n &= \bfA_{\Om \Ga}(\widetilde\bfx^n)\bfv_\Ga^n \\
	\dot\bfx^n &= \bfv^n \,.
\end{aligned}
\end{equation}

\subsection{Convergence experiments using radially-symmetric solutions}

\subsubsection{Constructing the radially-symmetric solutions in $\R^{m+1}$}

We consider radially-symmetric solutions to the coupled bulk--surface problem \eqref{robin-bvp}--\eqref{ODE} with $\clb = 0$ and a time-independent constant inhomogeneity $Q$. For curves in two dimensions the exact same exact solutions were constructed and used in \cite[Section~6.4.1]{EKS19}. (Note that the parameter $\alpha$ here equals to $1 / \alpha$ in \citep{EKS19}.)
For later reference we performed these computations in $m$-dimensional surfaces in $\R^{m+1}$, and will report on experiments in $\R^3$.

We recall that the surface normal and mean curvature on a sphere $\Ga\t$ of radius $R\t$ are given by
\begin{equation*}
\nu = \frac{x}{R\t} \andquad H = \frac{m}{R\t} .
\end{equation*}

(a) Let us first construct the radially-symmetric solution of the Robin problem \eqref{robin-bvp} with $\clb = 0$ (i.e.~equation~\eqref{robin}).

Recall that in $\R^{m+1}$ the Laplace--Beltrami operator for a radially-symmetric function $u$ is given by
\begin{equation*}
\laplace u = \frac{1}{r^m} \pa_r \Big( r^m \pa_r u \Big) .
\end{equation*}
Therefore, 
\begin{equation*}
u = u(r) = \frac{r^2}{2(m+1)} + c_0 
\end{equation*}
satisfies the bulk equation \eqref{robin-pde}, i.e. $-\laplace u = -1$. 

Now we need to determine the constant $c_0$ such that the Robin boundary condition \eqref{robin-bc} holds.
Recall that in $\R^{m+1}$ the gradient of a radially-symmetric function $u$ is given by
\begin{equation*}
\nb u = \pa_r u \, \hat{\boldsymbol{r}} ,
\end{equation*}
where $\hat{\boldsymbol{r}}$ is the unit vector corresponding to the radial coordinate. We note here that on the sphere  $\Ga\t$ we have $\hat{\boldsymbol{r}} = \nu$.

We now plug in the solution into the boundary condition \eqref{robin-bc} of the Robin problem, and compute on $\Ga\t$
\begin{align*}
Q + \beta H = &\
\Big(\pa_r \Big( \frac{r^2}{2(m+1)} + c_0 \Big) \, \hat{\boldsymbol{r}}\Big)\Big|_{R\t} \cdot \nu + \aalpha \Big( \frac{r^2}{2(m+1)} + c_0 \Big)\Big|_{R\t} \\
= &\ \frac{R\t}{m+1} + \aalpha \Big( \frac{R\t^2}{2(m+1)} + c_0 \Big) .
\end{align*}
Upon plugging in $H = m/R\t$ and rearranging, we obtain
\begin{align*}
c_0 =  \frac{1}{\alpha} \bigg( Q + \beta \frac{m}{R\t} - \frac{R\t}{m+1} \bigg) - \frac{R\t^2}{2(m+1)} ,
\end{align*}
and hence the radially-symmetric solution of the Robin problem is given by
\begin{equation}
\label{eq:exact solution - u}
u(x,t) = u(r,t) = \frac{r^2}{2(m+1)} + \frac{1}{\alpha} \bigg( Q + \beta \frac{m}{R\t} - \frac{R\t}{m+1} \bigg) - \frac{R\t^2}{2(m+1)} .  
\end{equation}

(b) We will now formulate the ODE governing the surface evolution and determine its solution.

We start by recalling that $X(x^0,t) = R\t x_0$ (with $x^0$ on the unit sphere) which satisfies $\dot X = v \circ X$, where 
\begin{equation*}
v = V \nu = \bigl(-\beta H + \alpha u \bigr) \, \nu ,
\end{equation*}
where the normal velocity, by the above computations, is given by
\begin{align*}
V = &\ - \beta \frac{m}{R\t} + \alpha u(R\t,t) 
= Q  - \frac{R\t}{m+1} .
\end{align*}
The ODE for the radially-symmetric surface evolution therefore reduces to 
\begin{align*}
\dot R\t = &\ Q  - \frac{R\t}{m+1} , \\
R(0) = &\ R^0 ,
\end{align*}
which has the solution, for all $t$,
\begin{equation}
\label{eq:exact solution - R}
R\t = (R^0 - (m+1)Q) e^{-t/(m+1)} + (m+1)Q .
\end{equation}

\subsubsection{Convergence experiments}

For the convergence experiment we chose the following initial values and parameters: The initial surface $\Ga^0$ is a two-dimensional sphere of radius $R^0 = 1.5$, the time-independent constant forcing is set to $Q = 1.5$, the initial concentration $u^0(x,0)$ is given by \eqref{eq:exact solution - u} for all $x \in \Ga^0$. That is we are in a situation where a solution exists on $[0,\infty)$ and the exact solutions for $\Ga[X]$ and $u(\cdot,t)$ are given in \eqref{eq:exact solution - R} and \eqref{eq:exact solution - u}, respectively. 
The above choices are the exact same as the one chosen for \cite[Figure~3]{EKS19} (except the computations here are performed in the two-dimensional case), similarly to \citep{EKS19} the parameters $\alpha$ and $\beta$ are varied. The regularization parameter is set to $\mu = 0$.

We started algorithm \eqref{eq:full discretization} from the nodal interpolations of the exact initial values $\Ga[X(\cdot,t_i)]$, $\nu(\cdot,t_i)$, $H(\cdot,t_i)$, and $u(\cdot,t_i)$, for $i=0,\dotsc,q-1$.
In order to illustrate the convergence results of Theorem~\ref{theorem:main}, we have computed the errors between the numerical solution \eqref{eq:full discretization} and (the nodal interpolation of the) exact solutions of the coupled bulk--surface problem \eqref{robin-bvp}--\eqref{ODE} for the above radially-symmetric solution in dimension $m=2$.  

In Figure~\ref{fig:conv_spacetime_alpha1_beta1} we report on the errors between the numerical solution and the interpolation of the exact solution until the final time $T=1$, for a sequence of meshes (see plots) and for a sequence of time steps $\tau_{k+1} = \tau_k / 2$. 
The logarithmic plots report on the $L^\infty(H^1)$ norm of the errors against the mesh width $h$ in Figure~\ref{fig:conv_spacetime_alpha1_beta1} top, and against the time step size $\tau$ in Figure~\ref{fig:conv_spacetime_alpha1_beta1} bottom.
The lines marked with different symbols and different colours correspond to different time step sizes and to different mesh refinements on the top and bottom, respectively.

In both plots in Figure~\ref{fig:conv_spacetime_alpha1_beta1} we can observe two regions: on the top, a region where the spatial discretization error dominates, matching the $O(h^2)$ order of convergence of Theorem~\ref{theorem:main} (see the reference lines), and a region, with small mesh size, where the temporal discretization error dominates (the error curves flatten out). For the graphs in the bottom, the same description applies, but with reversed roles. Convergence of fully discrete methods is not proved here, but $O(\tau^2)$ is expected for the 2-step BDF method, cf.~\citep{MCF}.


The convergence in time and in space as shown by Figure~\ref{fig:conv_spacetime_alpha1_beta1} 
are in agreement with the theoretical convergence results (note the reference lines). 
We obtained similar convergence plots for varying $\alpha$ and $\beta$ parameters. 

\begin{figure}[htbp]
\includegraphics[width=\textwidth]{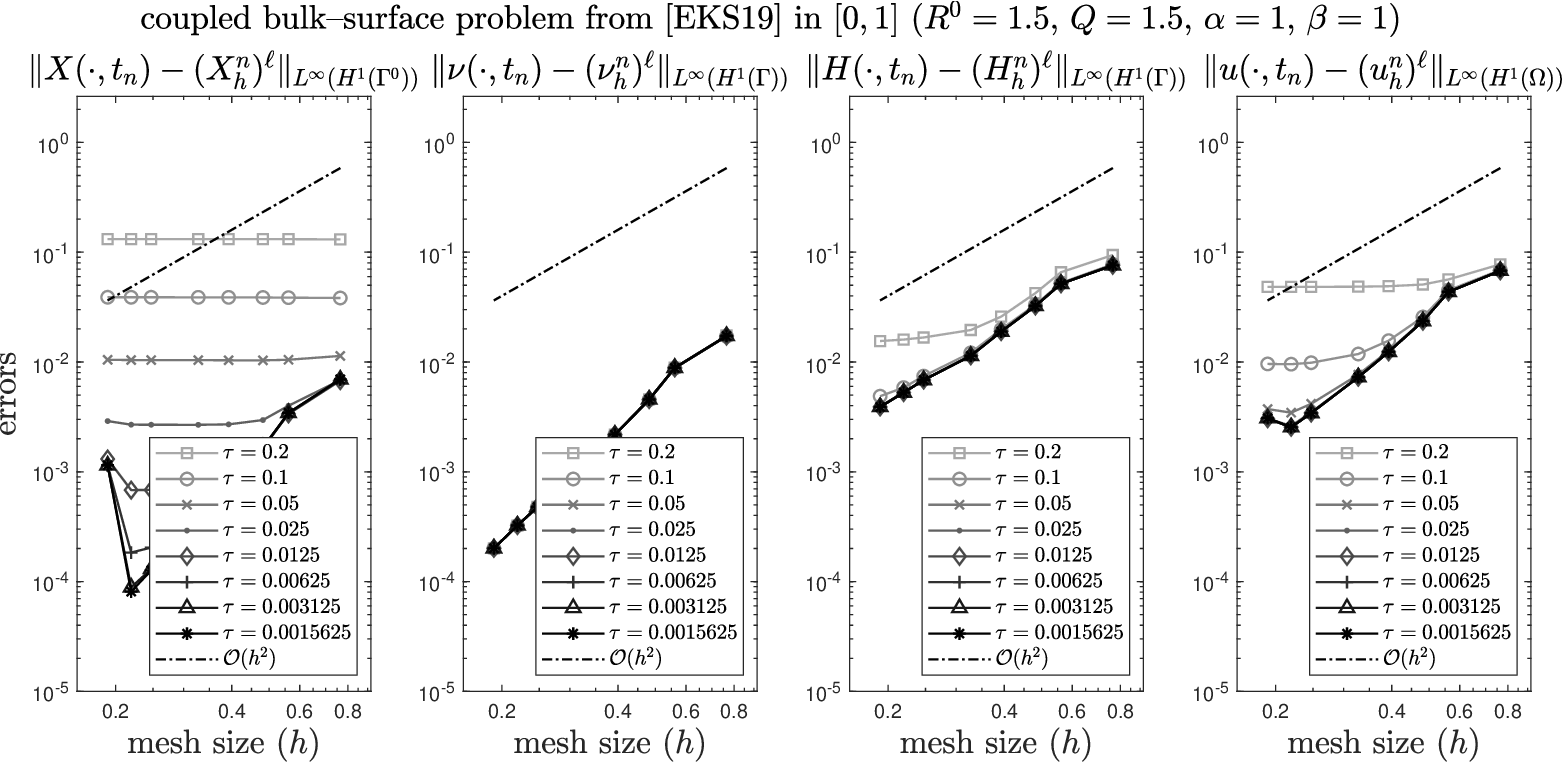}

\smallskip
\includegraphics[width=\textwidth,clip,trim={0 0 0 20}]{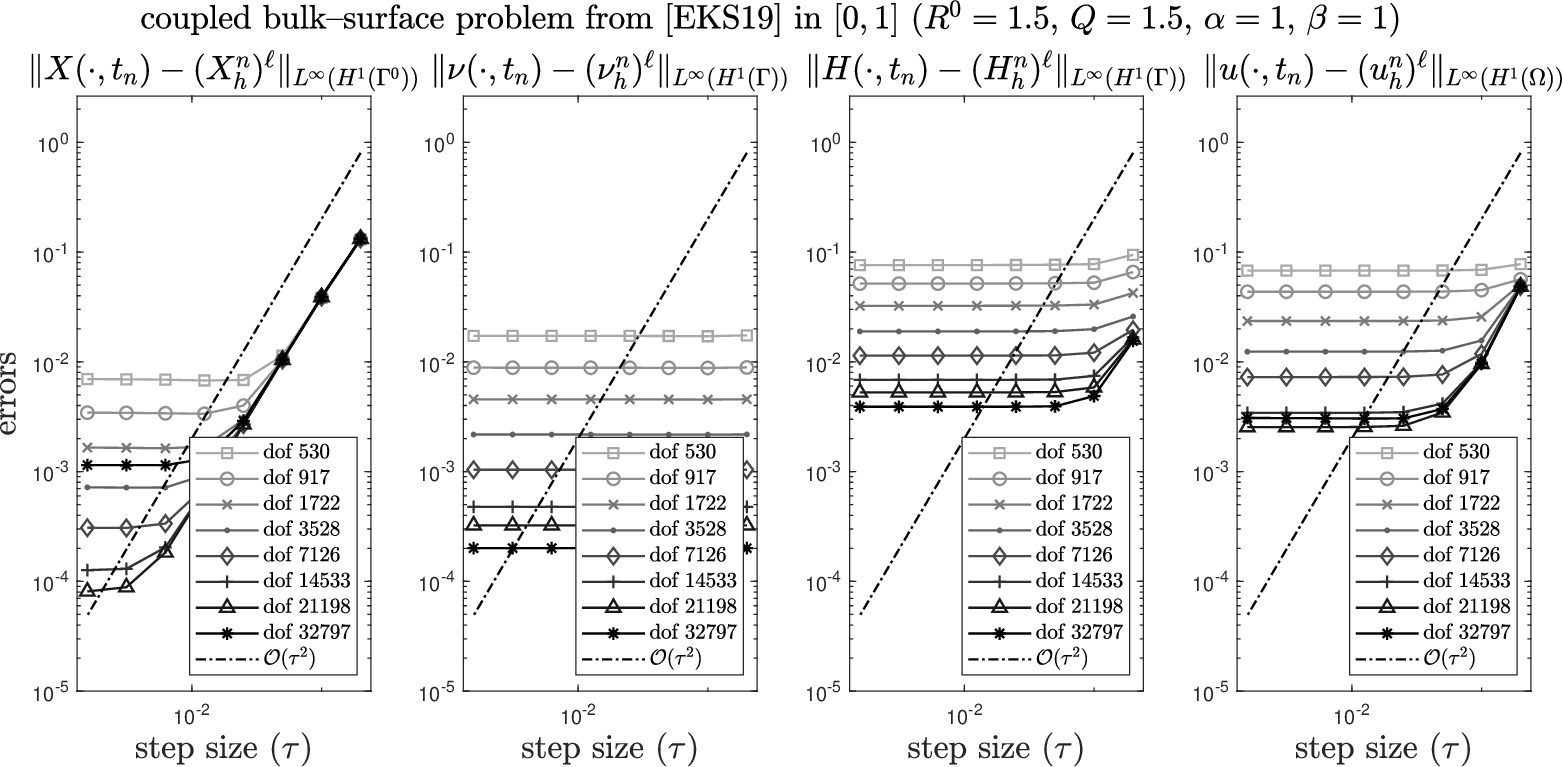}
\caption{Spatial and temporal convergence of the BDF2 / quadratic ESFEM discretization of the coupled bulk--surface problem with $\alpha = 1$ and $\beta = 1$.}
\label{fig:conv_spacetime_alpha1_beta1}
\end{figure}

%
%
%

\subsection{Experiments on $\alpha$ and $\beta$ dependence}

As in \cite[Section~6.4.3]{EKS19} as an initial domain $\Om \subset \R^3$ we choose the ellipsoid with radii $0.5$, $0.5$, and $1$, i.e.
\begin{equation*}
\Om = \Big\{ x \in \R^3 \mid \sqrt{ \tfrac{x_1^2}{0.5^2} + \tfrac{x_2^2}{0.5^2} + x_3^2 } \leq 1 \Big\} , \andquad \Ga = \pa\Om .
\end{equation*}
The initial domain and surface is approximated with quadratic bulk--surface finite elements with $7162$ and $1430$ degrees-of-freedom in the bulk and on the boundary, respectively.
The time step size is set to $\tau = 10^{-3}$. 
The model parameters are set to $Q = 1.5$ and $\mu = 0$, while $\alpha$ and $\beta$ are varied for different experiments. 

The initial data for the geometry ($\nu_h^0$ and $H_h^0$) were obtained by interpolation, while for the Robin problem $u_h^0$ solves the discretized Robin problem \eqref{robin-h}.

In Figure~\ref{fig:EKS_alpha01_beta1} and \ref{fig:EKS_alpha1_beta01} we report on the numerical solution at different times with $\alpha = 10$ and $\beta = 1$ and $\alpha = 1$ and $\beta = 0.1$, respectively. For the sake of a direct comparison we are using the identical values as \cite[Figure~5]{EKS19} and \cite[Figure~6]{EKS19}, respectively. The two experiments are performed in different dimensions. (Note that the parameter $\alpha$ here equals to $1 / \alpha$ in \citep{EKS19}.)

\begin{figure}[htbp]
\centering
\includegraphics[width=0.7\textwidth,clip,trim={0 50 0 18}]{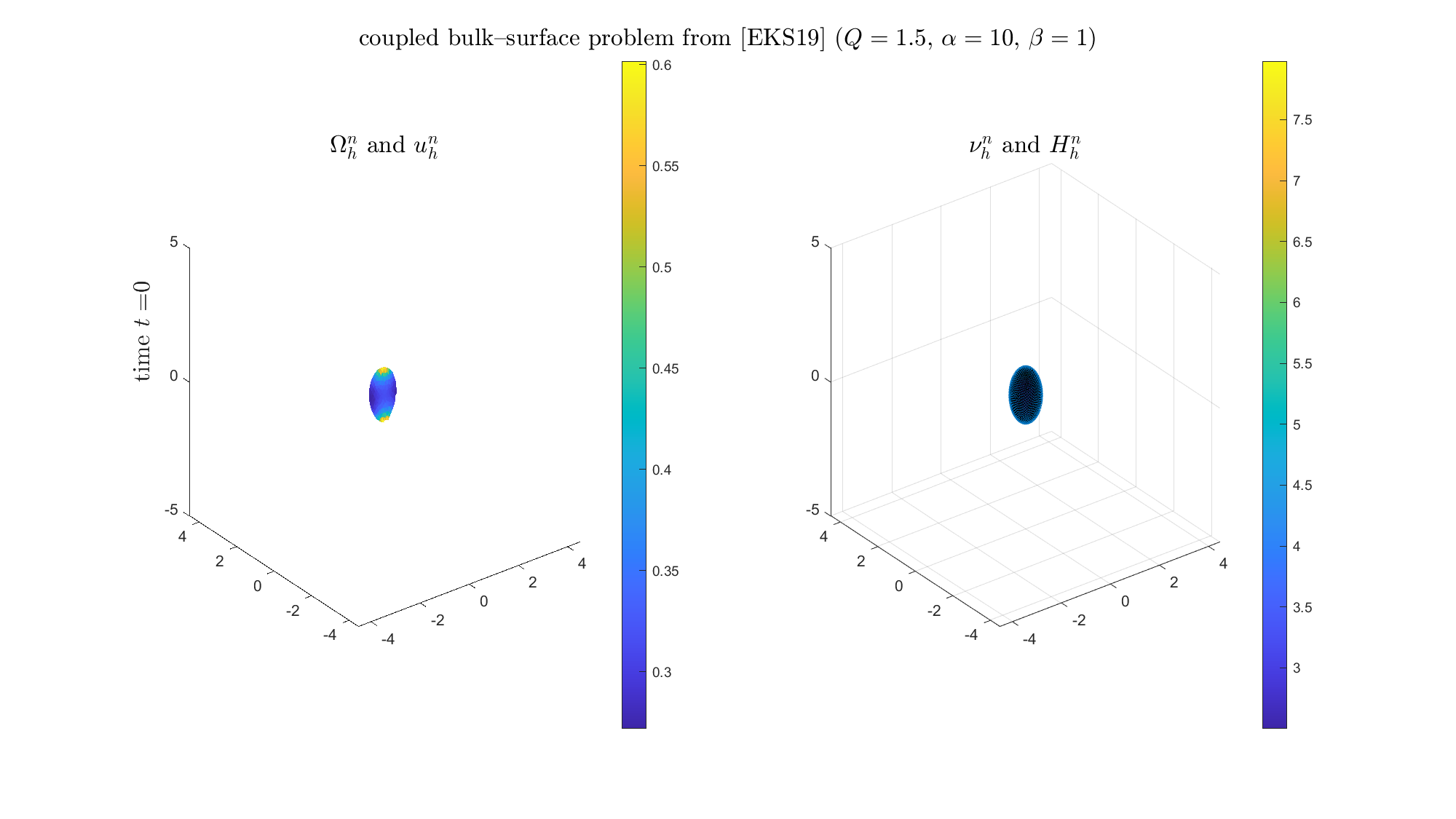}
\includegraphics[width=0.7\textwidth,clip,trim={0 50 0 40}]{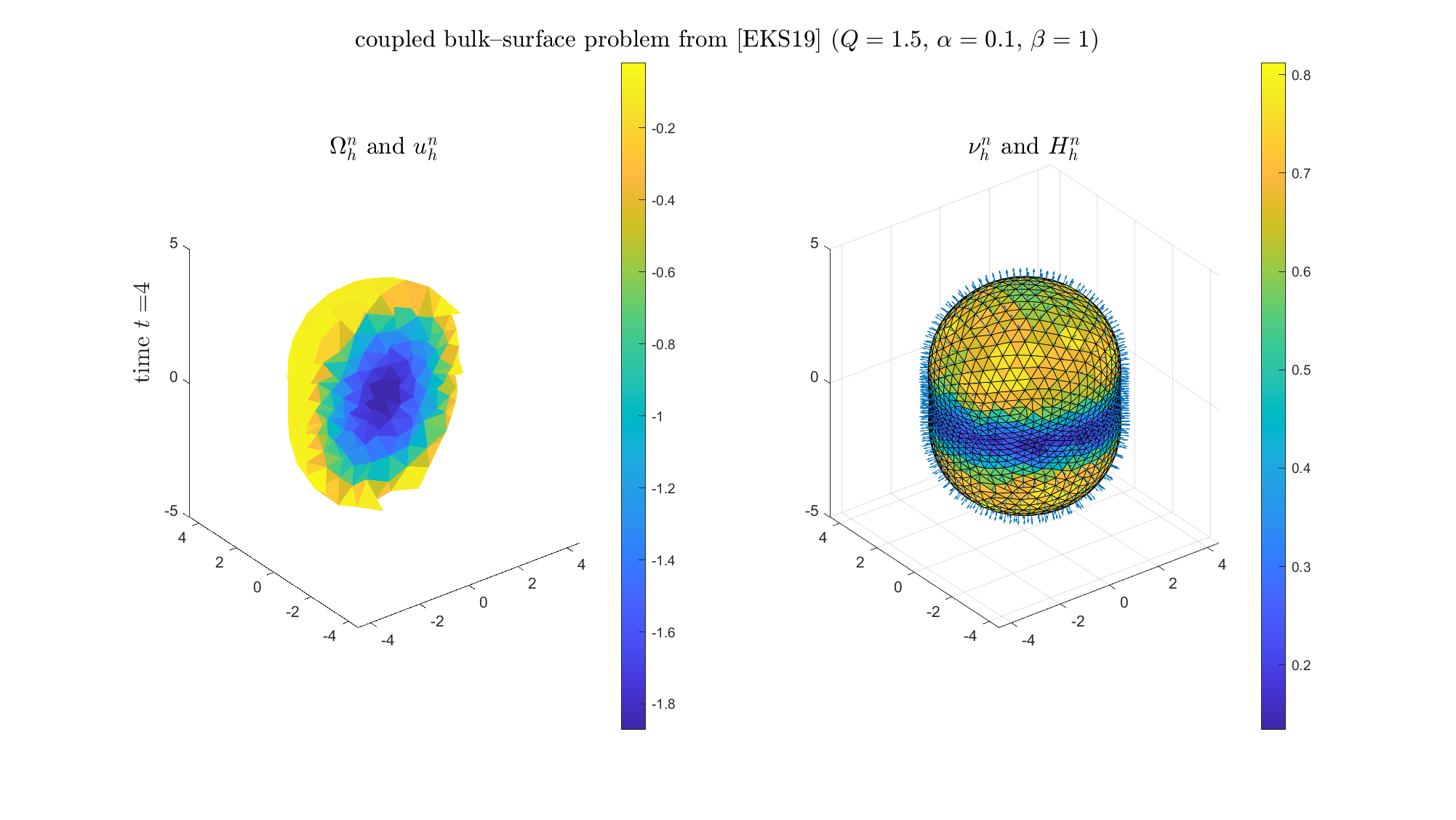}
\includegraphics[width=0.7\textwidth,clip,trim={0 50 0 40}]{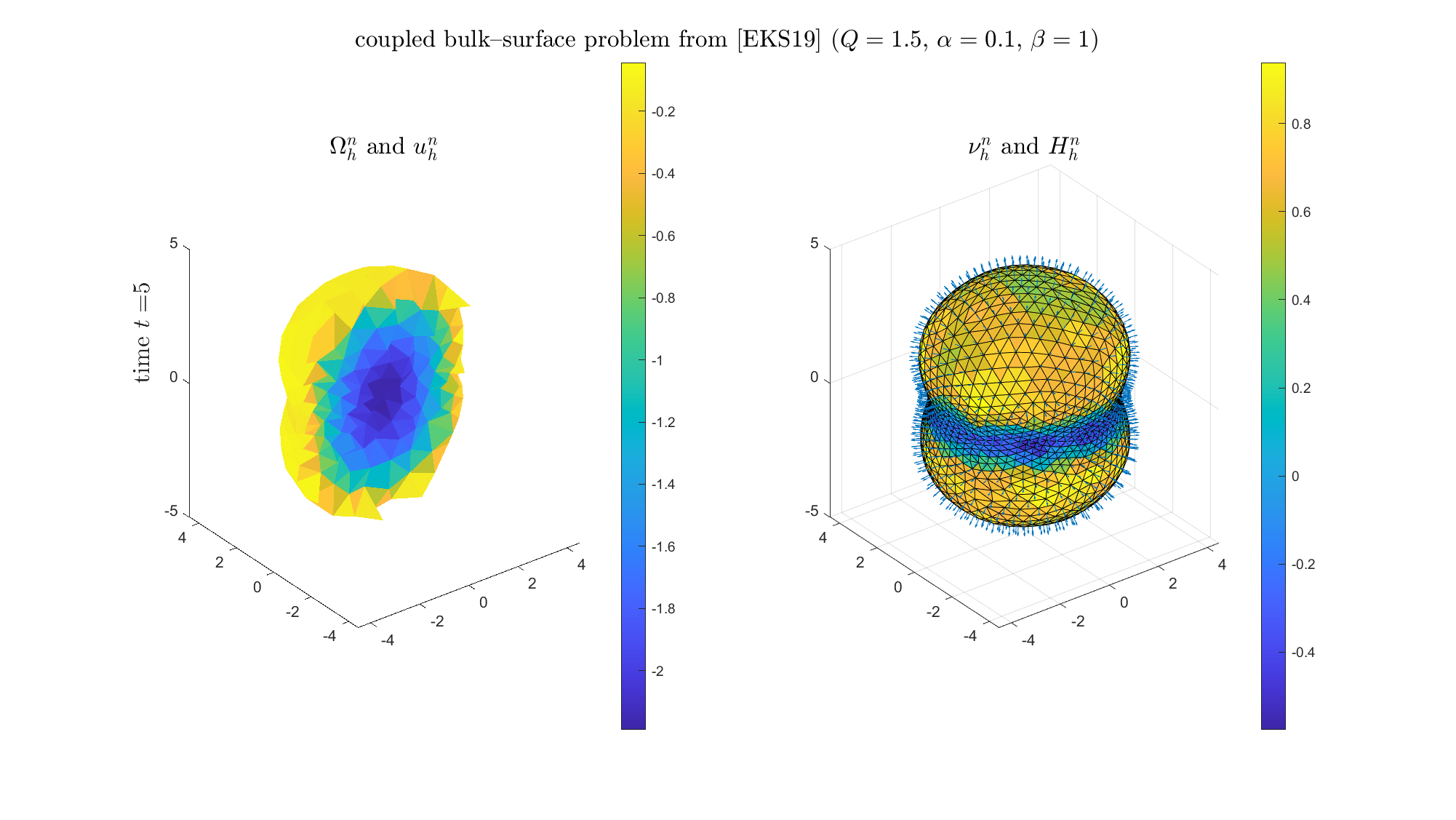}
\includegraphics[width=0.7\textwidth,clip,trim={0 50 0 40}]{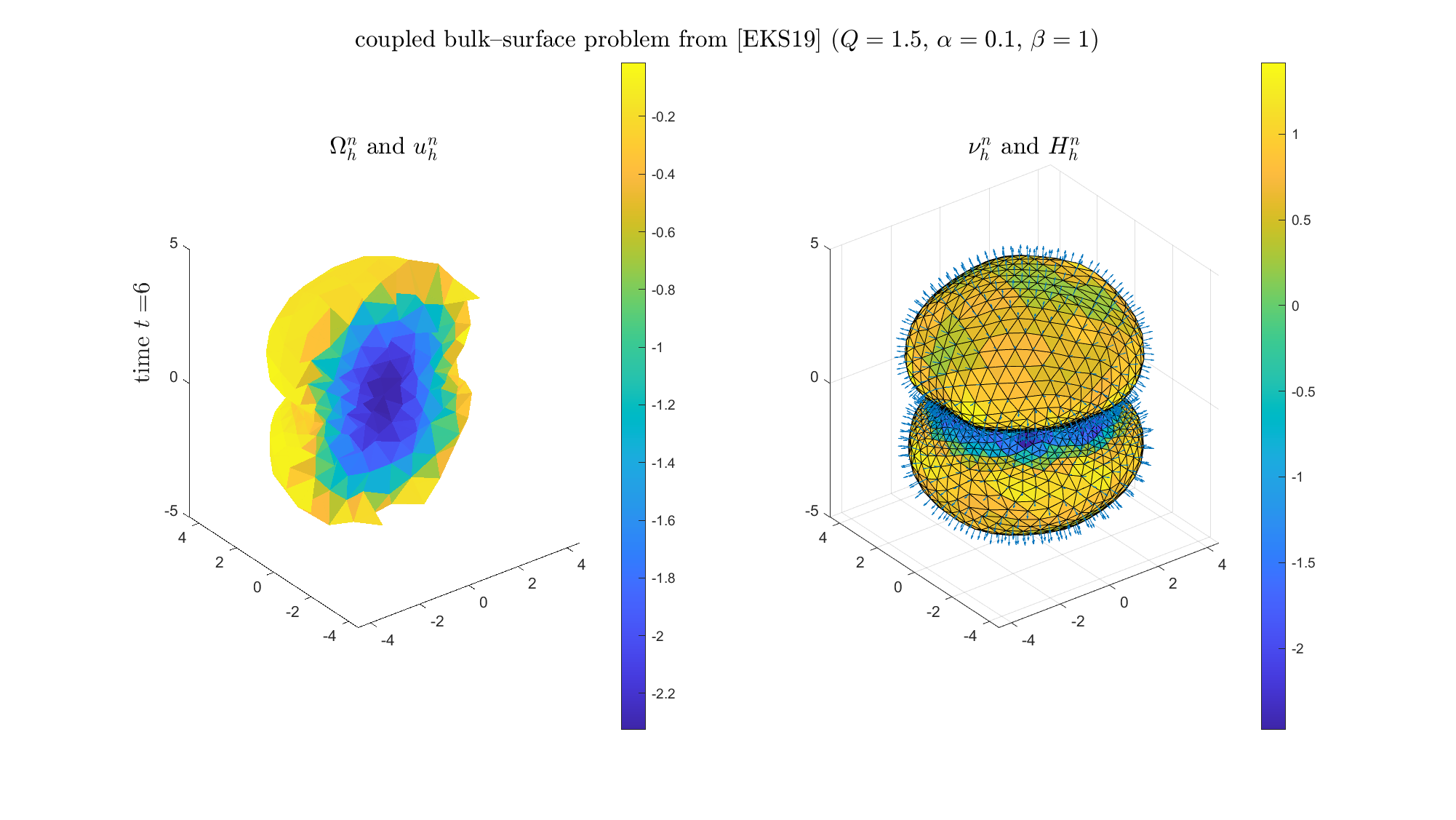}
\caption{Numerical solution of the coupled bulk--surface problem with $\alpha = 10$ and $\beta = 1$.}
\label{fig:EKS_alpha01_beta1}
\end{figure}

\begin{figure}[htbp]
\centering
\includegraphics[width=0.7\textwidth,clip,trim={0 50 0 20}]{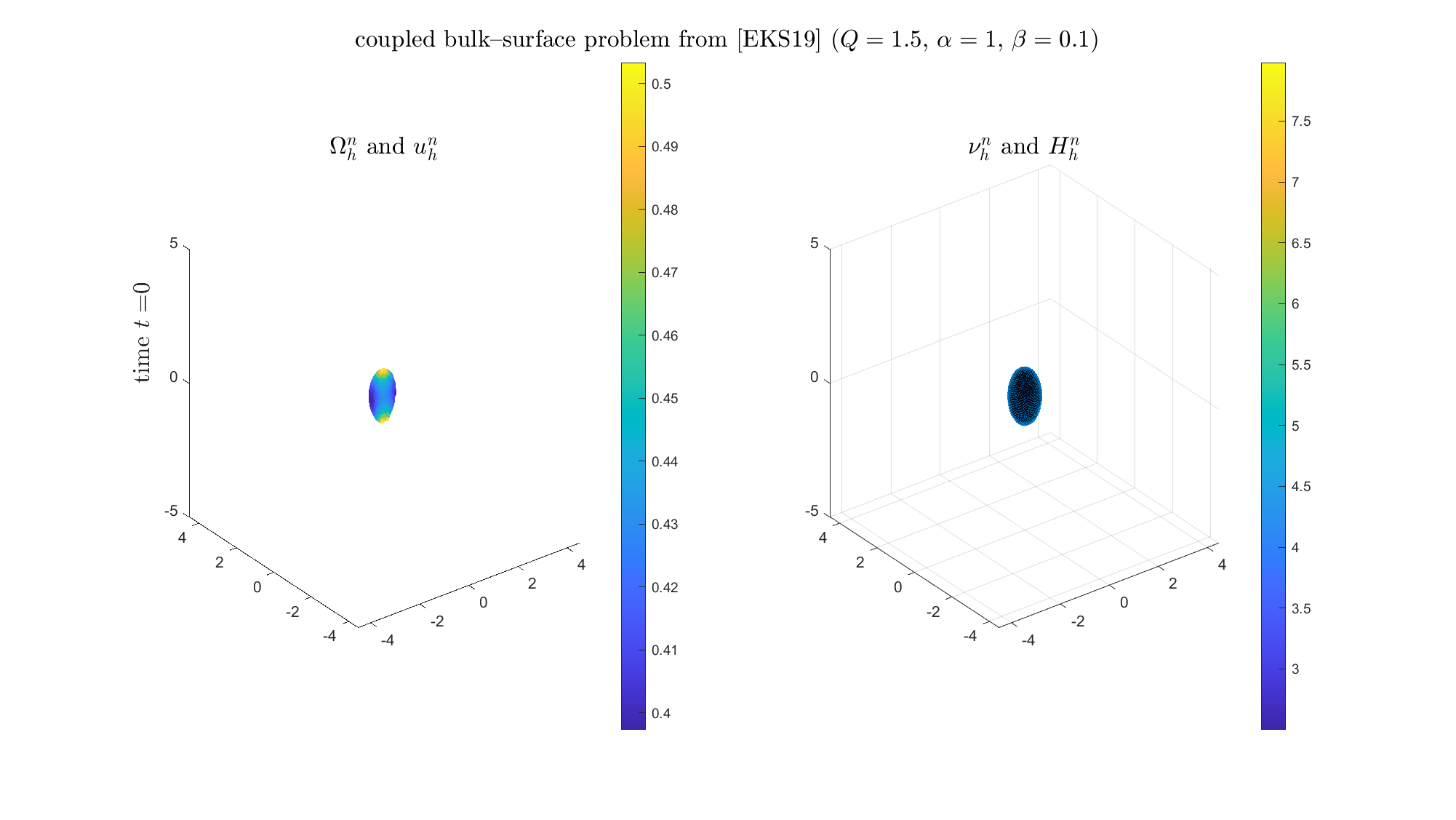}
\includegraphics[width=0.7\textwidth,clip,trim={0 50 0 40}]{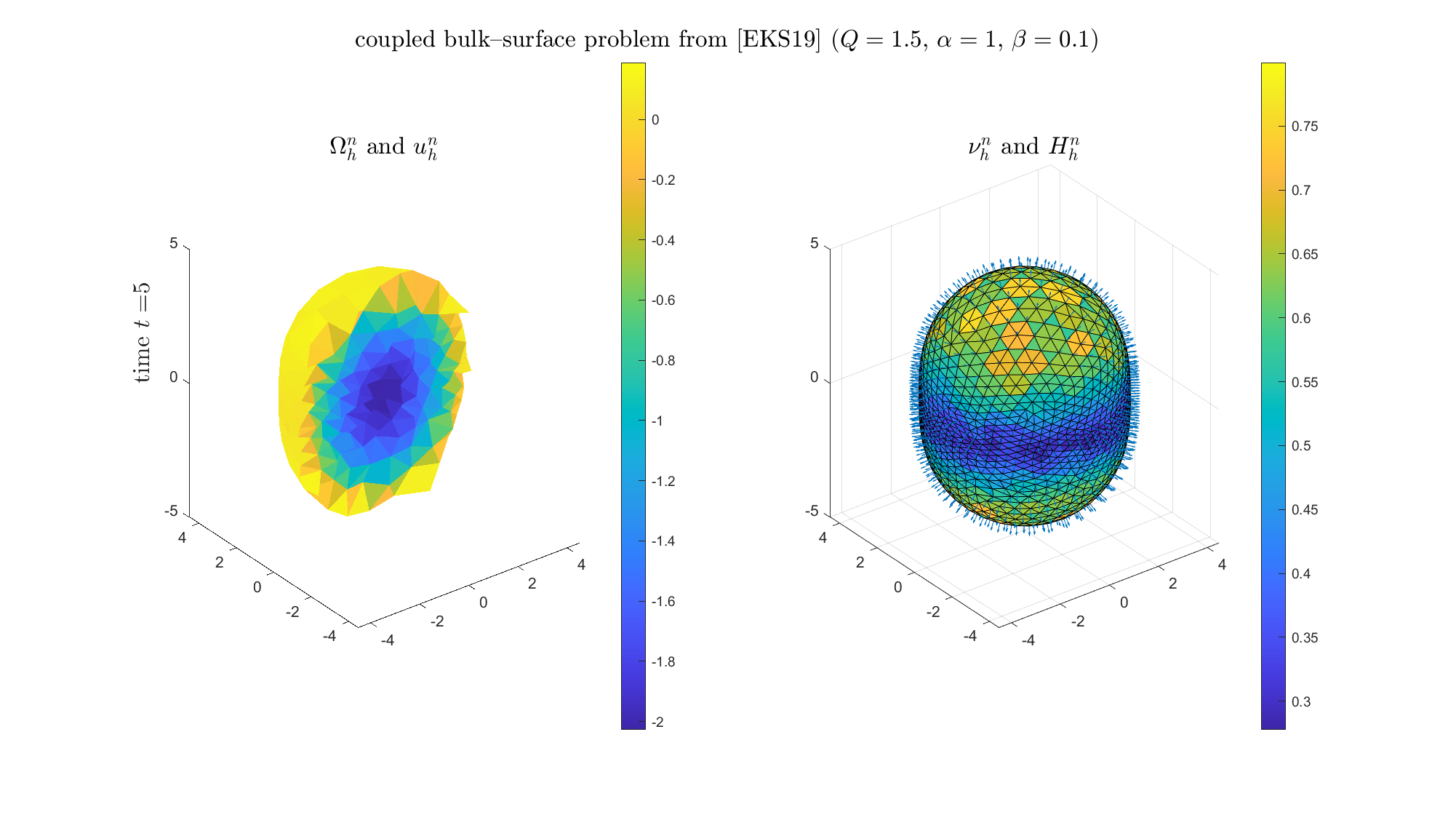}
\includegraphics[width=0.7\textwidth,clip,trim={0 50 0 40}]{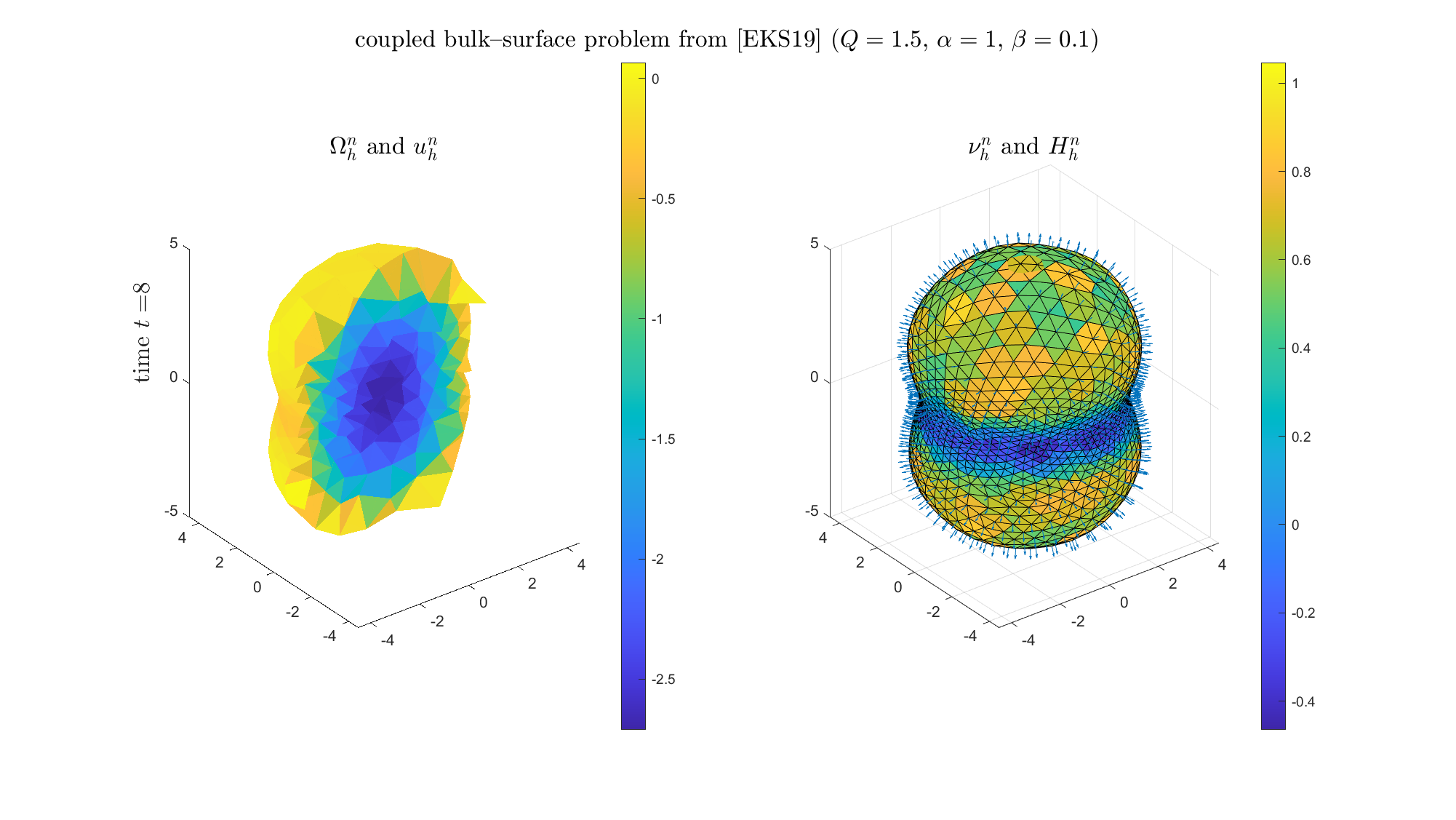}
\includegraphics[width=0.7\textwidth,clip,trim={0 50 0 40}]{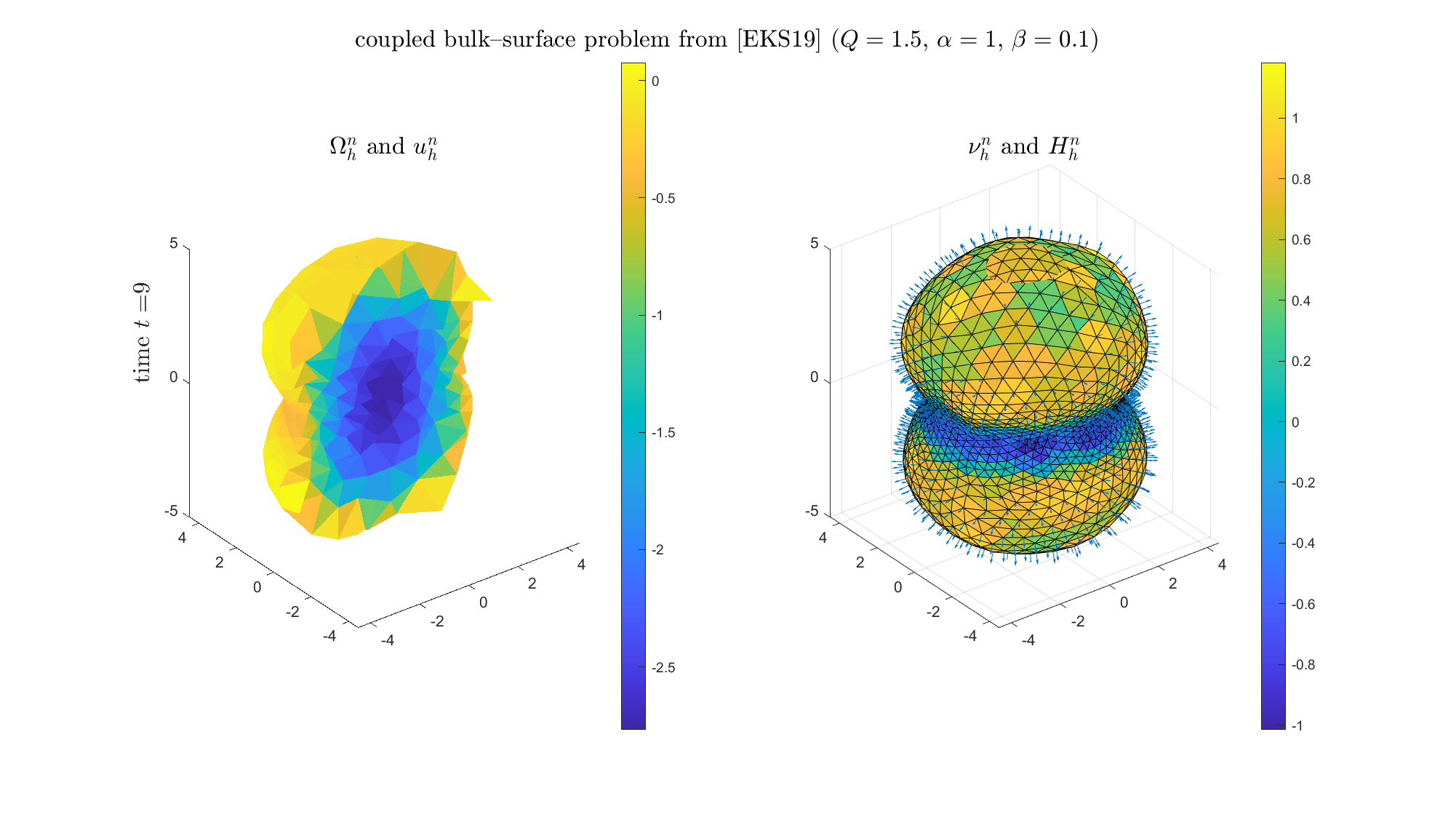}
\caption{Numerical solution of the coupled bulk--surface problem with $\alpha = 1$ and $\beta = 0.1$.}
\label{fig:EKS_alpha1_beta01}
\end{figure}

\subsection{Experiments on the effect of the regularization}

We performed a numerical experiment reporting on the effect of the regularization  (cf.~Section~\ref{section:regularization}). 

The three columns of Figure~\ref{fig:EKS_reg} show the numerical solution with varying regularization parameter $\mu = 0$, $0.01$, $0.1$, $1$ (from left to right) with a fixed source $Q = 1.5$, and fixed model parameters $\alpha = 1$ and $\beta = 1$. 
The initial data were generated exactly as before. The bulk and surface mesh has $7162$ and $1430$ degrees-of-freedom, respectively, while the time step size is set to $\tau = 10^{-3}$.
As observable in Figure~\ref{fig:EKS_reg}, the regularization (even with $\mu = 1$) has minimal effects both on the bulk--surface mesh evolution and the surface variable $u_h^n$.

\begin{figure}[htbp]
\includegraphics[width=\textwidth,clip,trim={0 140 0 110}]{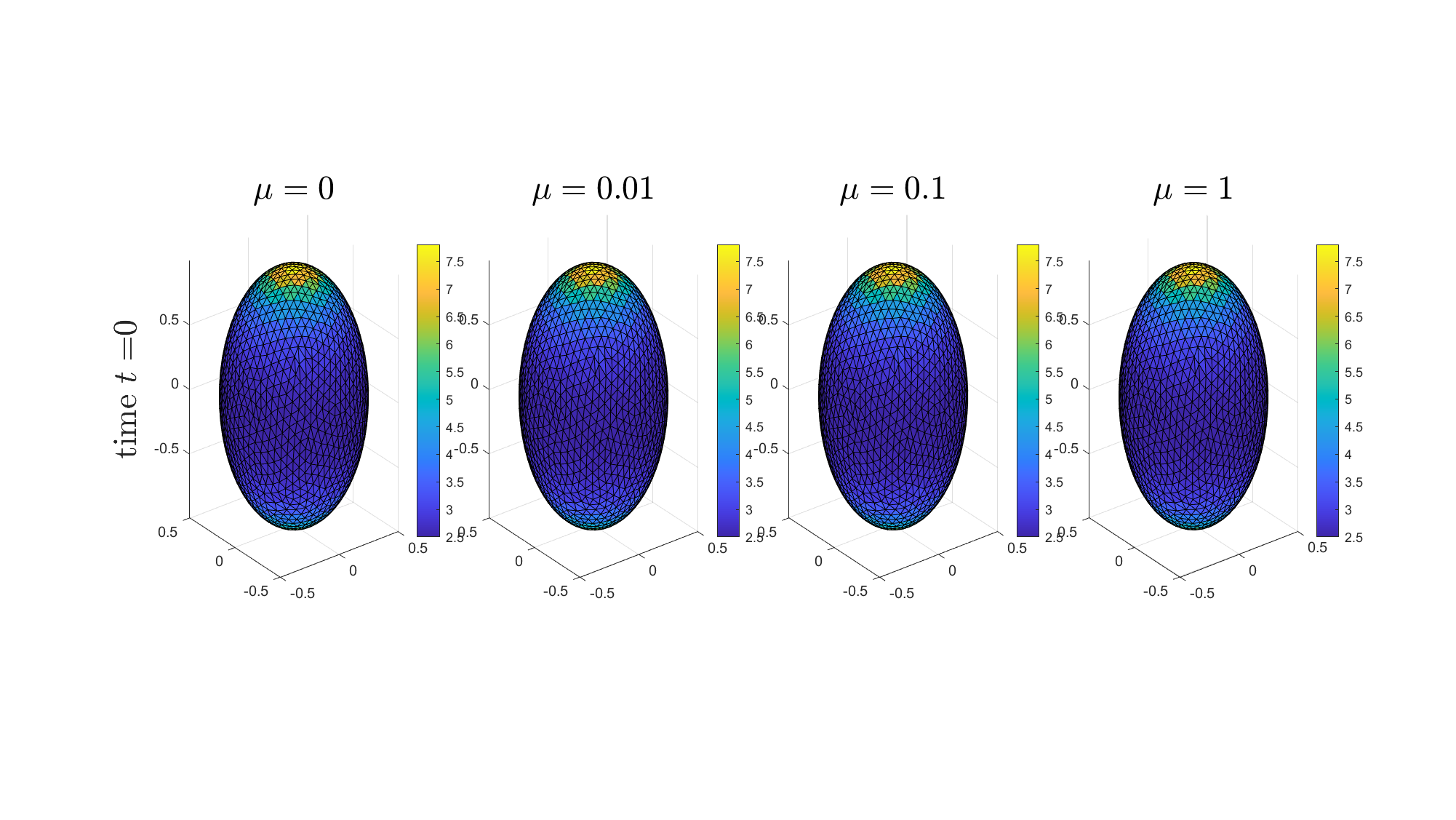}
\includegraphics[width=\textwidth,clip,trim={4 170 18 150}]{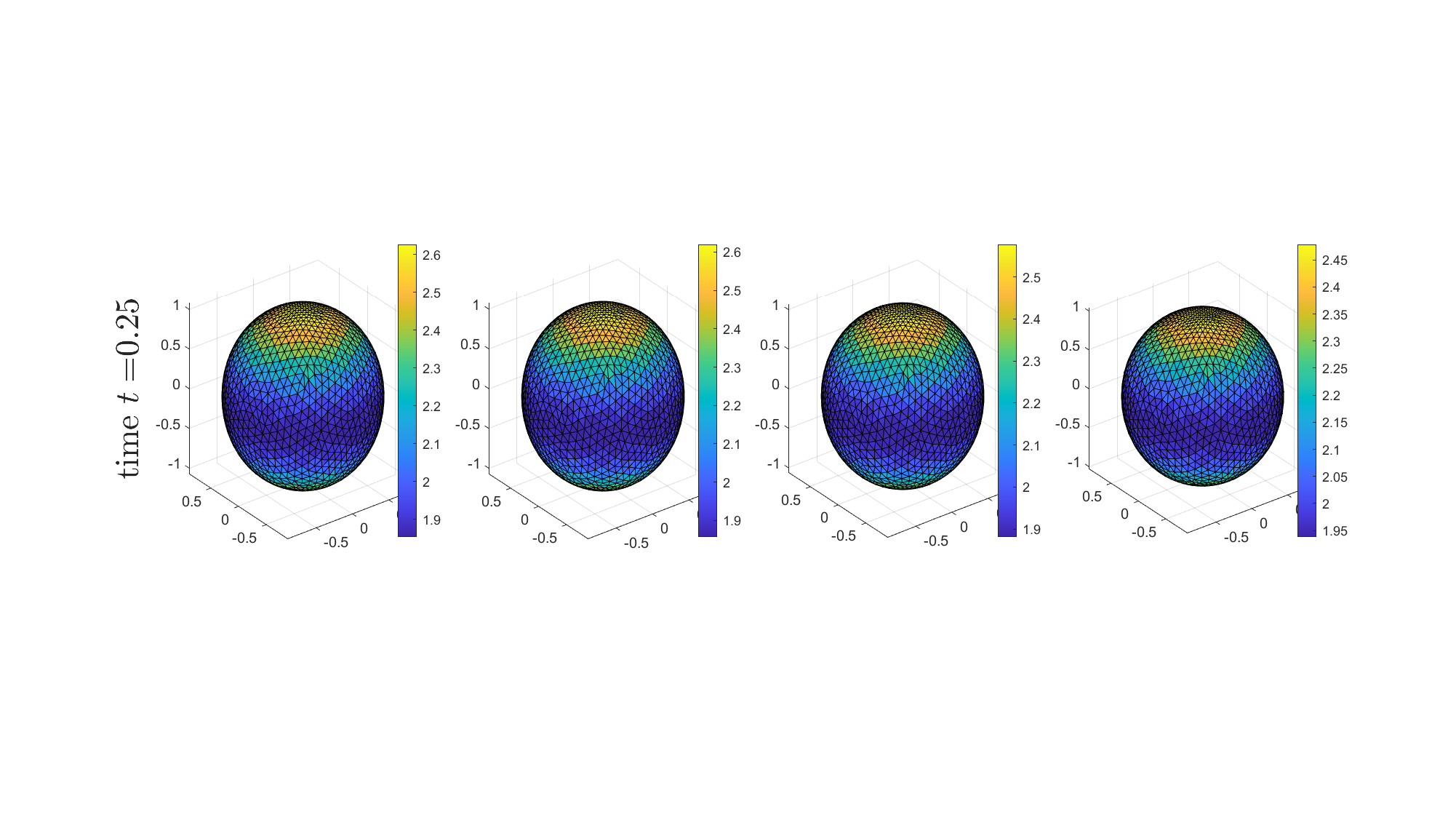}
\includegraphics[width=\textwidth,clip,trim={0 170 0 150}]{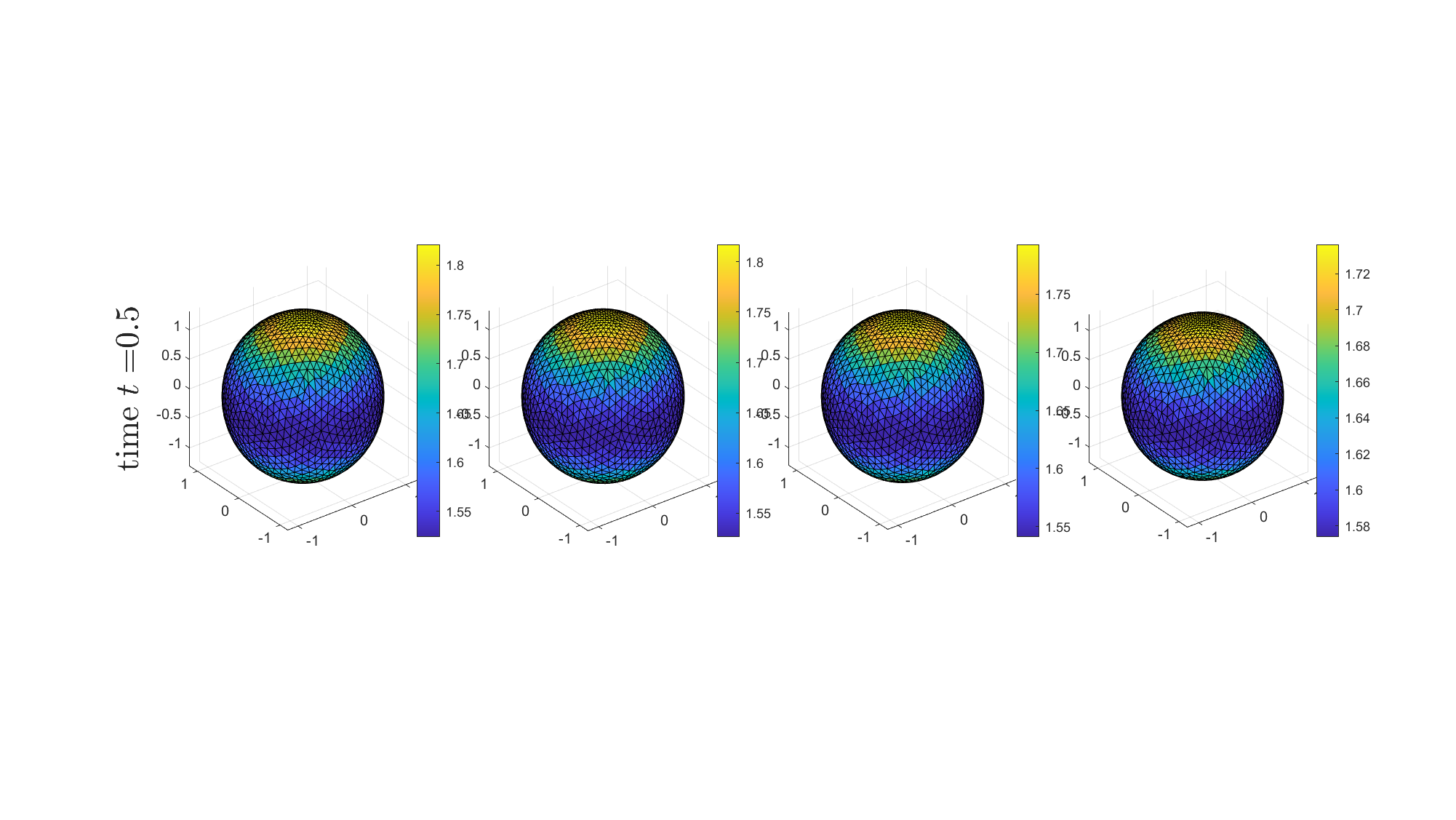}
\includegraphics[width=\textwidth,clip,trim={0 170 0 150}]{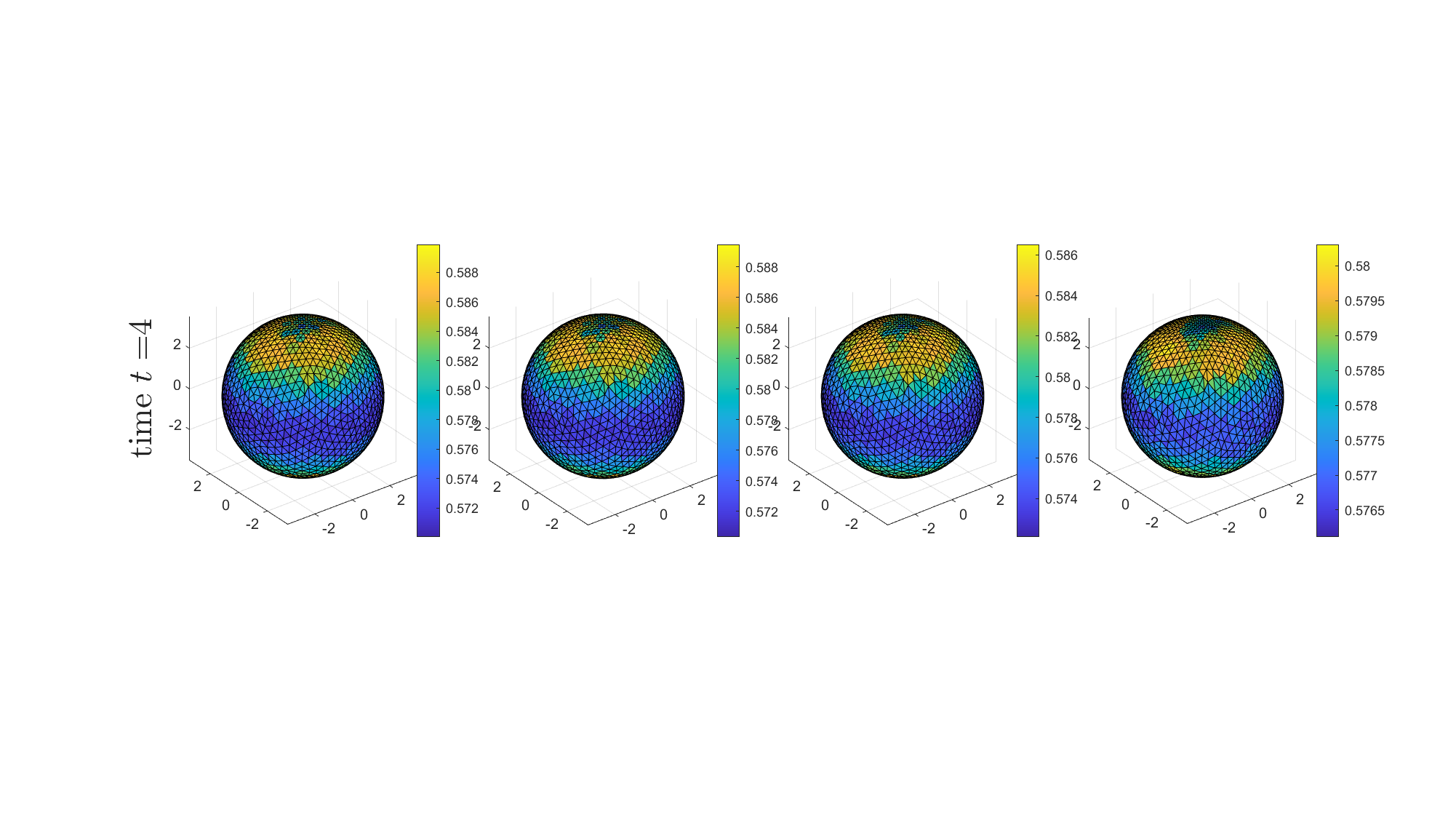}
\includegraphics[width=\textwidth,clip,trim={0 170 0 150}]{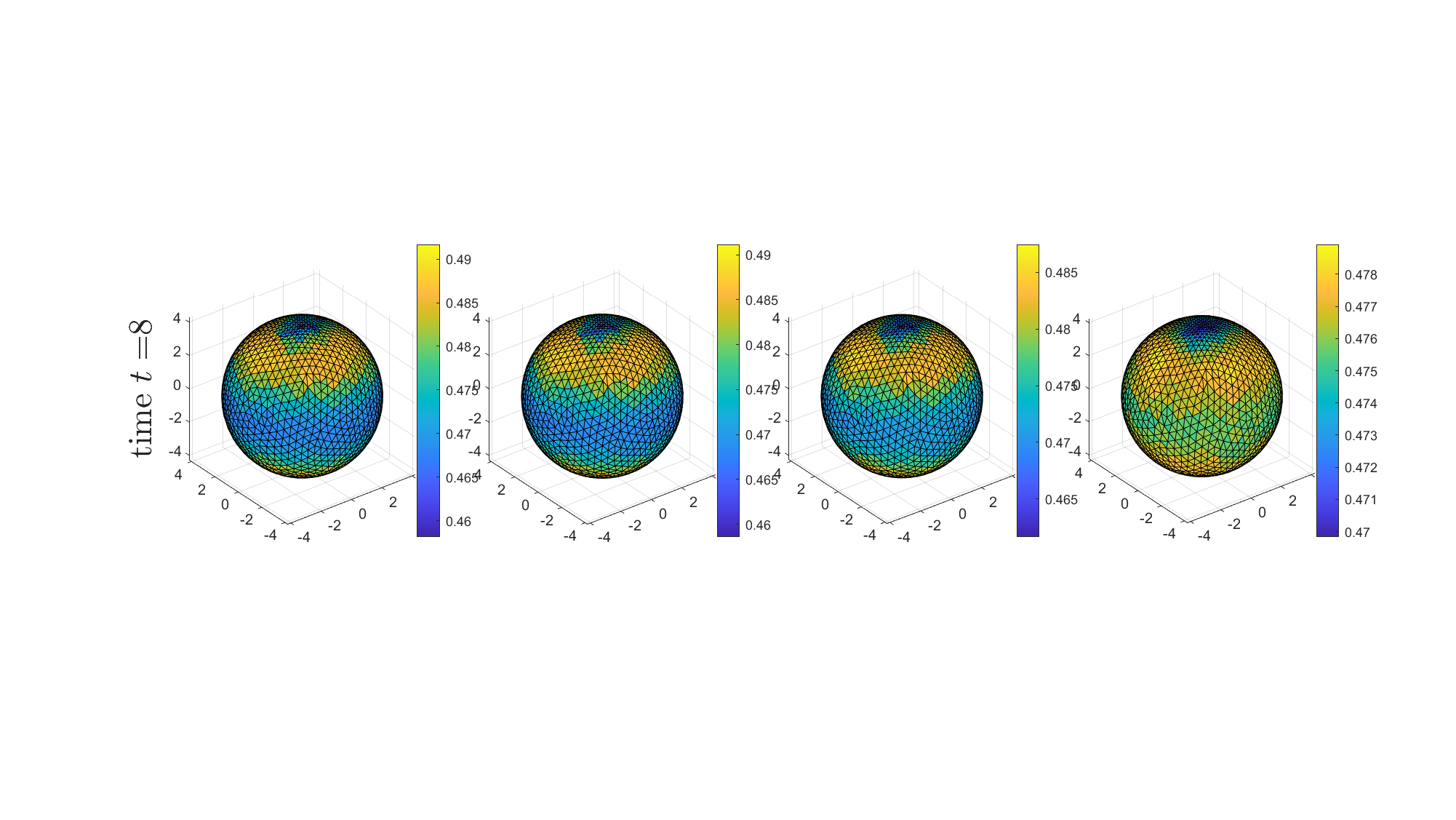}
\caption{Numerical solution ($\Ga_h^n$ and $u_h^n$) of the coupled bulk--surface problem with varying regularization parameter $\mu = 0$, $0.01$, $0.1$, $1$ (column-wise from left to right) $\alpha = 1$ and $\beta = 1$.}
\label{fig:EKS_reg}
\end{figure}



\section{Appendix: Stability bounds of finite element approximations to the Laplace equation with inhomogeneous boundary conditions}
\label{section:laplace} 

We give regularity estimates of finite element approximations to the Laplace equation with inhomogeneous Dirichlet and Robin boundary conditions. Such bounds play an important role in the stability analysis of the numerical method for the tumour growth model and are also of independent interest. As we are not aware of these finite element stability results in the literature, we present them here together with their proofs.

\subsection{Dirichlet problem}

Consider the Dirichlet problem 
\begin{subequations}
\begin{align}
	\label{dir-pde}
	-\Delta u &=0 \quad\text{ in }\varOmega, 
	\\
	\label{dir-bc}
	u&=g \quad\text{ on } \Ga .
\end{align}
\end{subequations}
It follows directly from the weak formulation and the trace theorem that
\begin{equation}\label{u-g-bound}
\| u \|_{H^{1}(\varOmega)} \le C\, \| g \|_{H^{1/2}(\Ga)}.
\end{equation}
Consider now a quasi-uniform finite element discretization of the Dirichlet problem, with boundary data given by the finite element function $g_h$ on the boundary $\Ga_h$ of the 
computational domain $\varOmega_h$ obtained by quasi-uniform finite element interpolation of the boundary $\Ga$ of $\varOmega$. There we have the bulk finite element space $\calV_h$ and
the boundary finite element space $\calS_h$, which is the trace space of $\calV_h$: for every $v_h\in \calV_h$,
its trace $\gamma_h v_h$ is in $\calS_h$. We let $\calV_h^0$ be the subspace of $\calV_h$ with trace $0$.
For a given $g_h\in \calS_h$, the finite element approximation $u_h\in \calV_h$ is determined by
$$
\int_{\Om_h} \nabla u_h \cdot \nabla \varphi_h =0 \quad \text{ for all }\, \varphi_h \in \calV_h^0, \qquad \gamma_h u_h= g_h.
$$
There is a discrete bound analogous to \eqref{u-g-bound}.

\begin{proposition} \label{prop:Dirichlet-h} In the above situation, the finite element approximation $u_h$ obtained from a  quasi-uniform and shape-regular family of triangulations is bounded by
\begin{equation}\label{uh-gh-bound}
	\qquad \| u_h \|_{H^1(\varOmega_h)} \le C\, \| g_h \|_{H^{1/2}(\Ga_h)},
\end{equation}
where $C$ depends on bounds of finitely many derivatives of a parametrization of the boundary surface $\Ga$ and on the quasi-uniformity and shape regularity bounds, but is independent of $h$ and $g_h\in \calS_h \subset H^{1/2}(\Ga_h)$.
\end{proposition}
For linear finite elements on a two-dimensional polygonal domain $\varOmega_h$, such a bound was proved by Bramble, Pasciak \& Schatz \citep{BrPS86}, Lemma 3.2, and for linear finite elements on smooth domains, with a polygonal computational domain $\varOmega_h$, this was proved by Bramble \& King \citep{BrK94} (as an immediate corollary to their Theorem~1). We are not aware of a proof in the literature of the more general result for quadratic or higher-degree finite elements with non-polygonal computational domain, nor of the result in the three-dimensional case.

\medskip\noindent
\begin{proof} 
Let $\widetilde{I}_h^{\textnormal{SZ}} \colon H^1(\varOmega_h)\to \calV_h$ be the isoparametric Scott--Zhang finite element interpolation operator, see \cite[Section~4.4]{HansboLarsonLarsson}, and see \citep{ScottZhang} for the simpler case of polygonal domains, as well as \citep{CamachoDemlow} for the case of smooth closed surfaces.
The operator $\widetilde{I}_h^{\textnormal{SZ}}$ is a stable projection that preserves nodal values, in particular boundary values in $\calS_h$, i.e.~$\trh \widetilde{I}_h^{\textnormal{SZ}} v_h= \trh v_h$ for all $v_h\in \calV_h$.

Let $E \colon H^{1/2}(\Ga) \to H^1(\varOmega)$ be the solution operator of the Dirichlet problem: $u=Eg$, with $\|E\|_{H^1(\Om) \leftarrow H^{1/2}(\Ga)} \leq C$ by \eqref{u-g-bound}.  
We consider a discrete harmonic extension operator $E_h \colon  S_h(\Ga_h) \to S_h(\varOmega_h)$ given by
$$
E_h g_h = \widetilde{I}_h^{\textnormal{SZ}} (E g_h^\ell)^{-\ell}.
$$
Note that $\gamma_h E_h g_h = g_h$.
The finite element approximation $u_h$ is defined equivalently by
$u_h= E_h g_h +w_h$, where $w_h\in \calV_h^0$ satisfies
$$
\int_{\Om_h} \nabla w_h \cdot \nabla \varphi_h = -\int_{\Om_h} \nabla (E_h g_h) \cdot \nabla \varphi_h\qquad \text{for all } \, \varphi_h \in \calV_h^0.
$$
With $\varphi_h=w_h$, we obtain immediately
$$
\| w_h \|_{H^1_0(\varOmega_h)} \le \|  E_h g_h \|_{H^1(\varOmega_h)} =:\alpha,
$$
and so we have $\| u_h \|_{H^1(\varOmega_h)} \le 2\alpha$.
We estimate (with different constants $C$)  
\begin{align*}
	\alpha &= \|  \widetilde{I}_h^{\textnormal{SZ}} (E g_h^\ell)^{-\ell} \|_{H^1(\varOmega_h)} \le C \, \|  (E g_h^\ell)^{-\ell} \|_{H^1(\varOmega_h)} 
\end{align*} 
by the stability of the isoparametric Scott--Zhang interpolation, 
see \cite[Eq.~(4.45)]{HansboLarsonLarsson} with $m=s=1$;
$$
\le \|  E g_h^\ell \|_{H^1(\varOmega)}
$$
by the $H^1$ norm equivalence of lifted functions;
$$
\le C\, \|    g_h^\ell \|_{H^{1/2}(\Ga)} 
$$
by the bound \eqref{u-g-bound};  
$$
\le C\, \|    g_h \|_{H^{1/2}(\Ga_h)} 
$$
by the $H^{1/2}$ norm equivalence of the lift. 
This proves \eqref{uh-gh-bound}.
\end{proof}

\subsection{Robin boundary value problem}

Consider the Robin boundary value problem 
\begin{subequations}
\label{robin-aux}
\begin{align}
	\label{robin-aux-pde}
	- \Delta u &= 0 \quad\text{ in }\Om \,,\\
	\label{robin-aux-bc}
	\dnu{u} +  u   &= g \quad\text{ on }\Ga ,
\end{align}
\end{subequations}
with the weak formulation to find $u\in H^1(\Om)$ such that
\begin{equation} \label{robin-weak-aux}
\int_{\Om} \nb u \cdot \nb \varphi +  \int_{\Ga} \tr u \;\tr \varphi =  \int_{\Ga} g \;\tr \varphi \quad\ \text{ for all $\varphi\in H^1(\varOmega)$}.
\end{equation}
We are particularly interested in the boundary value map $g\mapsto \gamma u$. We recall the following regularity result, 
which is readily obtained from e.g.~\cite[Chapter~5]{Taylor}. 

\begin{lemma} \label{lem:robin}
Let $\Om$ be a sufficiently regular bounded domain with boundary $\Ga$. For every $g \in L^2(\Ga)$, 
the Robin boundary value problem \eqref{robin-aux} has a weak solution $u\in H^{3/2}(\Om)$, and its trace is bounded by
\begin{equation}
	\label{eq:robin - solution bounds}
	\| \tr u \|_{H^1(\Ga)} \le C \|g\|_{L^2(\Ga)},
\end{equation}
where $C$ depends on bounds of finitely many derivatives of a parametrization of the boundary surface $\Ga$.
\end{lemma}

\begin{proof} The Lax--Milgram lemma provides a unique weak solution $u\in H^1(\Om)$ for every $g\in H^{-1/2}(\Ga)$, with
$$
\| \tr u \|_{H^{1/2}(\Ga)} \le c_\gamma \| u \|_{H^1(\Ga)} \le C \|g\|_{H^{-1/2}(\Ga)}
$$ 
by the trace theorem. Furthermore, by \citep{Taylor}, p.~410 and Section 5.7, we have $u\in H^2(\Om)$ for every $g\in H^{1/2}(\Ga)$, and
\begin{equation}\label{u-H32}
	\| \tr u \|_{H^{3/2}(\Ga)} \le c_\gamma \| u \|_{H^2(\Ga)} \le C \|g\|_{H^{1/2}(\Ga)}.
\end{equation}
The stated result then follows by interpolation of Sobolev spaces.
\end{proof}

For the finite element approximation $u_h\in \calV_h$, defined by
\begin{equation}\label{robin-h-aux}
\int_{\Om_h} \nb u_h \cdot \nb \varphi_h +  \int_{\Gah} \trh u_h \, \trh \varphi_h =  \int_{\Gah} g_h \, \trh \varphi_h
\quad\ \text{ for all $\varphi_h\in  \calV_h$,}
\end{equation}
we have an analogous bound to Lemma~\ref{lem:robin}.

\begin{proposition} \label{prop:robin-h} In the situation of Lemma~\ref{lem:robin}, the finite element approximation $u_h$ obtained from a  quasi-uniform and shape-regular family of triangulations satisfies
$$
\| \trh u_h \|_{H^1(\Ga_h)} \le C \|g_h\|_{L^2(\Ga_h)},
$$
where $C$ depends on bounds of finitely many derivatives of a parametrization of the boundary surface $\Ga$ and on the quasi-uniformity and shape regularity bounds, but is independent of the mesh size $h$ and of $g_h\in L^2(\Ga_h)$.
\end{proposition}

\begin{proof} We denote the left-hand sides of \eqref{robin-weak-aux} and \eqref{robin-h-aux}  by $a(u,\varphi)$ and $a_h(u_h,\varphi_h)$, respectively, and the right-hand sides by $m(g,\tr\varphi)$ and $m_h(g_h,\trh\varphi_h)$, respectively.
Note that $a(\cdot,\cdot)$ is an $H^1(\Om)$-elliptic bilinear form, and $m(\cdot,\cdot)$ is the $L^2(\Ga)$ inner product.

We consider the Ritz map $\wt R_h \colon H^1(\Om)\to \calV_h\subset H^1(\Om_h)$ that defines $\wt R_h w\in\calV_h$ through
$$
a_h(\wt R_h w, \varphi_h)=a(w,\varphi_h^\ell) \quad\text{ for all }\ \varphi_h\in\calV_h.
$$
We define $R_h \colon H^1(\Om)\to \calV_h^\ell \subset H^1(\Om)$ by the lift $R_h w = (\wt R_h w)^\ell$.
It is known from \cite[Lemma~3.8]{ElliottRanner_unified} that 
\begin{align*}
	\|  R_h w \|_{H^1(\Om)} &\le C\, \| w \|_{H^1(\Om)} &\text{for all }\ w \in H^1(\Om),
	\\
	\| R_h w -w \|_{H^1(\Om)} &\le Ch\, \| w \|_{H^2(\Om)} &\text{for all }\ w \in H^2(\Om),
\end{align*}
which further implies, by interpolation of Sobolev spaces,
\begin{equation}
	\label{eq:interpolated Ritz error - Om}
	\| R_h w -w \|_{H^1(\Om)} \le Ch^{1/2}\, \| w \|_{H^{3/2}(\Om)} \quad\text{ for all }\ w \in H^{3/2}(\Om) .
\end{equation}
In particular this holds true for the solution $u\in H^{3/2}(\Om)$ of \eqref{robin-aux} with the lifted finite element function $g=g_h^\ell$.

We estimate
\begin{equation}\label{uh-norm-split}
	\| \trh u_h \|_{H^1(\Ga_h)} \le C \, \| \trh (u_h - \wt R_h u )\|_{H^{1}(\Ga_h)} + \| \trh \wt R_h u \|_{H^{1}(\Ga_h)}.
\end{equation}

(i) Consider the first term on the right-hand side. By an inverse inequality and the trace inequality (via a pair of norm equivalences) we obtain
\begin{align*}
	\| \trh (u_h - \wt R_h u )\|_{H^{1}(\Ga_h)} &\le Ch^{-1/2} \| \trh (u_h - \wt R_h u )\|_{H^{1/2}(\Ga_h)}
	\le Ch^{-1/2} \| u_h - \wt R_h u \|_{H^{1}(\Om_h)} \,.
\end{align*}
To bound this term, we note that by the equations \eqref{robin-weak-aux} and \eqref{robin-h-aux} for $u$ and $u_h$ and the definition of the Ritz map,
$$
a_h(u_h - \wt R_h u, \varphi_h) = m_h(g_h,\trh\varphi_h) - m(g_h^\ell,(\trh\varphi)^\ell).
$$
It is known from \cite[Lemma~5.5]{DziukElliott_L2}, \cite[Lemma~5.6]{highorderESFEM} that the right-hand side is bounded by 
\begin{align*}
	| m_h(g_h,\trh\varphi_h) - m(g_h^\ell,(\trh\varphi)^\ell) | 
	\le &\ C h^2 \| g_h \|_{L^2(\Ga_h)}\, \|\trh \varphi_h \|_{L^2(\Ga_h)} 
	\leq C h^2 \| g_h \|_{L^2(\Ga_h)}\, \|\varphi_h \|_{H^1(\Om_h)}.
\end{align*}
(For isoparametric finite elements of polynomial degree $k$ the above estimate holds with $h^{k+1}$ instead of $h^2$.)

By the $h$-uniform equivalence of the norm induced by $a_h$ and the $H^1(\Om_h)$-norm on $\calV_h$, this yields
$$
\| u_h - \wt R_h u \|_{H^{1}(\Om_h)} \le C h^2 \| g_h \|_{L^2(\Ga_h)} 
$$
and hence
$$
\| \trh (u_h - \wt R_h u )\|_{H^{1}(\Ga_h)} \le C h^{3/2} \| g_h \|_{L^2(\Ga_h)}.
$$

(ii) Now we turn to the second term on the right-hand side of \eqref{uh-norm-split}. By the $H^1$ norm equivalence under the lift \cite[equation~(2.16)]{Demlow2009} 
we have, with $R_h u = (\wt R_h u)^\ell$,
$$
\| \trh \wt R_h u \|_{H^{1}(\Ga_h)} \le C \| (\trh \wt R_h u)^\ell \|_{H^{1}(\Ga)} = C \| \tr R_h u \|_{H^{1}(\Ga)} \,.
$$
Let $R_h^\Ga \colon H^1(\Ga)\to \calS_h^\ell \subset H^1(\Ga)$ be the Ritz map for $-\Delta_\Ga u + u =0$ on $\Ga$ (see, e.g., \cite[Definition~6.1]{highorderESFEM}), i.e., $R_h^\Ga w=(\wt R_h^\Ga w)^\ell$, where $w_h = \wt R_h^\Ga w\in \calS_h$ is determined by
$$
\int_{\Ga_h} \Bigl(\nabla_\Ga w_h \cdot \nabla_\Ga \varphi_h + w_h\, \varphi_h\Bigr) =
\int_{\Ga} \Bigl(\nabla_\Ga w \cdot \nabla_\Ga \varphi_h^\ell + w\, \varphi_h^\ell\Bigr)
\quad\ \text{for all $\varphi_h\in \calS_h$.}
$$
By the proof of Theorem~6.3 \citep{highorderESFEM} (working with $k=0$ and the $H^1$ stability of the Ritz map in the estimates of part (b) of the proof therein) that 
\begin{align*}
	\|  R_h^\Ga w \|_{H^1(\Ga)} &\le C\, \| w \|_{H^1(\Ga)} &\text{for all }\ w \in H^1(\Ga),
	\\
	\| R_h^\Ga w -w \|_{L^2(\Ga)} &\le Ch\, \| w \|_{H^1(\Ga)} &\text{for all }\ w \in H^2(\Ga).
\end{align*}
Note that we have formulated the error estimate such that we have the $H^1(\Ga)$ norm on the right-hand side. 
By interpolation of Sobolev spaces this further yields 
\begin{equation}
	\label{eq:interpolated Ritz error - Ga}
	\| R_h^\Ga w -w \|_{H^{1/2}(\Ga)} \le Ch^{1/2} \, \| w \|_{H^{1}(\Ga)} .
\end{equation} 
We estimate
$$
\| \tr R_h u \|_{H^{1}(\Ga)} \le  \| \tr R_h u - R_h^\Ga \tr u \|_{H^{1}(\Ga)} + \| R_h^\Ga \tr u \|_{H^{1}(\Ga)}.
$$
Using first an inverse inequality (via a pair of norm equivalences) and then the trace inequality with \eqref{eq:interpolated Ritz error - Om} and \eqref{eq:interpolated Ritz error - Ga}, \eqref{eq:robin - solution bounds} (where $g = g_h^\ell$), the first term on the right-hand side is bounded by
\begin{align*}
	\| \tr R_h u - R_h^\Ga \tr u \|_{H^{1}(\Ga)} &\le Ch^{-1/2}  \| \tr R_h u - R_h^\Ga \tr u \|_{H^{1/2}(\Ga)}
	\\
	&\le  Ch^{-1/2} \bigl( \| \tr R_h u - \tr u \|_{H^{1/2}(\Ga)} + \|\tr u - R_h^\Ga \tr u \|_{H^{1/2}(\Ga)} \bigr) 
	\\
	&\le   Ch^{-1/2} \bigl( \|  R_h u -  u \|_{H^{1}(\Om)} + \|\tr u - R_h^\Ga \tr u \|_{H^{1/2}(\Ga)} \bigr) 
	\\
	&\le Ch^{-1/2} \bigl( Ch^{1/2} \| u \|_{H^{3/2}(\Om)} + Ch^{1/2} \| \tr u \|_{H^1(\Ga)} \bigr)
	\\
	&\le C \| g \|_{L^2(\Ga)} . 
\end{align*}
The second term is bounded by 
$$
\| R_h^\Ga \tr u \|_{H^{1}(\Ga)} \le C  \, \| \tr u \|_{H^1(\Ga)} \le C \| g \|_{L^2(\Ga)} . 
$$
Combining the bounds of  (i) and (ii) and noting that the $L^2$ norm equivalence under the lift yields
$\| g \|_{L^2(\Ga)} = \| g_h^\ell \|_{L^2(\Ga)} \le C  \| g_h \|_{L^2(\Ga_h)}$, we finally obtain the stated bound for $\| \trh u_h \|_{H^1(\Ga_h)}$.
\end{proof}

\section*{Acknowledgement}
The work of Christian Lubich is supported by Deutsche Forschungsgemeinschaft (DFG, German Research Foundation) -- Project-ID 258734477 -- SFB 1173.

The work of Bal\'azs Kov\'acs is funded by the Heisenberg Programme of the Deutsche Forschungsgemeinschaft (DFG, German Research Foundation) -- Project-ID 446431602.

\bibliographystyle{abbrvnat}
\bibliography{bulksurface_literature}

\end{document}